\documentclass[10pt]{article}

\DeclareMathAlphabet{\mathpzc}{OT1}{pzc}{m}{it}

\usepackage{amsmath}
\usepackage[english]{babel}
\usepackage{amssymb}
\usepackage{amsfonts}
\usepackage{amsthm}
\usepackage{mathrsfs}
\usepackage{ifsym}
\usepackage{fontenc}
\usepackage{textcomp}
\usepackage{tipa}
\usepackage{enumerate}
\usepackage[pdftex]{color, graphicx}
\numberwithin{equation}{section}

\begin{document}
 
\title{{\bf  A proof of Sugawara's conjecture on Hasse-Weber ray class invariants}}    
\author{Patrick Morton}

\maketitle

\begin{abstract}
In this paper a proof is given of Sugawara's conjecture from 1936, that the ray class field of conductor $\mathfrak{m} \neq 1$ over an imaginary quadratic field $K$ is generated over $K$ by a single primitive $\mathfrak{m}$-division value of the $\tau$-function, first defined by Weber and then modified by Hasse in his 1927 paper giving a new foundation of the theory of complex multiplication.
\end{abstract}

\section{Introduction.}
\label{sec:1}

Hasse's well-known paper \cite{h1} contains the first complete proof of Kronecker's {\it Jugendtraum} which only uses modular functions of level one and the Weierstrass $\wp$-function.  An earlier, different proof was given by Takagi and Fueter \cite{fue1,fue2, t}, using modular functions of level four which were also defined in terms of the $\wp$-function.  As part of his proof, Hasse shows that the ray class field $\textsf{K}_\mathfrak{m}$ over an imaginary quadratic field $K$ of a given conductor $\mathfrak{m} \neq 1$ is generated over $K$ by $j(\mathfrak{k})$ and $\tau(\mathfrak{k}^*)$:
$$\textsf{K}_\mathfrak{m} = K(j(\mathfrak{k}), \tau(\mathfrak{k}^*));$$
here $\mathfrak{k}^*$ is a ray class in the ray class group of $K$ modulo $\mathfrak{m}$, $\mathfrak{m}\mathfrak{k}^{-1}$ is the ordinary ideal class containing $\mathfrak{k}^*$, $j(w)$ is Klein's $j$-function, and $\tau(\mathfrak{k}^*) = \tau(\rho,\mathfrak{a})$ is Weber's $\tau$-function, for an element $\rho$ and ideals $\mathfrak{r}, \mathfrak{a}$ of $K$ satisfying
$$\rho \cong \frac{\mathfrak{r} \mathfrak{a}}{\mathfrak{m}}, \ \textrm{integral} \ \mathfrak{r} \in \mathfrak{k}^*, \ \mathfrak{a} \in \mathfrak{k},$$
and $(\mathfrak{r}, \mathfrak{m}) = 1$.  (See below and \cite[p. 138]{h1}.)  The value $\tau(\rho, \mathfrak{a})$ is independent of the ideals $\mathfrak{a} \in \mathfrak{k}$ and $\mathfrak{r} \in \mathfrak{k}^*$ and the choice of the generator $\rho$, and therefore depends only on the ray class $\mathfrak{k}^*$.  \medskip

In \cite[p. 572]{w}, Weber defines the $\tau$-function as
\begin{equation*}
\tau(u, \mathfrak{a}) = \begin{cases}
\frac{g_2 g_3}{G} \wp(u), & g_2 g_3 \neq 0,\\
\frac{\wp(u)^2}{g_2}, & g_3 = 0,\\
\frac{\wp(u)^3}{g_3}, & g_2 = 0;
\end{cases}
\end{equation*}
where $\wp(u) = \wp(u, \mathfrak{a})$, $\wp'(u)^2 = 4\wp(u)^3-g_2\wp(u)-g_3$, and 
$$16G = g_2^3-27g_3^2 = \Delta.$$
On the other hand, Hasse, in \cite[p. 127]{h1}, sets
\begin{equation*}
\tau(u, \mathfrak{a}) = \begin{cases}
-2^7 3^5 \frac{g_2 g_3}{\Delta}\cdot \wp(u), & g_2 g_3 \neq 0,\\
\ \ 2^8 3^4 \frac{g_2^2}{\Delta}\cdot \wp(u)^2, & g_3 = 0,\\
\ \ -2^9 3^6 \frac{g_3}{\Delta}\cdot  \wp(u)^3, & g_2 = 0.
\end{cases}
\end{equation*}
See also \cite[p. 34]{d2}.  Except for the constant factors in front, both of these normalizations agree with the one given in \cite[p.135]{si}.  I use Hasse's normalization in this paper, though I am always assuming $g_2 g_3 \neq 0$, i.e., the corresponding quadratic field is not $\mathbb{Q}(\omega)$ ($\omega = (-1+\sqrt{-3})/2$) or $\mathbb{Q}(i)$.  (We can also eliminate quadratic fields $K$ with class number $1$, since the conjecture to be proved is obviously true for them.)  In this case, note that the factor $\lambda$ multiplying $\wp(u)$ satisfies
$$\lambda^6 = \frac{12^6j(\mathfrak{a})^2(j(\mathfrak{a})-1728)^3}{\Delta}, \ \ \lambda = -2^7 3^5 \frac{g_2 g_3}{\Delta}.$$

As has been pointed out in \cite{ksy, jks1}, Hasse asked Hecke whether the ray class field $\textsf{K}_\mathfrak{m}$ of conductor $\mathfrak{m} \neq 1$ could be generated over $K$ by $\tau(\mathfrak{k}^*)$ alone, and mentioned this question also in \cite[p. 87]{h2}.  See the discussion in \cite[pp. 88-91]{flr}, especially on p. 91.  This was conjectured by Sugawara to be the case in his papers \cite{su1, su2}, which give a partial answer to this question.  An equivalent statement was also conjectured by Hasse in \cite[p. 85]{h2}.  Hasse's question and Sugawara's papers are referred to in \cite[p. 60]{d3}.  Also see \cite[p. 132]{r}. \medskip

Sugawara \cite{su1, su2} showed that the answer to Hasse's question is {\it yes}, if one of three conditions holds for the modulus $\mathfrak{m}$:
\begin{align}
\label{eqn:1} & \ \ 4 \mid \mathfrak{m}, \ \ \mathfrak{m} \neq 4.\\
\label{eqn:2} & \ \ \textit{There is a prime divisor}  \ \mathfrak{p} \ \textit{of} \ 2 \ \textit{in} \ K \ \textit{with} \ \mathfrak{p}^2 \mid \mathfrak{m} \ \textit{and} \ \varphi(\mathfrak{m}) \ge 6.\\
\label{eqn:3} & \ \ \Psi(\mathfrak{m}) = N(\mathfrak{m}) \prod_{\mathfrak{p} \mid \mathfrak{m}}{\left(1-\frac{2}{N(\mathfrak{p})}\right)} \ge 5.
\end{align}

In this paper I will complete the proof of Sugawara's conjecture, starting with his conditions ({\ref{eqn:1})-(\ref{eqn:3}).  Progress on this conjecture has also been made by Jung, Koo, Shin and Yoon, who prove in several papers \cite{ksy, jks1} that Sugawara's conjecture is true for ideals $\mathfrak{m} = (n)$, where $n$ is a natural number.  In \cite{ksy, jks1} the methods are analytic, whereas the arguments I give here are algebraic and arithmetic, though they do rely on some results for modular functions in several previous papers.  Also see \cite{jks2}, where the coordinates of torsion points on specially constructed elliptic curves are shown to generate ray class fields.  \medskip

The proof I give here proceeds by cases.  Sugawara's conditions show that his conjecture is true for all but finitely many ideals $\mathfrak{m}$, for a fixed imaginary quadratic number field $K$.  On the other hand, this leaves open the question for several infinite families of pairs $(K,\mathfrak{m})$.  We note the following concerning the possible remaining pairs. \medskip

Condition (\ref{eqn:3}) shows that the conjecture is true for any prime ideal $\mathfrak{m} = \mathfrak{p}$ whose norm is at least $7$, and also for any $\mathfrak{m}$ divisible by such a prime ideal, if $\mathfrak{m}$ is not divisible by a first degree prime divisor of $2$.  The same condition (with the same restriction on prime divisors of $2$) applies if $\mathfrak{m}$ is divisible by any inert prime other than $2$ in $K/\mathbb{Q}$.  If a prime ideal $\mathfrak{p}$ of degree $1$ does divide $(2,\mathfrak{m})$, then $\mathfrak{p}^2 \mid \mathfrak{m}$ in order for $\mathfrak{m}$ to be a conductor.  In this case, if $\mathfrak{m} = \mathfrak{p}^2 \mathfrak{a}$, where $(\mathfrak{a}, 2) =1$ and $\varphi(\mathfrak{a}) \ge 3$, then by \eqref{eqn:2}, Sugawara's conjecture is true.  The same holds if $\mathfrak{m} = \mathfrak{p}^e \mathfrak{a}$ with $e = 3$ and $\mathfrak{a} \neq 1$ or $e \ge 4$.  Thus, we restrict ourselves to prime divisors of $2, 3$ and $5$ in building the ideal $\mathfrak{m}$, and we have the following possibilities, where $\wp_l$ denotes a prime divisor of degree $1$ of $l$ and $(2)$ is inert:
\begin{align}
\notag \mathfrak{m} &= (2), \wp_2^2 \cong 2, \wp_2^3, \wp_2^2\wp_2'^2 \cong 4;\\
\label{eqn:4} \mathfrak{m} &= \wp_3 \wp_3', \wp_3^2, \wp_3^2 \wp_3', (2) \wp_3\wp_3', \wp_2^2 \wp_3;\\
\notag \mathfrak{m} &= \wp_5, \wp_3 \wp_5, \wp_3 \wp_3' \wp_5.
\end{align}
In the cases with an unramified first degree prime ideal, the discriminant $d_K$ of $K$ satisfies $d_K \equiv 1$ mod $8$ (for $\wp_2$); $d_K \equiv 1$ mod $3$ (for $\wp_3$); and $d_K \equiv 1,4$ mod $5$ (for $\wp_5$).  Five of the above cases are covered by the results of \cite{ksy, jks1} and seven are not; namely, the cases
$$\mathfrak{m} = \wp_2^3, \wp_3^2, \wp_3^2 \wp_3', \wp_2^2 \wp_3, \wp_5, \wp_3 \wp_5, \wp_3 \wp_3' \wp_5$$
require a new argument.  In particular, the sixth and seventh cases in this list are by far the most difficult.  (Note that the ideal $\mathfrak{m} = (2) \wp_3$ is not a possibility because it is not a conductor, since $\textsf{K}_\mathfrak{m} \supseteq \textsf{K}_{(2)}$ and $\varphi(\mathfrak{m})/2 = \varphi((2)) = [\textsf{K}_{(2)}:\Sigma]$.)  \medskip

Converting to algebraic notation, the $\tau$-invariants for $\mathfrak{m}$ and $K \neq \mathbb{Q}(\sqrt{-3})$ or $\mathbb{Q}(\sqrt{-4})$ can be written in the form 
$$\tau(\mathfrak{k}^*) = -2^7 3^5 \frac{g_2 g_3}{\Delta} X(P) = h(P),$$
where $P$ is a primitive $\mathfrak{m}$-division point with $X$-coordinate $X(P)$ on an elliptic curve $E$ in Weierstrass normal form with complex multiplication by $R_K$ and invariants $g_2, g_3, \Delta$.  It is clear that the expressions $h(P)$ are independent of the particular model of the elliptic curve $E$ which is used to compute the torsion points $P$.  (See \cite[p. 135]{si}.) \medskip

The general method presented here is to take an elliptic curve $E$ of a special form which is determined in each case by the ideal $\mathfrak{m}$, and give explicit formulas for enough points in $E[\mathfrak{m}]$  to be able to determine the $\tau$-invariants for $\mathfrak{m}$.  We take $E$ to be either: the Legendre normal form $E_2$, the Deuring normal form $E_3$, or the Tate normal form $E_n$ for a point of order $n$, where $n \in \{4, 5, 9, 12\}$.  In each case, the curve $E$ is defined by certain parameters lying either in the Hilbert class field of $K$ or in an abelian extension of small conductor over $K$.  In many of the cases, the prime ideal factorization of one of these parameters plays a decisive role in our arguments.  Most of these factorizations have been determined in previous papers \cite{lm, m1, m2, m3, am1}.  In the last two cases in (\ref{eqn:4}), the ray class invariants for $\mathfrak{m}$ are computed by finding the points of order $3$ on $E_5$, which leads to some interesting arithmetic relationships.  See Sections \ref{sec:9} and \ref{sec:10}.  \medskip

In each case we must show that the $\tau$-invariants $\tau(\mathfrak{k}^*)$ for ray classes $\mathfrak{k}^*$ corresponding to the ideal class $\mathfrak{k}$ and $j$-invariant $j(\mathfrak{k})$ are distinct, as $\mathfrak{k}$ runs over all ideal classes in the ring of integers $R_K$ of $K$.  Hasse shows in \cite[pp. 83-85]{h2} that the ray class polynomial
$$T_\mathfrak{m}(t,j(\mathfrak{k})) = \prod_{\mathfrak{k}^* \ \textrm{for} \ \mathfrak{k}}{(t-\tau(\mathfrak{k}^*))} \in K[t,j(\mathfrak{k})],$$
whose roots are the $\tau$-invariants corresponding to a given ideal class, has coefficients in the Hilbert class field $\Sigma = \textsf{K}_1$ and is irreducible over this field.  Thus, the set of $\tau$-invariants for the ideal class $\mathfrak{k}$ is a complete set of conjugates over $\Sigma$, and different ray class polynomials are conjugate by automorphisms of $\Sigma/K$.  We will show that the invariants $\tau(\mathfrak{k}^*)$ for a given class $\mathfrak{k}$ determine $j(\mathfrak{k})$.  To do this we will often assume that the automorphism $\psi \in \textrm{Gal}(\textsf{K}_\mathfrak{m}/K)$ takes $j(\mathfrak{k})$ to $j(\mathfrak{k})^\psi = j(\mathfrak{k}')$ and leaves invariant the set $\{\tau(\mathfrak{k}^*)| \ \mathfrak{k}^*\ \textrm{for} \ \mathfrak{k}\}$ taken as a whole.  Then our task is to show that $\psi$ restricted to the Hilbert class field $\Sigma$ is $1$, implying that $j(\mathfrak{k})= j(\mathfrak{k}')$ and $\mathfrak{k} = \mathfrak{k}'$.  Since the $\tau$ invariants corresponding to a given class $\mathfrak{k}$ are distinct \cite[Satz 20]{h1}, this proves the assertion.  In Hasse's papers, he takes the ideal class corresponding to $\mathfrak{k}^*$ to be the class containing $\frac{\mathfrak{m}}{\mathfrak{k}^*}$, i.e., the ideal class containing the ideals $\frac{\mathfrak{m}}{\mathfrak{r}}, \mathfrak{r} \in \mathfrak{k}^*$.  It is clear, however, that we can take any correspondence between ray classes and ideal classes, as long as all ray classes {\it belonging} to a given ideal class {\it correspond} to a single class $\mathfrak{k}$ and this correspondence is preserved by automorphisms over $K$. \medskip

Combining the results proved here with Sugawara's results \cite{su1,su2} gives the following. \bigskip

\noindent {\bf Main Theorem.} {\it If $\textsf{K}_\mathfrak{m}$ is the ray class field with conductor $\mathfrak{m} \neq 1$ over the imaginary quadratic field $K$, then $\textsf{K}_\mathfrak{m} = K(\tau(\mathfrak{k}^*))$ is generated over $K$ by a single $\tau$-invariant for the ideal $\mathfrak{m}$.}  \bigskip

\noindent {\bf Corollary.} {\it If the ideal $\mathfrak{m}$ is a conductor of the corresponding ray class group, the full ray class equation
$$T_\mathfrak{m}(t) = \prod_{\mathfrak{k}}{T_\mathfrak{m}(t, j(\mathfrak{k}))}$$
taken over the ideal classes $\mathfrak{k}$ in $K$, is irreducible over $K$.} \bigskip

This corollary was conjectured by Hasse in \cite[p. 85]{h2}.  \medskip

In the cases when $d_K \equiv 1$ modulo $8$ or $3$, I also show that the $\tau$-invariant for any of the respective ideals $\wp_2, \wp_2 \wp_2', \wp_3$ generates the Hilbert class field $\Sigma = \textsf{K}_{1}$ of $K$.  Thus, for these two families of quadratic fields, the $j$-invariant is not needed to generate the abelian extensions of $K$.  See Theorems \ref{thm:2} and \ref{thm:7} in Sections \ref{sec:2} and \ref{sec:5}.  \medskip

The proof of the special cases in \eqref{eqn:4} involves substantial calculations with the Maple software throughout, but especially in Sections 6, 9, 10, and 11.  The readers wishing to work out some of the calculations for themselves can find some of the Maple commands that have been used in a supplementary document available from the author.

\section{The case $\frak{m} = (2)$.}
\label{sec:2}

In what follows $K = \mathbb{Q}(\sqrt{-d})$ is an imaginary quadratic field with discriminant $d_K$, $\Sigma$ is its Hilbert class field, and $\textsf{K}_\mathfrak{m}$ is the ray class field over $K$ of conductor $\mathfrak{m}$ (so that $\Sigma = \textsf{K}_1$).  We assume that $d_K \neq -3, -4$, so that the $\tau$-function is always given by
$$\tau(u,\mathfrak{a}) = -2^7 3^5 \frac{g_2 g_3}{\Delta} \wp(u).$$

We begin with the easiest case (from a computational point of view).  We consider the Legendre normal form
$$E_2: Y^2 = X(X-1)(X-a),$$
which we assume to have complex multiplication by the quadratic field $K$, whose discriminant $d_K$ satisfies $d_K \equiv 5$ (mod $8$).  Its associated Weierstrass normal form is
$$E': Y_1^2 =4 X_1^3-g_2X_1-g_3,$$
where $X_1 = X - \frac{a+1}{3}$,
$$g_2= \frac{4}{3}(a^2-a+1), \ \ g_3 = \frac{4}{27}(a+1)(a-2)(2a-1),$$
and
$$\Delta = 16a^2(a - 1)^2.$$
Thus
$$j(E_2) = j(a) = \frac{2^8(a^2-a+1)^3}{a^2(a-1)^2},$$
and
$$j(a)-1728 = \frac{64(a + 1)^2(a - 2)^2(2a - 1)^2}{a^2(a - 1)^2}.$$
We compute the ray class invariants for $\mathfrak{m} = (2)$ and a given ideal class $\mathfrak{k}$ for which $j(a) = j(\mathfrak{k})$ to be
\begin{align*}
\tau_0 & = \frac{-2^73^5g_2g_3}{\Delta}\left(0-\frac{a+1}{3}\right) = \frac{128(a^2 - a + 1)(a + 1)^2(a - 2)(2a - 1)}{a^2(a - 1)^2};\\
\tau_1 & = \frac{-2^73^5g_2g_3}{\Delta}\left(1-\frac{a+1}{3}\right) = \frac{128(a^2 - a + 1)(a + 1)(a - 2)^2(2a - 1)}{a^2(a - 1)^2};\\
\tau_2 & = \frac{-2^73^5g_2g_3}{\Delta}\left(a-\frac{a+1}{3}\right) = \frac{-128(a^2 - a + 1)(a + 1)(a - 2)(2a - 1)^2}{a^2(a - 1)^2}.
\end{align*}
We note that
\begin{equation}
\tau_0+\tau_1+\tau_2 = 0
\label{eqn:2.1}
\end{equation}
and the $\tau_i$ are roots of the cubic polynomial
$$F(X,j(\mathfrak{k})) = X^3 -3j(\mathfrak{k})(j(\mathfrak{k}) - 1728)X + 2j(\mathfrak{k})(j(\mathfrak{k}) - 1728)^2.$$
Hasse proved in \cite{h2} that this polynomial is irreducible over the Hilbert class field $\textsf{K}_1$ of $K$, when $d_K \equiv 5$ (mod $8$).  Furthermore, the formulas for the coefficients imply that
\begin{equation}
\frac{1}{\tau_0} + \frac{1}{\tau_1} + \frac{1}{\tau_2} = \frac{3}{2(j(\mathfrak{k})-1728)}.
\label{eqn:2.2}
\end{equation}
This immediately implies that if $\mathfrak{k}, \mathfrak{k}'$ are two ideal classes in $K$ for which the sets $\{\tau_0,\tau_1,\tau_2\}$ and $\{\tau_0',\tau_1',\tau_2'\}$ coincide, then $j(\mathfrak{k}) = j(\mathfrak{k}')$ and therefore $\mathfrak{k} = \mathfrak{k}'$.  This implies, by Hasse's argument \cite[p. 85]{h2}, that $K(\tau_i) = \textsf{K}_\mathfrak{m}$, so that Sugawara's conjecture is true for $\mathfrak{m} = (2)$ and $d_K \equiv 5$ (mod $8$). \medskip

We also note the relation
$$\frac{\tau_2-\tau_0}{\tau_1-\tau_0} = a.$$
It follows that $a \in \textsf{K}_{(2)}$ and the above formulas show that $K(a) = \textsf{K}_{(2)}$.  \bigskip

\newtheorem{thm}{Theorem}

\begin{thm}
If the Legendre normal form $E_2(a)$ has complex multiplication by the ring of integers in the quadratic field $K = \mathbb{Q}(\sqrt{d_K})$, where $d_K \equiv 5$ (mod $8$), then the parameter $a$ generates the ray class field $\textsf{K}_{(2)}$ of conductor $\mathfrak{m} = (2)$ over $K$.  Also, $\textsf{K}_{(2)}$ is generated over $K$ by a single $\tau$-invariant for the ideal $(2)$.
\label{thm:1}
\end{thm}

Now the same formulas hold if $4 \mid d_K$ and $d_K < -4$, in which case $(2) = \mathfrak{p}^2$ and $\varphi(\mathfrak{p}^2) = 2$.  In this case there are two ray class invariants $\tau(\mathfrak{k}^*)$ for each ideal class $\mathfrak{k}$.  This implies that the polynomial $F(x,j(\mathfrak{k}))$ is reducible over $K(j(\mathfrak{k}))$ and factors into a linear times an irreducible quadratic.  If the two ray class invariants (mod $(2)$) for two ideal classes agree, i.e. $\{\tau_i,\tau_k\} = \{\tau_i', \tau_k'\}$, corresponding to $\mathfrak{k}$ and $\mathfrak{k}'$, then the third ray class invariants (corresponding to $\mathfrak{m} = \mathfrak{p}$) must also agree, by (2.1), and then (2.2) implies that $j(\mathfrak{k}) = j(\mathfrak{k}')$.  Thus, Sugawara's conjecture also holds when $\mathfrak{m} = \mathfrak{p}^2 = (2)$.   This establishes the conjecture for the first two possibilities in the first line of Section \ref{sec:1}, \eqref{eqn:4}.  \medskip

Lastly, suppose that $(2) = \mathfrak{p}_1 \mathfrak{p}_2$ in $K$, so that $d_K \equiv 1$ (mod $8$).  Since $\varphi(\mathfrak{p}_1) = \varphi(\mathfrak{p}_2) = \varphi(\mathfrak{p}_1) \varphi(\mathfrak{p}_2) = 1$, there is one ray class invariant each for $\mathfrak{m} = \mathfrak{p}_1, \mathfrak{p}_2$, and $\mathfrak{p}_1 \mathfrak{p}_2 = (2)$.  Each of these invariants lies in $\textsf{K}_1$.  Together, they generate $\textsf{K}_1$, by (\ref{eqn:2.2}). \medskip

\newtheorem{conj}{Conjecture}

I conjecture the following.

\begin{conj}
If $d_K \equiv 1$ (mod $8$) and $(2) = \mathfrak{p}_1 \mathfrak{p}_2$, the invariants for $\mathfrak{m} = \mathfrak{p}_1, \mathfrak{p}_2$ have degree $2h(d_K)$ over $\mathbb{Q}$, so that each generates $\textsf{K}_1$ over $\mathbb{Q}$; while the invariant for $\mathfrak{p}_1 \mathfrak{p}_2$ has degree $h(d_K)$ over $\mathbb{Q}$ and generates $\textsf{K}_1$ over $K$.
\label{conj:1}
\end{conj}

We can approach this conjecture by showing that the discriminant of $F(X,j(\mathfrak{k}))$ is negative, when $\mathfrak{k} =\mathfrak{o}$ is the principal class.  This discriminant is
$$\textrm{disc}(F(X,j(\mathfrak{k}))) = 2^8 3^6 j^2(j - 1728)^3, \ \ j = j(\mathfrak{k}).$$
If this discriminant is negative, then two of the invariants $\tau_i(\mathfrak{k})$ are complex and one is real.  The complex invariants would have to be the invariants for $\mathfrak{m} = \mathfrak{p}_1, \mathfrak{p}_2$, which are interchanged by complex conjugation.  Hence, the real invariant is the invariant for $\mathfrak{m} = \mathfrak{p}_1 \mathfrak{p}_2$.\medskip

\newtheorem{lem}{Lemma}

\begin{lem}
If $d_K \equiv 1$ (mod $8$), then the $j$-invariant $ j(\mathfrak{o})$ of the principal class $\mathfrak{o}$ in $R_K$ satisfies $ j(\mathfrak{o}) < 1728$.
\label{lem:1}
\end{lem}

\begin{proof}
This will follow from the formula
$$j(\mathfrak{o}) - 1728 = \frac{64(a + 1)^2(a - 2)^2(2a - 1)^2}{a^2(a - 1)^2} = R(a)^2,$$
where
$$R(a) = \frac{8(a + 1)(a - 2)(2a - 1)}{a(a - 1)}.$$
From \cite[Eq. (7.1), p. 1979, (i)]{lm}, and taking $\mathfrak{p}_1$ to be the ideal $\wp_2$ in that paper, we have
$$j(\mathfrak{o}) -1728 = j(\mathfrak{p}_1)^\tau -1728 = R(a)^{2\tau} = R\left(-\frac{\xi^4}{\pi^4}\right)^{2\tau},$$
where $\tau = \left(\frac{\Sigma/K}{\mathfrak{p}_1}\right), a = -\frac{\xi^4}{\pi^4}$, and $\pi^4+\xi^4 = 1$, $\pi \cong \mathfrak{p}_1, \xi \cong \mathfrak{p}_2$.  (This uses that $\frac{\xi}{\pi} = \frac{\zeta_8^j \alpha}{2}$, with $\zeta_8 = e^{2\pi i/8}$, from \cite[p. 1978]{lm}.)  For this value of $a$, we have
\begin{align*}
R\left(-\frac{\xi^4}{\pi^4}\right)& = \frac{8(\pi^4 - \xi^4)(\pi^4 + 2\xi^4)(2\pi^4 + \xi^4)}{\pi^4 \xi^4 (\pi^4 + \xi^4)}\\
& = \frac{8(\pi^4 - \xi^4)(1+\xi^4)(1+\pi^4)}{\pi^4 \xi^4}.
\end{align*}
Also from \cite[pp. 2004, 1984]{lm}, the automorphism $\sigma: \xi \rightarrow \frac{\xi+1}{\xi-1} = \pi^{\tau^2}$ is complex conjugation, for which $\pi^\sigma = \xi^{\sigma \tau^{-2} \sigma} = \xi^{\tau^2}$.  Hence
\begin{align*}
R\left(-\frac{\xi^4}{\pi^4}\right)^{\tau \sigma} & = R\left(-\frac{\xi^4}{\pi^4}\right)^{\sigma \tau^{-1}}\\
& = -\left(\frac{8(\pi^4 - \xi^4)(1+\xi^4)(1+\pi^4)}{\pi^4 \xi^4}\right)^{\tau^2 \tau^{-1}}\\
& = -R\left(-\frac{\xi^4}{\pi^4}\right)^{\tau}.
\end{align*}
This implies that $R(a)^\tau$ is pure imaginary, hence $j(\mathfrak{o}) -1728 < 0$, proving the lemma.
\end{proof}

\begin{thm}
If $d_K \equiv 1$ (mod $8$) and $(2) = \mathfrak{p}_1 \mathfrak{p}_2$ in $K = \mathbb{Q}(\sqrt{-d})$, the $\tau$-invariant $\tau(\mathfrak{k}^*)$ for any of the ideals $\mathfrak{p}_1, \mathfrak{p}_2$ or $\mathfrak{p}_1 \mathfrak{p}_2$ and any ideal class $\mathfrak{k}$ generates the Hilbert class field $\textsf{K}_1 = \Sigma$ over $K$.
\label{thm:2}
\end{thm}

\begin{proof}
We have to show that each of the invariants $\tau_i$ generates $\Sigma$ over $K$.  By \cite[p. 1979, (ii)]{lm}, we can take $a = \pi^4$, where $(\pi) = \mathfrak{p}_1$, as in the proof of Lemma \ref{lem:1}.  \medskip

Let $\alpha \in \textrm{Gal}(\Sigma/K)$ be such that $j(\mathfrak{k}') = j(\mathfrak{k})^\alpha$ and $\{\tau_0,\tau_1, \tau_2\}^\alpha \cap \{\tau_0,\tau_1, \tau_2\}$ is non-empty.  Representing the action of $\alpha$ by primes, first assume that $\tau_0 = \tau_0^\alpha = \tau_0'$ for two ideal classes $\mathfrak{k}$ and $\mathfrak{k}'$.  
(Note that we don't know which invariants $\tau_i, \tau_i'$ correspond to the respective ideals $\mathfrak{p}_1, \mathfrak{p}_2, \mathfrak{p}_1 \mathfrak{p}_2$; but we do know that $\tau_i, \tau_i'$ correspond to the same ideal, since $\alpha$ fixes $\mathfrak{p}_1$ and $\mathfrak{p}_2$.)  Setting $x = a = \pi^4$ and $y = a' = \pi'^4 \cong \pi^4$ yields that
\begin{align*}
&\frac{\tau_0 - \tau_0'}{2^7} =\frac{(-y + x)(xy - 1)}{x^2(x - 1)^2 y^2(y - 1)^2}[2x^4 y^3 + 2x^3y^4 - 4x^4y^2 - 7x^3y^3 - 4x^2y^4\\
& \ \ + 2x^4y + 10x^3y^2 + 10x^2y^3 + 2xy^4 - 7x^3y - 8x^2y^2 - 7xy^3\\
& \ \  + 2x^3 + 10x^2y + 10xy^2 + 2y^3 - 4x^2 - 7xy - 4y^2 + 2x + 2y].
\end{align*}
The terms in square brackets whose total degrees in $x$ and $y$ are at least $3$ are divisible by $\pi^{12}$.  If the expression in square brackets is zero, then
$$- 4x^2 - 7xy - 4y^2 + 2x + 2y \equiv 0 \ (\textrm{mod} \ \pi^{12}).$$
Now, the powers of $\pi$ dividing the terms in this expression are, respectively, $10, 8, 10, 5, 5$.
It follows that $2x+2y \equiv 0$ (mod $\pi^8$) and therefore 
$$a+a' = \pi^4 + \pi'^4 = \pi^4+\pi^{4\alpha} \equiv 0 \ (\textrm{mod} \ \mathfrak{p}_1^7).$$
Hence, $a-a' = a+a'-2a' \equiv 0$ (mod $\pi^5$).  But then
\begin{align*}
&j(\mathfrak{k}) - j(\mathfrak{k}') = \frac{2^8(a^2-a+1)^3}{a^2(a-1)^2}-\frac{2^8(a'^2-a'+1)^3}{a'^2(a'-1)^2}\\
& = \frac{2^8(aa' - a + 1)(aa' - a' + 1)(a+a' - 1)(aa' - 1)(a-a')(aa' - a - a')}{a^2(a - 1)^2 a'^2(a' - 1)^2}\\
& = \frac{2^8}{a^2(a-1)^2} \frac{(a-a')(aa' - a - a')}{a'^2(a'-1)^2} \textsf{A},
\end{align*}
where
$$\textsf{A} = [(aa' - a + 1)(aa' - a' + 1)(a+a' - 1)(aa' - 1)].$$
Now $\frac{256}{a^2(a-1)^2} = \frac{256}{\pi^8 \xi^8}$ is a unit.  Also,
$\frac{(a-a')(aa' - a - a')}{a'^2(a'-1)^2}$ is divisible by $\pi^5 \pi^7/\pi^8 = \pi^4$, so 
$$j(\mathfrak{k}) \equiv j(\mathfrak{k}') = j(\mathfrak{k})^\alpha \ (\textrm{mod} \ \mathfrak{p}_1^4).$$
But the discriminant of the class equation $H_{d_K}(X)$ is not divisible by $2$, since $d_K \equiv 1 $ (mod $8$), which implies that $\alpha = 1$ and $j(\mathfrak{k}') = j(\mathfrak{k})^\alpha = j(\mathfrak{k})$. \medskip

This argument applies under the assumption that the expression in square brackets is $0$.  Otherwise, $y = x$ or $y = 1/x$, which implies $a = a'$ (since $a \neq 1/a'$), and we conclude again that $j(\mathfrak{k}') = j(\mathfrak{k})$. \medskip

Next suppose that $\tau_1 = \tau_1'$.  Setting $x = a = \pi^4$ and $y = a' = \pi'^4 \cong \pi^4$ yields that
\begin{align*}
\frac{\tau_1 - \tau_1'}{2^7} &= \frac{1}{x^2(x - 1)^2y^2(y - 1)^2}(x-y)(xy - x -y)[2x^4y^3 + 2x^3y^4 - 2x^4y^2\\
& \ \ -9 x^3y^3 -2x^2y^4 +9 x^3y^2 + 9x^2y^3 - 2xy^4 -13x^2y^2 + 8x^2y + 8xy^2 \\
& \ \  - 4x^2 -16xy - 4y^2 +8x + 8y -4].
\end{align*}
The terms in square brackets whose total degrees in $x$ and $y$ are at least $2$ are divisible by $\pi^{8}$.  If the expression in square brackets is zero, then
$$8x + 8y - 4 \equiv 0 \ (\textrm{mod} \ \pi^{8}).$$
But the powers of $\pi$ dividing the terms in this expression are, respectively, $7, 7$ and $2$.  Thus, the left side is exactly divisible by $\pi^2$ and this congruence is impossible.  Hence we must have $(x - y)(xy - x - y) = 0$, so that $y = x$ or $y = \frac{x}{x-1}$.  Both expressions are anharmonic transformations in $a$ which fix the $j$-invariant.  Hence $j(\mathfrak{k}') = j(\mathfrak{k})$.  \medskip

If we apply similar reasoning to the relation $\tau_2 = \tau_2'$ we find that
\begin{align*}
&\frac{\tau_2 - \tau_2'}{2^7} = \frac{1}{x^2(x - 1)^2y^2(y - 1)^2}(y-x)(y-1+x)[4x^4y^4 - 8x^4y^3 - 8x^3y^4\\
& \ \ + 4x^4y^2 + 16x^3y^3 + 4x^2y^4 - 8x^3y^2 - 8x^2y^3 + 13x^2y^2 - 9x^2y - 9xy^2\\
& \ \   + 2x^2 + 9xy + 2y^2 - 2x - 2y].
\end{align*}
There are two terms in $x$ and $y$ (in square brackets) with minimum valuation; and these terms are $2x+2y$, each with valuation $5$.  The next smallest valuation is $8$, which occurs for the term $9xy$.  If the expression in square brackets is zero, we conclude that
$$2a + 2a' \equiv 0 \ (\textrm{mod} \ \pi^8),$$
giving that
$$a + a' \equiv 0 \ (\textrm{mod} \ \pi^7).$$
Now the same argument as in the case $\tau_0 = \tau_0'$, applied to the difference $j(\mathfrak{k}) - j(\mathfrak{k}')$, shows that
$$j(\mathfrak{k}) \equiv j(\mathfrak{k}') \ (\textrm{mod} \ \mathfrak{p}_1^4)$$
and hence $j(\mathfrak{k}) = j(\mathfrak{k}')$ and $\mathfrak{k} = \mathfrak{k}'$.  Otherwise, $y = x$ and $a = a'$ (since $a'$ cannot equal $1-a$), giving the same conclusion.  \medskip

It follows that the sets $S(\mathfrak{k}) = \{\tau_0, \tau_1, \tau_2\}$ are disjoint for different ideal classes $\mathfrak{k}$.  Since the conjugates over $K$ of a $\tau$-invariant for any of the ideals $\mathfrak{p}_1, \mathfrak{p}_2$, or $\mathfrak{p}_1\mathfrak{p}_2$ are also $\tau$-invariants for the same ideals, we conclude that the conjugates of any $\tau(\mathfrak{k}^*)$ for one of these ideals are distinct.  This proves the theorem.
\end{proof}

\section{The case $\frak{m} = (4)$.}
\label{sec:3}

In this section we assume $\mathfrak{m} = (4) = \mathfrak{p}_1^2 \mathfrak{p}_2^2$ and the discriminant $d_K \equiv 1$ (mod $8$).  We compute the ray class invariants for divisors of $\frak{m} = (4)$ using the Tate normal form $E_4$, on which $(0,0)$ has order $4$:
$$E_4: Y^2+XY+bY = X^3+b X^2,\ \ b = \frac{1}{\alpha^4} = \frac{\beta^4-16}{16\beta^4},$$
where $16\alpha^4+16\beta^4 = \alpha^4 \beta^4$ and $\mathbb{Q}(\alpha) = \mathbb{Q}(\beta) = \mathbb{Q}(\beta^4) = \Sigma = \textsf{K}_1$, as in \cite{lm}.
(These quantities are related to the quantities $\pi, \xi$ which appeared in the proofs of Lemma \ref{lem:1} and Theorem \ref{thm:2} by $\frac{\beta}{2} = \xi$ and $\frac{\beta}{\zeta_8^j \alpha} = \pi$.  See \cite{lm}, p. 1978.)  The arguments in \cite{lm} yield specific values of $\alpha$ and $\beta$ for which $E_4$ has complex multiplication by $R_K$ in this case.
The coefficients of the Weierstrass normal form of $E_4$ are
\begin{align*}
g_2 & = \frac{(\beta^4 + 4\beta^3 + 8\beta^2 - 16\beta + 16)(\beta^4 - 4\beta^3 + 8\beta^2 + 16\beta + 16)}{192\beta^8},\\
g_3 & = -\frac{(\beta^4 + 16)(\beta^2 + 4\beta - 4)(\beta^2 - 4\beta - 4)(\beta^4 + 24\beta^2 + 16)}{13824\beta^{12}};
\end{align*}
with
$$\Delta = \frac{(\beta - 2)^4(\beta + 2)^4(\beta^2 + 4)^4}{4096\beta^{20}}.$$
Thus,
\begin{align*}
j(E_4) = j(\beta) & = \frac{(\beta^4 + 4\beta^3 + 8\beta^2 - 16\beta + 16)^3(\beta^4 - 4\beta^3 + 8\beta^2 + 16\beta + 16)^3}{\beta^4(\beta - 2)^4(\beta + 2)^4(\beta^2 + 4)^4 },\\
j(\beta) - 1728 & = \frac{(\beta^4 + 16)^2(\beta^2 + 4\beta - 4)^2(\beta^2 - 4\beta - 4)^2(\beta^4 + 24\beta^2 + 16)^2}{\beta^4(\beta - 2)^4(\beta + 2)^4(\beta^2 + 4)^4}.
\end{align*}

We will see that here, the (primitive and nonprimitive) ray class invariants for divisors of $\mathfrak{m}$ which are not divisors of $(2)$ are roots of the polynomial
\begin{align*}
F(x,j(\mathfrak{k}))= & X^6 - 15j(j - 1728)X^4 + 40 j(j - 1728)^2X^3 - 45 j^2 (j - 1728)^2X^2 \\
& + 24j^2(j - 1728)^3X - j^2(j - 1728)^3(5j - 55296), \ \ j = j(\mathfrak{k}).
\end{align*}
This implies that the ray class invariants $\tau_i = \tau(\mathfrak{k}_i^*)$ for $\mathfrak{m} = \mathfrak{p}_1^2, \mathfrak{p}_2^2, \mathfrak{p}_1^2 \mathfrak{p}_2, \mathfrak{p}_1 \mathfrak{p}_2^2$ and $\mathfrak{m} = (4)$ satisfy 
\begin{equation}
\sum_{i=0}^5{\frac{1}{\tau_i}} = \frac{24}{5j(\mathfrak{k})-55296}.
\label{eqn:3.1}
\end{equation}
We label the $X$-coordinates of points of order $4$ on $E_4$ in order as
\begin{align*}
X_0 & = 0, \ X_1 = -\frac{(\beta - 2)(\beta^2 + 4)}{8\beta^3}, \ X_2 = -2b = -\frac{(\beta^4-16)}{8\beta^4},\\
X_3 & = X_1(i\beta) = \frac{-(\beta+2i)(\beta^2-4)}{8\beta^3}, \  X_4 = X_1(-\beta) = -\frac{(\beta + 2)(\beta^2 + 4)}{8\beta^3},\\
X_5 & = X_1(-i\beta) = \frac{-(\beta-2i)(\beta^2-4)}{8\beta^3}.
\end{align*}
See \cite[pp. 1971-1974]{lm}.  The corresponding ray class invariants are given by
$$\tau_i = -2^7 3^5 \frac{g_2 g_3}{\Delta}\left(X_i + \frac{4b+1}{12}\right), \ \ b = \frac{\beta^4-16}{16\beta^4}.$$
A calculation on Maple shows that these six quantities are the roots of the polynomial $F(X,j(\mathfrak{k}))$ given above.  In particular,
$$\tau_0 = -2^7 3^5 \frac{g_2 g_3}{\Delta}\frac{4b+1}{12} =  -2^7 3^5 \frac{g_2 g_3}{\Delta} \frac{5\beta^4 - 16}{48\beta^4}.$$
This gives
\begin{align}
\label{eqn:3.2} \frac{\tau_3-\tau_0}{\tau_0} & = -6\frac{\beta(\beta - 2)(\beta + 2)(\beta+2i)}{5\beta^4 - 16},\\
\label{eqn:3.3} \frac{\tau_5-\tau_0}{\tau_0} & = -6\frac{\beta(\beta - 2)(\beta + 2)(\beta-2i)}{5\beta^4 - 16}.
\end{align}
We also note that the map $\sigma: \beta \rightarrow 2\frac{\beta+2}{\beta-2}$ induces the permutation
$$\tau_0^\sigma = \tau_4, \ \ \tau_1^\sigma = \tau_2, \ \ \tau_3^\sigma = \tau_5.$$
By the results of \cite[Prop. 8.2]{lm}, $\sigma$ is an automorphism of $\textsf{K}_1 = \mathbb{Q}(\beta)$ (the same as the automorphism $\sigma$ in Lemma \ref{lem:1}, since $\beta = 2\xi$).  It is also clear that $\tau_0, \tau_1, \tau_2, \tau_4 \in \textsf{K}_1$, while $\tau_3, \tau_5 \in \medskip \textsf{K}_1(i) = \textsf{K}_\mathfrak{m}$, the ray class field of conductor $\mathfrak{m} = (4)$, in the case that $2 = \mathfrak{p}_1 \mathfrak{p}_2$. \medskip

Since
$$\frac{\tau_1-\tau_0}{\tau_2-\tau_0} = \frac{\beta}{\beta+2},$$
we see that $K(\tau_0, \tau_1, \tau_2) = \textsf{K}_1 = \Sigma$. \medskip

\begin{thm}
The invariants $\tau_0, \tau_1, \tau_2, \tau_4$ are the roots of the polynomial
\begin{align*}
G(X, &\tau_3, \tau_5) =  \ 4X^4 + 4(\tau_3 + \tau_5)X^3 -6 (\tau_3^2 +6\tau_3 \tau_5 +\tau_5^2)X^2\\
& + 4(\tau_3 + \tau_5)(\tau_3^2 + 5 \tau_3 \tau_5 + \tau_5^2)X- \tau_3^4 - 6\tau_3^3 \tau_5 - 6\tau_3^2 \tau_5^2 - 6\tau_3 \tau_5^3 - \tau_5^4.
\end{align*}
\label{thm:3}
\end{thm}
\noindent {\it Proof.} Taking the resultant
\begin{align*}
\textrm{Res}_\beta(&(5\beta^4 - 16)(\tau_3-X)+6\beta(\beta - 2)(\beta + 2)(\beta+2i)X,\\
 &(5\beta^4 - 16)(\tau_5-X)+6\beta(\beta - 2)(\beta + 2)(\beta-2i)X)
\end{align*}
yields the poloynomial
$$2^{26}3^4 X^4 G(X,\tau_3,\tau_5).$$
By the above relations \eqref{eqn:3.2} and \eqref{eqn:3.3} we see that $G(\tau_0,\tau_3, \tau_5) = 0$.  Applying the map $\sigma$ gives that $G(\tau_4, \tau_5, \tau_3) = G(\tau_4, \tau_3, \tau_5) = 0$, since the coefficients of $G$ are symmetric in $\tau_3$ and $\tau_5$.  Similarly, the relations
\begin{align*}
\frac{\tau_3-\tau_1}{\tau_1} & = \frac{12(1+i)\beta(\beta - 2)(\beta + 2i)}{\beta^4-12\beta^3+24\beta^2-48\beta+16},\\
\frac{\tau_5-\tau_1}{\tau_1} & = \frac{12(1-i)\beta(\beta - 2)(\beta-2i)}{\beta^4-12\beta^3+24\beta^2-48\beta+16},
\end{align*}
and the resultant
\begin{align*}
\textrm{Res}_\beta&((\beta^4-12\beta^3+24\beta^2-48\beta+16)(\tau_3-X)-12(1+i)\beta(\beta - 2)(\beta+2i)X,\\
 & (\beta^4-12\beta^3+24\beta^2-48\beta+16)(\tau_5-X)-12(1-i)\beta(\beta - 2)(\beta-2i)X)\\
 & = 2^{26}3^4 X^4 G(X,\tau_3,\tau_5)
\end{align*}
yield that $\tau_1$ is a root of $G(X, \tau_3, \tau_5)$.  Applying $\sigma$ yields that $\tau_2$ is also a root.  This proves the theorem. $\square$ \bigskip

Note that
$$\textrm{disc}(G(X,\tau_3,\tau_5)) = -2^8(\tau_3+5\tau_5)^3(5\tau_3+\tau_5)^3(\tau_3-\tau_5)^6.$$
Furthermore,
$$(\tau_3+5\tau_5)(5\tau_3+\tau_5) =(j(\beta)-1728) R(\beta),$$
where
$$R(\beta) = \frac{36(\beta^4 + 4\beta^3 + 8\beta^2 - 16\beta + 16)^2(\beta^4 - 4\beta^3 + 8\beta^2 + 16\beta + 16)^2}{\beta^4(\beta + 2)^4(\beta - 2)^4}$$
is zero exactly when $j(\beta) =0$, which is excluded since we are assuming $d_K \neq -3, -4$.  Thus, in all cases we are considering, $G(X, \tau_3, \tau_5)$ has distinct roots.  \medskip

This theorem implies Sugawara's conjecture for $\mathfrak{m} = \mathfrak{p}_1^2 \mathfrak{p}_2^2$.  Namely, the invariants $\tau_3, \tau_5$ are the ray class invariants corresponding to a given ideal class $\mathfrak{k}$, since these are the only invariants for which $K(j(\mathfrak{k}),\tau(\mathfrak{k^*})) = \textsf{K}_\mathfrak{m} = K(j(\mathfrak{k}),i)$.  The other four invariants lie in $\textsf{K}_1 = K(\beta)$.  For this note that $\frac{\varphi(\mathfrak{m})}{2} = 2$.  Now, if $\mathfrak{k}, \mathfrak{k}'$ are two ideal classes for which $\{\tau_3, \tau_5\} = \{\tau_3', \tau_5'\}$, then the theorem implies that $S = \{\tau_i: i = 0, 1, 2, 4\}$ coincides with $S' = \{\tau_i': i = 0, 1, 2, 4\}$.  Hence, \eqref{eqn:3.1} implies that $j(\mathfrak{k}) = j(\mathfrak{k}')$ and $\mathfrak{k} = \mathfrak{k}'$. \medskip

If $(2) = \mathfrak{p}$ and $\mathfrak{m} = \mathfrak{p}^2$, then $\varphi(\mathfrak{m}) = 12$, in which case the conjecture follows immediately from \eqref{eqn:2} in Section \ref{sec:1}.  \medskip

Finally, if $(2) = \mathfrak{p}^2$, then $\mathfrak{m} = \mathfrak{p}^4$ and $\varphi(\mathfrak{m}) = 8 >6$, so this case also follows from \eqref{eqn:2} in Section \ref{sec:1}.  See \cite{su2}. \medskip

So far we have the following:

\begin{thm} For all divisors $\mathfrak{m}$ of the ideals $(2)$ or $(4)$, for which $\textsf{K}_\mathfrak{m} \neq \textsf{K}_1$, Sugawara's conjecture holds, namely, $\textsf{K}_\mathfrak{m} = K(\tau(\mathfrak{k}^*))$ is generated over the quadratic field $K$ by a single ray class invariant for the modulus $\mathfrak{m}$.
\label{thm:4}
\end{thm}

\section{The case $\mathfrak{m} = \wp_2^3$.}
\label{sec:4}

In this section we consider the case $d_K \equiv 1$ (mod $8$) and $\mathfrak{m} = \wp_2'^3$, to align with the notation in \cite{am1}.  (It involves no loss of generality to replace the ideal $\wp_2^3$ in \eqref{eqn:4} by $\wp_2'^3$.) \medskip

On the same curve we appealed to in Section \ref{sec:3},
$$E_4(b): \ Y^2+XY+bY = X^3+b X^2, \ \ b = \frac{1}{16}-\frac{1}{\beta^4},$$
the point $Q=(-2b,2b\beta_1 \beta_3)$ has order $4$, where
$$\beta_1 \beta_3 = \frac{\beta+2i}{2\beta} \frac{\beta-2i}{2\beta} = \frac{\beta^2+4}{4\beta^2}.$$
Note that $\xi = \beta/2$ and $\pi = \beta/(\zeta_8^j \alpha)$ (for some odd integer $j$; see \cite[pp. 1978, 1984]{lm}) satisfy $\pi^4 + \xi^4 =1$ and are the same quantities that we encountered in the proof of Lemma \ref{lem:1} in Section \ref{sec:2}.  \medskip

If $P=(x,y)$ is a point on $E_4(b)$ for which $2P = Q$, then $P$ has order $8$.  Hence, if $x=X(P)$, we have
$X(2P) = -2b$ or
\begin{equation*}
0 = \frac{(x^4 - bx^2 - 2b^2 x - b^3)}{(x + b)(4x^2 + x + b)} +2b = \frac{x^4 + 8bx^3 + b(8b + 1)x^2 + 2b^2 x + b^3}{(x + b)(4x^2 + x + b)}.
\end{equation*}
We now solve
$$f_8(x) = x^4 + 8bx^3 + b(8b + 1)x^2 + 2b^2 x + b^3 = 0,$$
with $b = \frac{1}{16}-\frac{1}{\beta^4} = \frac{1}{16}-\frac{1}{16\xi^4}= \frac{\pi^4}{16(\pi^4-1)}$.  We have
\begin{align*}
f_8(x-2b) &= x^4 + (-16b^2 + b)x^2 + (32b^3 - 2b^2)x - 16b^4 + b^3\\
& = x^4+px^2+qx+r,
\end{align*}
with
$$p = \frac{(\xi^4 - 1)}{16\xi^8}, \ \ q = -\frac{(\xi^4 - 1)^2}{128\xi^{12}}, \ \ r = \frac{(\xi^4 - 1)^3}{4096\xi^{16}}.$$
Then the cubic resolvent of $f_8(x)$ factors completely:
$$x^3 - 2px^2 + (p^2 - 4r)x + q^2=(x-\theta_1)(x-\theta_2)(x-\theta_3),$$
with roots
\begin{align*}
\theta_1 & = \frac{(\xi^4 - 1)}{16\xi^8} = -\frac{\pi^4}{16\xi^8},\\
\theta_2 & = -\frac{(\xi^2 + 1)(\xi^2 - 1)^2}{32\xi^8} = -\frac{\pi^4(1-\xi^2)}{32\xi^8},\\
\theta_3 & = \frac{(\xi^2 - 1)(\xi^2 + 1)^2}{32\xi^8} = -\frac{\pi^4(\xi^2+1)}{32\xi^8}.
\end{align*}
\bigskip
Thus, the roots of $x^4+px^2+qx+r=0$ are (by \cite[p. 182]{vdw})
\begin{align*}
x_1 & = \frac{1}{2}\left(\sqrt{-\theta_1}+\sqrt{-\theta_2} + \sqrt{-\theta_3}\right)\\
& = \frac{\pi^2}{8\xi^4} + \frac{\pi^2\sqrt{2} \sqrt{1-\xi^2}}{16\xi^4} + \frac{\pi^2 \sqrt{2} \sqrt{\xi^2+1}}{16\xi^4};\\
x_2 & = \frac{1}{2}\left(\sqrt{-\theta_1}-\sqrt{-\theta_2} - \sqrt{-\theta_3}\right)\\
& = \frac{\pi^2}{8\xi^4} - \frac{\pi^2\sqrt{2} \sqrt{1-\xi^2}}{16\xi^4} - \frac{\pi^2 \sqrt{2} \sqrt{\xi^2+1}}{16\xi^4};\\
 x_3 & = \frac{1}{2}\left(-\sqrt{-\theta_1}+\sqrt{-\theta_2} - \sqrt{-\theta_3}\right)\\
& = -\frac{\pi^2}{8\xi^4} + \frac{\pi^2\sqrt{2} \sqrt{1-\xi^2}}{16\xi^4} - \frac{\pi^2 \sqrt{2} \sqrt{\xi^2+1}}{16\xi^4};\\
x_4 & =\frac{1}{2}\left(-\sqrt{-\theta_1}-\sqrt{-\theta_2} + \sqrt{-\theta_3}\right)\\
& = -\frac{\pi^2}{8\xi^4} - \frac{\pi^2\sqrt{2} \sqrt{1-\xi^2}}{16\xi^4} + \frac{\pi^2 \sqrt{2} \sqrt{\xi^2+1}}{16\xi^4}.
\end{align*}
These roots correspond to points $P_i$ for which $X(P_i) = x_i-2b$ are the roots of $f_8(x) = 0$.
Since
$$x_1+x_2 =  \frac{\pi^2}{4\xi^4}, \ \ x_1 x_2 = \frac{\pi^6}{64\xi^8},$$
$x_1$ and $x_2$ are roots of
$$m(X) = X^2- \frac{\pi^2}{4\xi^4}X+\frac{\pi^6}{64\xi^8} \in \Sigma[X].$$
Similarly, $x_3$ and $x_4$ are roots of
$$\tilde{m}(x) = X^2+\frac{\pi^2}{4\xi^4}X - \frac{\pi^6}{64\xi^8} = X^2+(x_1 + x_2) X-x_1x_2.$$
The discriminant of $m(X)$ is
$$\textrm{disc}(m(X)) = \frac{\pi^4}{16\xi^8}-4\frac{\pi^6}{64\xi^8} = \frac{\pi^4(1-\pi^2)}{16\xi^8},$$
so the roots of $m(X)$ lie in $\Sigma(\sqrt{1-\pi^2})$.  Furthermore, $\sqrt{1-\pi^2} \sqrt{1+\pi^2} = \sqrt{1-\pi^4} = \xi^2$, 
so $\Sigma(\sqrt{1-\pi^2}) = \Sigma(\sqrt{1+\pi^2}) = \textsf{K}_{\wp_2'^3}=F$, by \cite[Thm. 1]{am1}.  Similarly,
$$\textrm{disc}(\tilde{m}(X)) = \frac{\pi^4}{16\xi^8}+4\frac{\pi^6}{64\xi^8} = \frac{\pi^4(1+\pi^2)}{16\xi^8},$$
and a root of $\tilde{m}(x)$ also generates $F/\Sigma$. \medskip

Note the relations
\begin{align*}
X(P_1)+X(P_2) &= x_1+x_2-4b = \frac{\pi^2}{4\xi^4} -\frac{4\pi^4}{16(\pi^4 - 1)}\\
& = \frac{-\pi^2- \pi^4}{4(\pi^4 - 1)}\\
& = -\frac{\pi^2}{4(\pi^2-1)};\\
X(P_3) +X(P_4) & = x_3+x_4-2b = -\frac{\pi^2}{4\xi^4} -\frac{4\pi^4}{16(\pi^4 - 1)}\\
& = \frac{\pi^2- \pi^4}{4(\pi^4 - 1)}\\
& = -\frac{\pi^2}{4(\pi^2+1)}.
\end{align*}
Setting $\lambda = -2^7 3^5 \frac{g_2 g_3}{\Delta}$, it follows that
\begin{align*}
\tau(\mathfrak{k}_1^*) + \tau(\mathfrak{k}_2^*) &= \lambda \left(X(P_1)+X(P_2) +2\frac{4b+1}{12}\right)\\
& =  \lambda \left(\frac{-\pi^2- \pi^4}{4(\pi^4 - 1)} + \frac{1}{6} \left(\frac{\pi^4}{4(\pi^4-1)}+1\right)\right)\\
& = -\lambda \frac{\pi^4+6\pi^2+4}{24(\pi^4-1)},
\end{align*}
and
\begin{align*}
\tau(\mathfrak{k}_3^*) + \tau(\mathfrak{k}_4^*) &= \lambda \left(X(P_3)+X(P_4) +2\frac{4b+1}{12}\right)\\
& =  \lambda \left(\frac{\pi^2- \pi^4}{4(\pi^4 - 1)} + \frac{1}{6} \left(\frac{\pi^4}{4(\pi^4-1)}+1\right)\right)\\
& = -\lambda \frac{\pi^4-6\pi^2+4}{24(\pi^4-1)}.
\end{align*}
This gives that
\begin{equation}
\label{eqn:4.1} \frac{\tau(\mathfrak{k}_3^*) + \tau(\mathfrak{k}_4^*)}{\tau(\mathfrak{k}_1^*) + \tau(\mathfrak{k}_2^*)} = \frac{\pi^4-6\pi^2+4}{\pi^4+6\pi^2+4}.
\end{equation}
Now I claim that the pairs $\{\tau(\mathfrak{k}_1^*), \tau(\mathfrak{k}_2^*)\}$ and $\{\tau(\mathfrak{k}_3^*), \tau(\mathfrak{k}_4^*)\}$ are the $\tau$-invariants for the ideals $\mathfrak{m}_1 = \wp_2'^3$ and $\mathfrak{m}_2 = \wp_2'^3 \wp_2$.  They cannot be the invariants for proper divisors of these ideals, because each of these pairs generates $\textsf{K}_{\wp_2'^3}$ over $\Sigma$, and the invariants for $\wp_2'^2$ and $\wp_2'^2 \wp_2$ lie in $\Sigma$, since their $\varphi$-values divided by $2$ are both $1$.  They cannot be the invariants for larger ideals dividing $(8)$, such as $\mathfrak{m}' = \wp_2'^3 \wp_2^2$, because this ideal satisfies Sugawara's condition \eqref{eqn:2} and $\varphi(\wp_2'^3 \wp_2^2) = 8 > 6$, so that the $\tau$-invariant for $\mathfrak{m}'$ generates $\textsf{K}_{\mathfrak{m}'}$, which has degree $4$ over $\Sigma$.  Furthermore, the field $\textsf{K}_{\wp_2^3}$ is disjoint from $\textsf{K}_{\wp_2'^3}$ over $\Sigma$.  Finally, $\mathfrak{m}_1$ and $\mathfrak{m}_2$ are the only two ideals dividing $(8)$ whose corresponding conductors are $\wp_2'^3$, and each of these ideals has two invariants.  Thus, since $m(x)$ and $\tilde{m}(x)$ are quadratic and irreducible over $\Sigma$, by the above calculations, each of the aforementioned pairs is a pair of conjugates over $\Sigma$ and corresponds to one of the ideals $\mathfrak{m}_1 = \wp_2'^3$ and $\mathfrak{m}_2 = \wp_2'^3 \wp_2$. 
\medskip

These considerations allow us to prove Sugawara's conjecture for the ideal $\mathfrak{m}_1$.  Suppose the polynomial $T_{\mathfrak{m}_1}(X,\mathfrak{k})$ equals $T_{\mathfrak{m}_1}(X,\bar{\mathfrak{k}})$ for two ideal classes $\mathfrak{k}, \bar{\mathfrak{k}}$.  Then there is an automorphism $\sigma$ of $\Sigma/K$ for which $j(\mathfrak{k})^\sigma = j(\bar{\mathfrak{k}})$, but $T_{\mathfrak{m}_1}(X,\mathfrak{k})^\sigma = T_{\mathfrak{m}_1}(X,\bar{\mathfrak{k}})$.  Since $\sigma$ fixes $K$, it also fixes the ideals $\mathfrak{m}_1$ and $\mathfrak{m}_2$.  Let $\psi$ be an extension of $\sigma$ to the field $\Sigma_{\wp_2'^3}$.  Then we have either that
\begin{align*}
\{\tau(\mathfrak{k}_1^*), \tau(\mathfrak{k}_2^*)\}^\psi &= \{\tau(\bar{\mathfrak{k}}_1^*), \tau(\bar{\mathfrak{k}}_2^*)\},\\
\{\tau(\mathfrak{k}_3^*), \tau(\mathfrak{k}_4^*)\}^\psi &= \{\tau(\bar{\mathfrak{k}}_3^*), \tau(\bar{\mathfrak{k}}_4^*)\}
\end{align*}
or
\begin{align*}
\{\tau(\mathfrak{k}_1^*), \tau(\mathfrak{k}_2^*)\}^\psi &= \{\tau(\bar{\mathfrak{k}}_3^*), \tau(\bar{\mathfrak{k}}_4^*)\},\\
\{\tau(\mathfrak{k}_3^*), \tau(\mathfrak{k}_4^*)\}^\psi &= \{\tau(\bar{\mathfrak{k}}_1^*), \tau(\bar{\mathfrak{k}}_2^*)\}.
\end{align*}
Hence, denoting the conjugate of $\pi$ corresponding to $ j(\bar{\mathfrak{k}})$ by $\bar{\pi} = \pi^\sigma$, (\ref{eqn:4.1}) yields two possibilities:
\begin{equation}
\label{eqn:4.2} \frac{\pi^4-6\pi^2+4}{\pi^4+6\pi^2+4} = \frac{\bar{\pi}^4-6\bar{\pi}^2+4}{\bar{\pi}^4+6\bar{\pi}^2+4}
\end{equation}
or
\begin{equation}
\label{eqn:4.3} \frac{\pi^4-6\pi^2+4}{\pi^4+6\pi^2+4} = \frac{\bar{\pi}^4+6\bar{\pi}^2+4}{\bar{\pi}^4-6\bar{\pi}^2+4}.
\end{equation}
The possibility \eqref{eqn:4.2} implies that
$$\frac{\pi^2+\frac{4}{\pi^2}-6}{\pi^2+\frac{4}{\pi^2}+6} = \frac{\bar{\pi}^2+\frac{4}{\bar{\pi}^2}-6}{\bar{\pi}^2+\frac{4}{\bar{\pi}^2}+6}.$$
Since both sides of this equation are linear fractional, this gives that
\begin{equation}
\label{eqn:4.4} \pi^2+\frac{4}{\pi^2} = \bar{\pi}^2+\frac{4}{\bar{\pi}^2}.
\end{equation}
However,
$$x^2+\frac{4}{x^2}-y^2-\frac{4}{y^2} = \frac{(xy - 2)(xy + 2)(-y + x)(x + y)}{x^2y^2},$$
and \eqref{eqn:4.4} implies that $\bar{\pi} $ equals one of $\pi, -\pi, 2/\pi$ or $-2/\pi$.  The last two are impossible, since $\bar{\pi} \cong \wp_2$ and $2/\pi \cong \wp_2'$.  if $\bar{\pi} = -\pi$, then $\pi$ and $-\pi$ would be conjugates and $b_d(x)$, the minimal polynomial of $\pi$, would satisfy $b_d(-x) = b_d(x)$ and be a polynomial in $x^2$.  But then $\pi^2$ would have degree less than the degree of $\pi$ over $\mathbb{Q}$, contradicting the fact that $\mathbb{Q}(\pi^2) = \mathbb{Q}(\pi^4) = \Sigma$.  (See \cite[Thm. 8.1]{lm}.)  Hence, \eqref{eqn:4.2} implies that $\bar{\pi} = \pi$.  \medskip

To show \eqref{eqn:4.3} cannot happen, we note the identity
$$\frac{x^2 - 6x + 4}{x^2 + 6x + 4} - \frac{y^2 + 6y + 4}{y^2 - 6y + 4} = -12\frac{(x + y)(xy + 4)}{(x^2 + 6x + 4)(y^2 - 6y + 4)}.$$
Thus, \eqref{eqn:4.3} would imply that $\bar{\pi}^2 = -\pi^2$ or $-4/\pi^2$, and both are impossible because $i = \sqrt{-1} \notin \Sigma$. \medskip

Now, $\bar{\pi} = \pi$ implies that $\bar{\beta} = 2\bar{\xi} = \pm 2\xi = \pm \beta$; and the fact that the $j$-invariant $j(E_4)$ is a rational function in $\beta^2$ (see Section \ref{sec:3}) shows that $j(\bar{\mathfrak{k}}) = j(\mathfrak{k})$.  This proves that $T(X,\mathfrak{k}) = T(X,\bar{\mathfrak{k}})$ can only happen if $j(\mathfrak{k}) = j(\bar{\mathfrak{k}})$ and therefore $\mathfrak{k} = \bar{\mathfrak{k}}$.  Therefore, Sugawara's conjecture holds for the ideals $\mathfrak{m} = \wp_2'^3$ and (by complex conjugation) $\wp_2^3$. \medskip

With this, we have proved Sugawara's conjecture for all four possibilities in the first line of \eqref{eqn:4}.  Note that if $(2) = \wp_2^2$, then $\varphi(\wp_2^2) = 2 = \varphi(\wp_2^3)/2$, so $\mathfrak{m} = \wp_2^3$ is not the conductor of $\textsf{K}_\mathfrak{m} = \textsf{K}_{\wp_2^2}$.  Hence, we do not have to consider this case when $2$ is ramified.  \medskip

\noindent {\bf Remark.} We note the cross-ratio
\begin{align*}
\frac{(x_1-x_3)(x_2-x_4)}{(x_1-x_4)(x_2-x_3)} &= \frac{\xi^6 + 2\pi^4 + \xi^4 - \xi^2 - 1}{-\xi^6 + 2\pi^4 + \xi^4 + \xi^2 - 1}\\
& = \frac{\xi^6 + \pi^4 - \xi^2}{-\xi^6 + \pi^4 + \xi^2}\\
& =  \frac{\xi^6 + 1-\xi^4 - \xi^2}{-\xi^6 + 1-\xi^4 + \xi^2}\\
& = \frac{(\xi^2 + 1)(\xi - 1)^2(\xi + 1)^2}{-(\xi - 1)(\xi + 1)(\xi^2 + 1)^2} = \frac{1-\xi^2}{1+\xi^2}.
\end{align*}
Since the $\tau$-invariants corresponding to the points $P_i$ are
\begin{align*}
\tau(\mathfrak{k}^*_i) &= -2^7 3^5\frac{g_2 g_3}{\Delta}\left(X(P_i) + \frac{4b+1}{12}\right)\\
& = -2^7 3^5\frac{g_2 g_3}{\Delta}\left(x_i-2b + \frac{4b+1}{12}\right),
\end{align*}
the cross-ratio of the $x_i$ equals the cross-ratio of the invariants $\tau(\mathfrak{k}^*_i)$.

\section{The case $\frak{m} = (3)$.}
\label{sec:5}

For the next three sections we work on the Deuring normal form of an elliptic curve:
$$E_3: \ Y^2+\alpha X Y+Y = X^3.$$
The points $P_1 = (0,0)$ and
\begin{equation}
\label{eqn:5.1} P_2 = \left(\frac{-3\beta}{\alpha(\beta-3)},\frac{\beta-3\omega^i}{\beta-3}\right),
\end{equation}
where $27\alpha^3+27\beta^3=\alpha^3 \beta^3$ and $\omega=\frac{-1+\sqrt{-3}}{2}$, are points of order $3$ on $E_3$.  For any discriminant $d_K \equiv 1$ (mod $3$), there are $\alpha, \beta$ which generate (separately) the Hilbert class field $\Sigma$ over $\mathbb{Q}$, and for which $E_3$ has complex multiplication by $R_K$. See \cite{m1}. Let $(3) = \wp_3 \wp_3' = \mathfrak{p}_1 \mathfrak{p}_2$ in $R_K$.  In addition $(\alpha, 3) = \wp_3'$ and $(\beta, 3) = \wp_3$.  \medskip

The Weierstrass normal form of the curve $E_3$ is
$$E': \ Y^2 = 4X^3 -g_2 X - g_3,$$
where
$$g_2= \frac{1}{12}(\alpha^4-24\alpha), \ \ g_3 = \frac{-1}{216}(\alpha^6-36 \alpha^3+216);$$
and $\Delta = \alpha^3 -27$.  Thus, 
$$j(E_3) = \frac{\alpha^3(\alpha^3-24)^3}{\alpha^3-27} = \frac{\beta^3(\beta^3+216)}{(\beta^3-27)^3} = j(\mathfrak{k}),$$
for some ideal class $\mathfrak{k}$, where the expression in $\beta$ is obtained using $\alpha^3 = \frac{27\beta^3}{\beta^3-27}$.  The ray class invariant
$$\tau(\mathfrak{k}^*) = -2^7 3^5 \frac{g_2 g_3}{\Delta}\left(X(P)+\frac{\alpha^2}{12}\right)$$
is the invariant for a suitable ray class $\mathfrak{k}^*$ for the modulus $\mathfrak{m} = \mathfrak{p}_1, \mathfrak{p}_2$, or $(3) = \mathfrak{p}_1 \mathfrak{p}_2$ in $K$.  Let
\begin{align*}
\tau_1 &= -2^7 3^5 \frac{g_2 g_3}{\Delta}\left(0+\frac{\alpha^2}{12}\right)\\
&= \frac{\alpha^3 (\alpha^3 - 24)(\alpha^6 - 36\alpha^3 + 216)}{\alpha^3-27},\\
\tau_2 &=\tau(\mathfrak{k}^*) = -2^7 3^5 \frac{g_2 g_3}{\Delta}\left(\frac{-3\beta}{\alpha(\beta-3)}+\frac{\alpha^2}{12}\right)\\
& = \frac{(\alpha^3-24)(\alpha^6-36\alpha^3+216)}{\alpha^3-27} \frac{\alpha^3(\beta-3)-36\beta}{(\beta-3)}
\end{align*}
be the $\tau$-invariants for the points $P_1$ and $P_2$  in \eqref{eqn:5.1}.  These two invariants clearly lie in $\Sigma$, so they are the invariants for the ideals $\mathfrak{p}_1, \mathfrak{p}_2$.
Replacing $\beta$ by $\omega \beta$, respectively $\omega^2 \beta$ results in replacing the point $P$ by the points
\begin{align*}
P_3 &= \left(\frac{-3\omega \beta}{\alpha(\omega \beta-3)},\frac{\omega \beta-3\omega^i}{\omega \beta-3}\right),\\
P_4 &= \left(\frac{-3\omega^2 \beta}{\alpha(\omega^2 \beta-3)},\frac{\omega^2 \beta-3\omega^i}{\omega^2 \beta-3}\right);
\end{align*}
which also lie in $E_3[3]$.  The corresponding $\tau$-invariants are
\begin{align*}
\tau_3 &= \frac{(\alpha^3-24)(\alpha^6-36\alpha^3+216)}{\alpha^3-27} \frac{\alpha^3(\omega \beta-3)-36\omega \beta}{(\omega \beta-3)},\\
\tau_4 &= \frac{(\alpha^3-24)(\alpha^6-36\alpha^3+216)}{\alpha^3-27} \frac{\alpha^3(\omega^2\beta-3)-36\omega^2\beta}{(\omega^2\beta-3)},
\end{align*}
which lie in $\textsf{K}_{3} = \Sigma(\omega)$ and are conjugate over $\Sigma$.  These are the invariants for $\mathfrak{m} = (3)$.  Replacing $\alpha^3$ by $\frac{27\beta^3}{\beta^3-27}$ in the above expressions yields the following formulas for $\tau_3, \tau_4$ in terms of $\beta$:
\begin{align*}
\tau_3 &= \frac{\beta(\beta^3+216)(\beta^6-540\beta^3-5832)}{(\beta^3-27)^3} \frac{(\beta^3 \omega+9\beta^2-108\omega)}{\omega\beta-3},\\
\tau_4 &=  \frac{\beta(\beta^3+216)(\beta^6-540\beta^3-5832)}{(\beta^3-27)^3} \frac{(\beta^3 \omega^2+9\beta^2-108\omega^2)}{\omega^2 \beta-3}.
\end{align*}
\medskip

A computation shows that
$$\frac{\tau_2-\tau_1}{\tau_1} = \frac{12}{\alpha^2}X(P_2) = \frac{-36\beta}{\alpha^3(\beta-3)}.$$
Replacing $\beta$ by $\omega \beta$ and $\omega^2 \beta$ yields
\begin{align*}
\frac{\tau_3-\tau_1}{\tau_1} &= \frac{12}{\alpha^2}X(P_3) = \frac{-36\beta}{\alpha^3(\beta-3\omega^2)}\\
\frac{\tau_4-\tau_1}{\tau_1} &= \frac{12}{\alpha^2}X(P_4) = \frac{-36\beta}{\alpha^3(\beta-3\omega)}.
\end{align*}
Taking quotients yields that
$$\frac{\tau_2 - \tau_1}{\tau_3-\tau_1} = \frac{\beta-3\omega^2}{\beta-3}, \ \ \frac{\tau_2 - \tau_1}{\tau_4-\tau_1} = \frac{\beta-3\omega}{\beta-3}.$$
These are the $Y$-coordinates of the points $P_2, -P_2$.  Replacing $\beta$ again by $\omega^i \beta$ for $i = 1,2$ shows that the $Y$-coordinates of all points in $E_3[3]$ are contained in $\textsf{K} = \mathbb{Q}(\tau_1, \tau_2, \tau_3, \tau_4)$. \medskip

Multiplying the two $Y$ coordinates above shows that
$$\frac{\beta-3\omega^2}{\beta-3} \frac{\beta-3\omega}{\beta-3} = \frac{\beta^2+3\beta+9}{(\beta-3)^2} \in \textsf{K}.$$
Now, by the equation for $E_3$, we have for $P_2=(x,y)$ that $y^2+y = x^3-\alpha x y.$
However, using $\alpha^3 = \frac{27\beta^3}{\beta^3-27}$ yields by \eqref{eqn:5.1} that
$$x^3 = \frac{-27\beta^3}{\alpha^3(\beta-3)^3} = -\frac{\beta^2+3\beta+9}{(\beta-3)^2} \in \textsf{K}.$$
This implies by the equation for $E_3$ that $\alpha x = \frac{-3\beta}{\beta-3} \in \textsf{K}$.  This gives, finally, that $\beta \in \textsf{K}$, and the above formulas imply that $j(\mathfrak{k}), \omega \in \textsf{K}$, as well.  Hence $\textsf{K} = \textsf{K}_3$.

\begin{thm}
The ray class field $\textsf{K}_3$ over $K = \mathbb{Q}(\sqrt{d_K})$, with $d_K \equiv 1$ (mod $3$), is generated over $\mathbb{Q}$ by the ray class invariants $\tau(\mathfrak{k}^*)$ for the divisors $\mathfrak{m} \neq 1$ of $(3)$ corresponding to any single absolute ideal class $\mathfrak{k}$.
\label{thm:5}
\end{thm}

We also have
$$\frac{\tau_2}{\tau_1}+\frac{\tau_3}{\tau_1}+\frac{\tau_4}{\tau_1} = \frac{-1}{3\beta^2}((\beta+6)^2+(\beta+6\omega^2)^2+(\beta+6\omega)^2)=-1;$$
hence,
\begin{equation}
\label{eqn:5.2} \tau_2+\tau_3+\tau_4=-\tau_1.
\end{equation}
This shows that $\textsf{K}_3 = K(\tau_2, \tau_3, \tau_4)$.  It remains to show that $\tau_2 \in K(\tau_3, \tau_4)$ and that $\Sigma= K(\tau_1) = K(\tau_2)$. \medskip

Computing the other elementary functions of the $\tau_i$ on Maple yields:
\begin{align*}
\sum_{m \neq n}{\tau_m \tau_n} &= -6j(\mathfrak{k})(j(\mathfrak{k})-1728),\\
\sum_{m \neq n \neq l}{\tau_m \tau_n \tau_l} &= -8j(\mathfrak{k})(j(\mathfrak{k})-1728)^2,\\
\prod_{i=1}^4{\tau_i} &= -3j(\mathfrak{k})^2(j(\mathfrak{k})-1728)^2.
\end{align*}
Hence the polynomial satisfied by the $\tau_i$ is
\begin{align*}
F(X,j(\mathfrak{k})) &= X^4 -6j(\mathfrak{k})(j(\mathfrak{k})-1728)X^2+8j(\mathfrak{k})(j(\mathfrak{k})-1728)^2 X \\
& \ \ -3j(\mathfrak{k})^2(j(\mathfrak{k})-1728)^2,
\end{align*}
which gives the relation
\begin{equation}
\label{eqn:5.3} \sum_{i=1}^4{\frac{1}{\tau_i}} = \frac{8}{3j(\mathfrak{k})}.
\end{equation}
Using this relation we see that if the $\tau_i$ ($1 \le i \le 4$) are the same for two different ideal classes $\mathfrak{k}, \mathfrak{k}'$, then $j(\mathfrak{k}) = j(\mathfrak{k}')$, so that $\mathfrak{k} = \mathfrak{k}'$. \medskip

Note the cross ratio
\begin{equation}
\label{eqn:5.4} (\tau_4,\tau_3; \tau_2, \tau_1) = \frac{(\tau_4-\tau_2)(\tau_3-\tau_1)}{(\tau_3-\tau_2)(\tau_4-\tau_1)} = -\omega.
\end{equation}
From above we have
$$\frac{\tau_4-\tau_1}{\tau_3-\tau_1} = \frac{\beta-3\omega^2}{\beta-3\omega}.$$
Now, we use the fact that $\sigma_1: \beta \rightarrow \frac{3(\beta+6)}{\beta-3}$ is an automorphism of $\Sigma/\mathbb{Q}$ fixing $\omega$ (see \cite{m1}).  An easy computation using the formulas for the $\tau_i$ in terms of $\beta$ shows that $\sigma_1$ interchanges the pairs $(\tau_1, \tau_2), (\tau_3, \tau_4)$.  We obtain that
$$\frac{\tau_3-\tau_2}{\tau_4-\tau_2} = \frac{\sigma_1(\beta)-3\omega^2}{\sigma_1(\beta)-3\omega} =-\omega^2 \frac{\beta-3\omega}{\beta-3\omega^2};$$
from which the above cross-ratio follows.  The cross-ratio and \eqref{eqn:5.2} then yield the following equation satisfied by $x = \tau_2$:
\begin{align*}
0  & = (2\tau_3 + x + \tau_4)(\tau_4 - x) + \omega(2\tau_4 + x + \tau_3)(\tau_3 - x)\\
 & = \omega^2x^2 -2(\tau_3 + \omega \tau_4)x + \omega (\tau_3-\omega \tau_4)^2.
\end{align*} 
Multiplying by $\omega$ gives the equation
\begin{equation}
\label{eqn:5.5} x^2 -2(\omega \tau_3 + \omega^2 \tau_4)x + (\omega \tau_3-\omega^2 \tau_4)^2 = 0,
\end{equation}
whose discriminant is $16\tau_3 \tau_4$.  Multiplying this equation by the equation obtained by replacing $\omega$ by $\omega^2$ gives the following quartic equation:
\begin{align*}
G(x; \tau_3,\tau_4) = &x^4+2(\tau_3+\tau_4)x^3+(3\tau_3^2 - 8\tau_3 \tau_4 + 3\tau_4^2)x^2\\
& +2(\tau_3 + \tau_4)(\tau_3^2 - 5\tau_3 \tau_4 + \tau_4^2)x +(\tau_3^2 + \tau_3 \tau_4 + \tau_4^2)^2.
\end{align*}
The second root of the quadratic equation (\ref{eqn:5.5}) is $\tilde \tau = \left(\frac{\beta-6}{\beta}\right)^2 \tau_1$, since
\begin{equation}
\label{eqn:5.6} \tau_2 + \tilde \tau = 2(\omega \tau_3 + \omega^2 \tau_4).
\end{equation}

\noindent {\it Proof of Sugawara's conjecture for $\mathfrak{m} = (3) = \mathfrak{p}_1 \mathfrak{p}_2$.} \medskip

Now suppose $\mathfrak{k}$ and $\mathfrak{k}'$ are two ideal classes for which $\{\tau_3, \tau_4\} = \{\tau_3', \tau_4'\}$.  Then the automorphism $\psi: j(\mathfrak{k}) \rightarrow j(\mathfrak{k}')$ over $K$ satisfies $\psi(\beta) = \beta'$, for some root $\beta'$ of $p_{d_K}(x)$, the minimal polynomial of $\beta$ over $\mathbb{Q}$.  (See \cite{m1}.)  We may extend this automorphism to $\textsf{K}_3 = \Sigma(\omega)$ by fixing $\omega$, since $d_K \equiv 1$ (mod $3$) implies that $\Sigma$ and $K(\omega)$ are linearly disjoint over $K$.  Then $\tau_3' = \tau_3^\psi, \tau_4' = \tau_4^\psi$.  \medskip

Assume first that $\tau_3' = \tau_3$ and $\tau_4' = \tau_4$.  Equation \eqref{eqn:5.5} shows that the roots $\tau_2', \left(\frac{\beta'-6}{\beta'}\right)^2\tau_1'$ have to coincide with the roots $\tau_2, \left(\frac{\beta-6}{\beta}\right)^2\tau_1$.  If $\tau_2' = \tau_2$, then $\tau_1' = \tau_1$ by the relation $\tau_2 + \tau_ 1 = -(\tau_3+\tau_4)=\tau_2' + \tau_1'$ (or using the cross-ratio).  In that case $j(\mathfrak{k}) = j(\mathfrak{k}')$ by \eqref{eqn:5.3}, so $\mathfrak{k} = \mathfrak{k}'$.  Assume instead that
\begin{align*}
\tau_2'& = \left(\frac{\beta-6}{\beta}\right)^2\tau_1,\\
\tau_2 & = \left(\frac{\beta'-6}{\beta'}\right)^2 \tau_1', \ \textrm{or}\\
\tau_1' &= \left(\frac{\beta'}{\beta'-6}\right)^2 \tau_2.
\end{align*}
Since $\psi$ is an automorphism, the cross-ratio \eqref{eqn:5.4} implies that
$$(\tau_4,\tau_3; \tau_2', \tau_1') = \frac{(\tau_4-\tau_2')(\tau_3-\tau_1')}{(\tau_3-\tau_2')(\tau_4-\tau_1')} = -\omega.$$
However, replacing $\tau_2'$ and $\tau_1'$ by the above expressions in terms of $\tau_1$ and $\tau_2$ gives
$$(\tau_4,\tau_3; \tau_2', \tau_1')+\omega = \frac{-\omega^2(\beta \beta' - 3\beta - 3\beta' + 36)(3\beta' + \beta - 12)(\beta - 3)}{(\beta \beta' - 3\omega^2 \beta' - 3\beta - 18\omega)(\omega \beta'- \beta - \beta' - 6\omega)(3\omega - \beta + 9)}.$$
Setting the numerator in this expression equal to zero and solving for $\beta'$ yields that
$$\beta' = \frac{-\beta}{3}+4, \ \ \textrm{or} \ \ \beta' = \frac{3(\beta-12)}{\beta-3}.$$
The first relation is impossible, since $\beta' = \beta^\psi$ is a conjugate of $\beta$, but $\frac{-1}{3}\beta$ is not an algebraic integer, using that $(\beta, 3) = \wp_3 =\mathfrak{p}_1$ from \cite[Lemma 2.3, Prop. 3.2]{m1}.  Thus we must have
$$\beta^\psi = \psi(\beta) = \frac{3(\beta-12)}{\beta-3}.$$
But we know that $\beta^{\sigma_1} = \frac{3(\beta+6)}{\beta-3}$, and thus
$$\beta^{\sigma_1 \psi} =  \psi \circ \sigma_1(\beta) = 6-\beta.$$
Now $\psi \in \textrm{Gal}(\textsf{K}_3/K)$, by assumption, and $\sigma_1 \notin \textrm{Gal}(\textsf{K}_3/K)$, since $\sigma_1$ switches the ideals $\wp_3 = \mathfrak{p}_1$ and $\wp_3' = \mathfrak{p}_2$.  For this, note that $\sigma_1(\beta) = 3 + \frac{27}{\beta-3}$ and $\beta = 3 +\gamma^3$, where $(\gamma) = \wp_3$ in $\Sigma$.  (Use \cite[Thm 3.4(i), p. 868]{m1} and the automorphism $\phi = \sigma_1 \circ \left(\Sigma/K, \wp_2\right)$ which switches $\alpha$ and $\beta$ from \cite[Prop. 3.2, p.865]{m1}.)  It follows that $\psi \sigma_1 = \sigma_1 \psi \notin \textrm{Gal}(\textsf{K}_3/K)$, either.  On the other hand
$$\beta^{\sigma_1\psi} - 3 = 3 - \beta \cong \wp_3^3 \cong \beta - 3.$$
This would show that $\psi \sigma_1$ does {\it not} switch $\wp_3$ and $\wp_3'$, giving a contradiction.  Hence, this case does not occur and we have the desired conclusion $\mathfrak{k} = \mathfrak{k}'$.  \medskip

Now assume that $\tau_3' = \tau_3^\psi = \tau_4$ and $\tau_4' = \tau_3$.  Then we have
$$(\tau_3,\tau_4; \tau_2', \tau_1') = \frac{(\tau_3-\tau_2')(\tau_4-\tau_1')}{(\tau_4-\tau_2')(\tau_3-\tau_1')} = -\omega,$$
which implies that
\begin{equation}
\label{eqn:5.7} (\tau_4,\tau_3; \tau_2', \tau_1') =-\omega^2.
\end{equation}
This leads to the conjugate equation of \eqref{eqn:5.5}, which is
$$x^2 -2(\omega^2 \tau_3 + \omega \tau_4)x + (\omega^2 \tau_3-\omega \tau_4)^2 = 0.$$
This equation also arises from \eqref{eqn:5.5} by applying the automorphism $\sigma_1$, which switches $\tau_3$ and $\tau_4$.  Therefore, its roots are $\tau_2^{\sigma_1} = \tau_1$ and
$$\left(\frac{\beta-6}{\beta}\right)^{2\sigma_1} \tau_1^{\sigma_1} = \left(\frac{\beta-12}{\beta+6}\right)^2 \tau_2.$$
Once again we have two cases, according as $\tau_1' = \tau_1$ or $\tau_1' = \left(\frac{\beta-12}{\beta+6}\right)^2 \tau_2$.  The first case implies as before that $\tau_2' = \tau_2$, so that $j(\mathfrak{k}) = j(\mathfrak{k}')$ from \eqref{eqn:5.3}.  Otherwise we have
$$\tau_1' = \left(\frac{\beta-12}{\beta+6}\right)^2 \tau_2, \ \ \tau_2' = \left(\frac{\beta'+6}{\beta'-12}\right)^2 \tau_1.$$
From \eqref{eqn:5.7} we find that the numerator of $(\tau_4,\tau_3; \tau_2', \tau_1') + \omega^2 = 0$ is
\begin{equation*}
\nu = (2\omega+1)[(\beta'^2 + 3\beta' + 63)\beta^3 -9(\beta'^2 -6\beta' +90)\beta^2 + 27(\beta' - 12)^2\beta - 27(\beta' - 12)^2].
\end{equation*}
It follows from $\nu = 0$ and $(9) = \wp_3^2 \wp_3'^2$ that $\wp_3'^2 \mid (\beta'^2 + 3\beta' + 63)\beta^3$.  However, $(\wp_3' , \beta) = 1$; and
$$\beta'^2 + 3\beta' + 63 = (\beta'-3)^2+9(\beta'-3)+81$$
is also not divisible by any prime divisor of $\wp_3'$ in $\Sigma$, since $\beta' - 3 = \gamma'^3 \cong \wp_3^3$.  This contradiction shows that this case cannot occur. \medskip

This proves the following.

\begin{thm} For the field $K = \mathbb{Q}(\sqrt{d_K})$, with $d_K \equiv 1$ (mod $3$), and $\mathfrak{m} = (3) =\wp_3 \wp_3'$, the ray class invariants corresponding to different ideal classes $\mathfrak{k}$ are distinct.  Hence, the ray class field $\textsf{K}_\mathfrak{m} = K(\tau(\mathfrak{k^*}))$ is generated over $K$ by a single ray class invariant $\tau(\mathfrak{k^*})$ ($=\tau_3$ or $\tau_4$) for the conductor $\mathfrak{m}$.
\label{thm:6}
\end{thm}

If $\mathfrak{m} = (3) = \mathfrak{p}^2$, then there are $\frac{\varphi(\mathfrak{m})}{2} = 3$ ray class invariants for each ideal class, so the argument is the same as in the case $\mathfrak{m} = \mathfrak{q}^2 = (2)$ in Section \ref{sec:2}.  If the invariants for $\mathfrak{k}, \mathfrak{k}'$ are the same, then by \eqref{eqn:5.2}, the single ray class invariants for these ideal classes and $\mathfrak{m} = \mathfrak{p}$ are also equal, and then \eqref{eqn:5.3} shows that $j(\mathfrak{k}) = j(\mathfrak{k}')$.  Thus, Sugawara's conjecture holds for $\mathfrak{m} = (3) = \mathfrak{p}^2$.  In this case, note that the conditions $\alpha, \beta \in \Sigma$ no longer obviously apply, but the algebraic formulas, including \eqref{eqn:5.2}, remain valid.  \medskip

Note that if $3 \cong \mathfrak{p}$, then $N(\mathfrak{p}) = 9$ implies that Sugawara's second argument \eqref{eqn:3} applies.  This also follows immediately from \eqref{eqn:5.3}, since the $4$ ray class invariants are conjugate over $\textsf{K}_1$ in this case. \medskip

This raises the question whether $\tau_1$ and $\tau_2$ individually generate $\textsf{K}_1$ over $K$, when $\mathfrak{m} = \mathfrak{p}_1$ or $\mathfrak{p}_2$ and $(3) = \mathfrak{p}_1 \mathfrak{p}_2$.  In the case of $\tau_1$, we have
\begin{align*}
\tau_1 &= \frac{\alpha^3 (\alpha^3-24)(\alpha^6-36\alpha^3+216)}{\alpha^3-27} = g(\alpha^3);\\
g(x) & = \frac{x(x-24)(x^2-36x+216)}{x-27}.
\end{align*}
If $\psi \in \textrm{Gal}(\textsf{K}_1/K)$ fixes $\tau_1(\mathfrak{k})$, then $\tau_1 = \tau_1^\psi = \tau_1' = \tau_1(\alpha')$, since $\psi$ fixes whichever ideal $\mathfrak{p}_1$ or $\mathfrak{p}_2$ the invariant $\tau_1$ belongs to.  Now factor the difference $g(x)-g(y)$ for $x=\alpha^3, y=\alpha'^3$:
\begin{equation*}
(x-27)(y-27)(g(x)-g(y)) = (x-y)(x^3y + x^2 y^2 + x y^3+3h(x,y)),
\end{equation*}
with $h(x,y) \in \mathbb{Z}[x,y]$.  If $x \neq y$, then the cofactor in this equation is $0$, which implies
$$x y(x^2+xy+y^2) \equiv 0 \ (\textrm{mod} \ \mathfrak{p}_1).$$
Since $\mathfrak{p}_1= \wp   _3$ is relatively prime to $xy = \alpha^3 \alpha'^3$, this gives that
$$x^3 - y^3 = (x-y)(x^2+xy+y^2) \equiv 0 \ (\textrm{mod} \ \mathfrak{p}_1).$$
Hence $\alpha^9 \equiv \alpha'^9$ (mod $\mathfrak{p}_1$).  We know that
$$j(\mathfrak{k}^*) = \frac{\alpha^3(\alpha^3-24)^3}{\alpha^3-27} \equiv \alpha^9 \ (\textrm{mod} \ \mathfrak{p}_1),$$
which implies
$$j(\mathfrak{k}^*)  \equiv \alpha^9 \equiv \alpha'^9 \equiv j(\mathfrak{k}'^*) \ (\textrm{mod} \ \mathfrak{p}_1).$$
But $j(\mathfrak{k}^*)$ and $j(\mathfrak{k}'^*)$ are roots of the class equation $H_{-d}(X)$, whose discriminant is not divisible by $3$.  (See \cite{d1a}.)  It follows from the last congruence that $j(\mathfrak{k}^*) = j(\mathfrak{k}'^*)$ and $\mathfrak{k}^* = \mathfrak{k}'^*$.  If $x=y$, then $\alpha^3 = \alpha'^3$ immediately gives $j(\mathfrak{k}^*) = j(\mathfrak{k}'^*)$ and the same conclusion. \medskip

Now the automorphism $\sigma_1$ of $\textsf{K}_1/\mathbb{Q}$ satisfies $\tau_1^{\sigma_1} = \tau_2$, so $\textsf{K}_1 = K(\tau_1) = K(\tau_2)$.

\begin{thm}
If $d_K \equiv 1$ mod $3$ and $(3) = \mathfrak{p}_1 \mathfrak{p}_2$, then the $\tau$-invariants for both $\mathfrak{p}_1$ and $\mathfrak{p}_2$ are distinct, so that the Hilbert class field $\textsf{K}_1$ of $K$ is generated by a single $\tau$-invariant for either $\mathfrak{p}_1$ or $\mathfrak{p}_2$.
\label{thm:7}
\end{thm}

\section{The case $\mathfrak{m} \mid 9$.}
\label{sec:6}

\subsection{The case $\mathfrak{m} = \wp_3'^2$.}

The Tate normal form for a point of order $9$ is
\begin{equation}
\label{eqn:6.1} E_9(t): \ Y^2+(1+t^2-t^3)XY+(1-t)(1-t+t^2)t^2Y = X^3+(1-t)(1-t+t^2)t^2X^2.
\end{equation}
Its $j$-invariant is
\begin{equation}
\label{eqn:6.2} j(E_9) = \frac{(t^3 - 3t^2 + 1)^3(t^9 - 9t^8 + 27t^7 - 48t^6 + 54t^5 - 45t^4 + 27t^3 - 9t^2 + 1)^3}{t^9(t - 1)^9(t^2 - t + 1)^3 (t^3 - 6t^2 + 3t + 1)}.
\end{equation}
This can be verified using the polynomial
\begin{equation*}
f_9(a,b) = a^5 - 6a^4 + a^3 b + 15a^3 - 6a^2 b + 3a b^2 - b^3 - 19a^2 + 9ab - 3b^2 + 12a - 4b - 3,
\end{equation*}
since $f_9(a,b) = 0$ is the condition that $P = (0,0)$ represents a point of order $9$ on the curve
$$Y^2+aXY+bY = X^3+bX^2.$$
(See \cite[pp. 248-250]{m}.  Note that the formula for $j(E_n)$ should have a minus sign.)  The curve $f_9(a,b) = 0$ has genus $0$ and is parametrized by
$$a = 1+t^2-t^3, \ \ b = (1-t)(1-t+t^2)t^2.$$
Now let
$$g(t) = \frac{t^3-3t^2+1}{t(t-1)}.$$
Noting that
$$t^9 - 9t^8 + 27t^7 - 48t^6 + 54t^5 - 45t^4 + 27t^3 - 9t^2 + 1 = (t^3-3t^2+1)^3-24t^3(t-1)^3$$
and
\begin{align*}
(t^2 - t + 1)^3 & =  (t^3 - 3t^2 + 1)^2 + 3(t^3 - 3t^2 + 1)t(t - 1) + 9t^2(t - 1)^2,\\
t^3 - 6t^2 + 3t + 1 & = t^3 - 3t^2 + 1 - 3(t^2 - t),
\end{align*}
it follows that
\begin{equation}
\label{eqn:6.3} j(E_9(t)) = \frac{g(t)^3(g(t)^3-24)^3}{g(t)^3-27}.
\end{equation}
This implies that $E_9$ is isomorphic to the Deuring normal form
$$E_3(\alpha): \ Y^2+\alpha XY + Y = X^3, \ \ \alpha = g(t),$$
whose $j$-invariant is
$$j(E_3) = \frac{\alpha^3(\alpha^3-24)^3}{\alpha^3-27}.$$
Now by Proposition 3.6(ii) of \cite{m0} and the remark thereafter, a point $P=(\xi,\eta)$ on $E_3(\alpha)$ satisfies $3P = \pm(0,0)$ whenever its $X$-coordinate satisfies $x^3 -(3+\alpha)x^2 +\alpha x+1 = 0$.  But this equation implies the relation 
$$\alpha = \frac{\xi^3-3\xi^2+1}{\xi(\xi-1)} = g(\xi).$$
Hence, the point $P= (\xi,\eta)$ is a point of order $9$ on $E_3(\alpha)$.  We will see that $P$ is a primitive $\mathfrak{m}$-division point on $E_3(\alpha)$. \medskip

We have the discriminant formula
$$\textrm{disc}(x^3 -(3+\alpha)x^2 +\alpha x+1) = (\alpha^2+3\alpha+9)^2.$$
In order to decide whether Sugawara's conjecture is true in the case $(3) = \wp_3\wp_3'$ and $\mathfrak{m} = \wp_3'^2$, we let $\alpha \in \Sigma$ be as in Section \ref{sec:5}.  Then $\alpha-3 \cong \wp_3'^3$ and $(\alpha) = \wp_3' \mathfrak{a}$, with $(\mathfrak{a},3) = 1$, by \cite{m1}.  It follows that the above discriminant is relatively prime to $\wp_3$.  Furthermore,
$$(\alpha^2+3\alpha+9)(\alpha-3) = \alpha^3-27 = \frac{27\alpha^3}{\beta^3} \cong \wp_3'^6 \mathfrak{c},$$
for some integral ideal $\mathfrak{c}$ prime to $(3)$.  In fact $\mathfrak{c} = (1)$, since
$$(\alpha^3-27)(\beta^3-27) = \alpha^3 \beta^3-27\alpha^3-27\beta^3 + 27^2 = 3^6.$$
Thus, it is clear that $(\alpha^2+3\alpha+9) \cong \wp_3'^3$. \medskip

Assuming $k(x) = x^3 -(3+\alpha)x^2 +\alpha x+1$ is irreducible over $\Sigma$, its root $\xi$ generates a cyclic cubic extension of $\Sigma$; its conjugates over $\Sigma$ are $\frac{1}{1-\xi}$ and $\frac{\xi-1}{\xi}$.  Also, since $g_2, g_3$ and $\Delta$ for the curve $E_3(\alpha)$ lie in $\Sigma$, the fact that the invariant
\begin{align*}
\tau(\mathfrak{k}^*) & = -2^7 3^5 \frac{g_2 g_3}{\Delta}\left(\xi+\frac{\alpha^2}{12}\right)\\
& = \frac{\alpha (\alpha^3-24)(\alpha^6-36\alpha^3+216)}{\alpha^3-27} (\alpha^2+12\xi)
\end{align*}
lies in $\textsf{K}_{(9)}$ implies that $L = \Sigma(\xi) \subset \textsf{K}_{(9)}$.  By the previous paragraph, the field $L$ has conductor $\mathfrak{m} = \wp_3'^2$ over $K$, and $\wp_3$ is unramified in $L/K$.  This shows that $L = \textsf{K}_{\wp_3'^2}$, since $\frac{\varphi_K(\wp_3'^2)}{2} = 3$.  This also shows that $L$ is the inertia field for the prime $\wp_3$ in $\textsf{K}_{(9)}/K$, since any subfield of $\textsf{K}_{(9)}/\Sigma$ not contained in $L$ must have a conductor which is divisible by $\wp_3$.  It remains to show that $k(x)$ is irreducible over $\Sigma$.  \medskip

This may be shown using the Newton polygon for the shifted polynomial
\begin{equation}
\label{eqn:6.4} k\left(x+\frac{\alpha}{3}+1\right) = x^3 -\frac{\alpha^2 +3\alpha +9}{3}x -\frac{(2\alpha + 3)(\alpha^2 + 3\alpha + 9)}{27},
\end{equation}
for a prime divisor $\mathfrak{p}$ of $\wp_3'$ in $\Sigma$.  I claim that the additive valuation $w_\mathfrak{p}$ of the last two coefficients is $2$.  For the coefficient of $x$, this follows from the above remarks, since
$$w_\mathfrak{p}\left(\frac{\alpha^2 +3\alpha +9}{3}\right) = 3-1 = 2.$$
For the constant term,
$$w_\mathfrak{p}\left(\frac{(2\alpha + 3)(\alpha^2 +3\alpha +9)}{27}\right) = w_\mathfrak{p}(2\alpha + 3) = w_\mathfrak{p}(9+2\gamma^3) = 2,$$
where $\alpha = 3+\gamma^3$, with $\gamma \cong \wp_3'$, by results of \cite{m1}.  It follows that the Newton polygon for the polynomial in \eqref{eqn:6.4} is the line segment joining the points $(0,2)$ and $(3,0)$, since $(1,2)$ and $(2,\infty)$ lie above this line segment.  The slope of this segment is $-2/3$, which implies the irreducibility of $k(x)$ over the completion $\Sigma_\mathfrak{p}$.  (See \cite[pp. 76-77]{vdw}.) \medskip

Now let
\begin{align*}
\tau_1 = \tau(\mathfrak{k}_1^*) & = -2^7 3^5 \frac{g_2 g_3}{\Delta}\left(\xi + \frac{\alpha^2}{12}\right)\\
& = \frac{\alpha(\alpha^3 - 24)(\alpha^6 - 36\alpha^3 + 216)(\alpha^2 + 12\xi)}{\alpha^3 - 27},\\
\tau_2 = \tau(\mathfrak{k}_2^*) & = -2^7 3^5 \frac{g_2 g_3}{\Delta}\left(\frac{1}{1-\xi} + \frac{\alpha^2}{12}\right)\\
& = \frac{\alpha(\alpha^3 - 24)(\alpha^6 - 36\alpha^3 + 216)(\alpha^2 \xi - \alpha^2 - 12)}{(\xi - 1)(\alpha^3 - 27)},
\end{align*}
\begin{align*}
\tau_3 = \tau(\mathfrak{k}_3^*) & = -2^7 3^5 \frac{g_2 g_3}{\Delta}\left(\frac{\xi-1}{\xi} + \frac{\alpha^2}{12}\right)\\
& = \frac{\alpha(\alpha^3 - 24)(\alpha^6 - 36\alpha^3 + 216)(\alpha^2 \xi + 12\xi - 12)}{ \xi(\alpha^3 - 27)},
\end{align*}
be the three $\tau$-invariants corresponding to the conjugates of $P$ and a fixed $j$-invariant $j(\mathfrak{k})$ (and ideal class $\mathfrak{k}$).  Since they form a complete set of conjugates over $\Sigma$, $P$ must be a $\wp_3^2$-division point or a $\wp_3'^2$-division point.  This is because $\frac{\varphi(\wp_3 \wp_3'^2)}{2} = \frac{\varphi(\wp_3' \wp_3^2)}{2} = 6$, so there would have to be $6$ conjugate invariants over $\Sigma$, if $P$ were a primitive $\mathfrak{m}'$-division point for $\mathfrak{m}' = \wp_3 \wp_3'^2$ or $\wp_3' \wp_3^2$.  For a similar reason, $P$ is not a primitive $(9)$-division point.  Furthermore, $\Sigma(\tau_i) = \textsf{K}_{\wp_3'^2}$ and $\textsf{K}_{\wp_3'^2}$ and $\textsf{K}_{\wp_3^2}$ are disjoint over $\Sigma$; thus, $P$ must be a $\wp_3'^2$-division point.  These invariants are roots of the polynomial
\begin{equation*}
T_\mathfrak{m}(X, \mathfrak{k}) = X^3 - c_1 X^2 + c_2 X - c_3 \in \Sigma[X],
\end{equation*}
where
\begin{align*}
c_1 & = \frac{3\alpha(\alpha^3 - 24)(\alpha^6 - 36\alpha^3 + 216)(\alpha^2 + 4\alpha + 12)}{\alpha^3 - 27},\\
c_2 & = \frac{3\alpha^3(\alpha^3 - 24)^2(\alpha^6 - 36\alpha^3 + 216)^2(\alpha^3 + 8\alpha^2 + 24\alpha + 48)}{(\alpha^3-27)^2},\\
c_3 & = \frac{\alpha^3(\alpha^3 - 24)^3(\alpha^6 - 36\alpha^3 + 216)^3(\alpha^6 + 12\alpha^5 + 36\alpha^4 + 144\alpha^3 - 1728)}{(\alpha^3 - 27)^3}.
\end{align*}
The $c_i$ are the sums of the products of the $\tau_j$ taken $i$ at a time.  The discriminant of $T_\mathfrak{m}(X, \mathfrak{k})$ is
$$\textrm{disc}(T_\mathfrak{m}(X,  \mathfrak{k})) = 12^6 \frac{\alpha^6(\alpha^3 - 24)^6(\alpha^6 - 36\alpha^3 + 216)^6}{(\alpha - 3)^6(\alpha^2 + 3\alpha + 9)^4}.$$
$T_\mathfrak{m}(X,  \mathfrak{k})$ is irreducible over $\Sigma$ by Hasse's results \cite{h1}.  Since its discriminant is a square, any of its roots generates a cyclic cubic extension, and since the $\tau_i \in \textsf{K}_{\wp_3'^2}$, we clearly have $\Sigma(\tau_i) = \textsf{K}_{\wp_3'^2}$.  Furthermore,
\begin{align}
\label{eqn:6.5} \frac{\tau_1-\tau_2}{\tau_2-\tau_3} &= -\xi, \ \ \frac{\tau_2-\tau_3}{\tau_3-\tau_1} = \frac{1}{\xi-1}, \ \ \frac{\tau_3-\tau_1}{\tau_1-\tau_2} = \frac{1-\xi}{\xi};\\
\label{eqn:6.6} \frac{\tau_1-\tau_3}{\tau_3-\tau_2} &= \xi-1, \ \ \frac{\tau_3-\tau_2}{\tau_2-\tau_1} = \frac{-1}{\xi}, \ \ \frac{\tau_2-\tau_1}{\tau_1-\tau_3} = \frac{\xi}{1-\xi}.
\end{align}
The values of these ratios are the negatives and the negative reciprocals of the roots of $k(x)$.  Suppose that $T_1(X) = T(X, \mathfrak{k}_1) = T(X, \mathfrak{k}_2) = T_2(X)$ for two different ideal classes $\mathfrak{k}_1, \mathfrak{k}_2$, corresponding to different conjugates $\alpha_1, \alpha_2$ of $\alpha$ over $K$ and corresponding roots $\xi_1, \xi_2$.  Then the negatives and negative reciprocals of the roots of $k_1(x) = x^3-(3+\alpha_1)x^2+\alpha_1 x+1$ must coincide with the corresponding expressions in the roots of $k_2(x) = x^3-(3+\alpha_2)x^2+\alpha_2 x+1$.  If, for example, $-\xi_1 = \frac{1}{\xi_2-1}$, then $\xi_1 = \frac{1}{1-\xi_2}$ is a conjugate of $\xi_2$ over $\Sigma$, whence it follows that $\alpha_1 = \alpha_2$.  The same holds if the ratios in \eqref{eqn:6.5} for $T_1(X)$ coincide with a permutation of the same ratios for $T_2(X)$.  The ratios in \eqref{eqn:6.6} are the reciprocals of the ratios in \eqref{eqn:6.5}, so that a similar statement holds if the ratios in \eqref{eqn:6.6} for the polynomials $T_i(X)$ are permutations of each other.  Now suppose that $-\xi_1 = \frac{\xi_2}{1-\xi_2}$.  Then $\frac{1}{\xi_1} = \frac{\xi_2-1}{\xi_2}$ is a conjugate of $\xi_2$, from which it would follow that $x^3 k_1(1/x) = x^3+\alpha_1 x^2-(3+\alpha_1)x+1$ coincides with $k_2(x) = x^3-(3+\alpha_2)x^2+\alpha_2 x+1$; hence, $\alpha_1 + \alpha_2 = -3$.  But this is impossible, since this would imply
$$0 = \alpha_1+\alpha_2+3 = (3+\gamma_1^3)+(3+\gamma_2^3) + 3 = 9 + \gamma_1^3+\gamma_2^3,$$
where $\gamma_1, \gamma_2 \in \Sigma$ and $(\gamma_1) = (\gamma_2) = \wp_3'$, implying that $\wp_3'^3 \mid 9$. \medskip

This shows that the ratios in \eqref{eqn:6.5} for $T_1(X)$ cannot coincide with the ratios in \eqref{eqn:6.6} for $T_2(X)$. Hence, we must have $\alpha_1 = \alpha_2$ and therefore $j(\mathfrak{k}_1) = j(\mathfrak{k}_2)$.  This proves that the polynomials $T_\mathfrak{m}(X, \mathfrak{k})$ are distinct for different ideal classes.

\begin{thm} If $\mathfrak{m} = \mathfrak{p}_1^2$, where $(3) = \mathfrak{p}_1 \mathfrak{p}_2$, then the ray class field $K_\mathfrak{m} = K(\tau_i)$ is generated over $K$ by a single $\tau$-invariant for the conductor $\mathfrak{m}$.
\label{thm:8}
\end{thm}

Note that the case $\mathfrak{m} = \mathfrak{p}^2 = (3)$ has been handled in Section \ref{sec:5}, and the case of any higher power of a first degree prime divisor $\mathfrak{p}$ of $3$ is taken care of by Sugawara's condition \eqref{eqn:3}. \medskip

\subsection{The case $\mathfrak{m} = \wp_3 \wp_3'^2$.}

Now since Sugawara's function $\Psi(\wp_3 \wp_3'^2) = 3$, we must also consider the ideal $\mathfrak{m} = \wp_3 \wp_3'^2$, for which $\varphi(\mathfrak{m}) = 12$.  Since $\textsf{K}_{\wp_3'^2} \subset \textsf{K}_{\mathfrak{m}}$ and $[\textsf{K}_{\mathfrak{m}}:\textsf{K}_{\wp_3'^2}] = 2$ , it is clear that $\textsf{K}_{\mathfrak{m}} = \textsf{K}_{\wp_3'^2} \textsf{K}_{(3)} = \textsf{K}_{\wp_3'^2}(\omega)$.  \medskip

We will construct an $\mathfrak{m}$-division point on the curve $E_3(\alpha)$, with $\alpha$ as above.  First note that we may take the $Y$-coordinate of the point $P = (\xi, \eta)$ to be either $\eta = \xi-\xi^2$ or $\eta = \frac{\xi^2}{\xi-1}$.  Thus $\eta \in \textsf{K}_{\wp_3'^2}$.  Now consider the sum of points
$$Q_{1i} = P + \left(\frac{-3\beta}{\alpha(\beta-3)}, \frac{\beta-3\omega^i}{\beta-3}\right) = (\xi, \xi-\xi^2)+\left(\frac{-3\beta}{\alpha(\beta-3)}, \frac{\beta-3\omega^i}{\beta-3}\right)$$
on $E_3(\alpha)$.  It is not hard to see that $Q_{12} \not = \pm Q_{11}$.  If, for example, $Q_{12} = -Q_{11}$, then 
$$2P = -\left(\frac{-3\beta}{\alpha(\beta-3)}, \frac{\beta-3\omega}{\beta-3}\right) -\left(\frac{-3\beta}{\alpha(\beta-3)}, \frac{\beta-3\omega^2}{\beta-3}\right);$$
but the right side is defined over $\Sigma$, while the left side is only defined over $\Sigma(\xi)$.  Hence,  $Q_{11}$ is defined over $\textsf{K}_\mathfrak{m}$ but not over the intermediate fields of $\textsf{K}_\mathfrak{m}/\Sigma$.  \medskip

The conjugates of $Q_{1i}$ over $\Sigma(\omega)$ are the points $Q_{2i}, Q_{3i}$, which are obtained by replacing $\xi$ in the above sum by its conjugates $\frac{1}{1-\xi}, \frac{\xi-1}{\xi}$, respectively.  The corresponding $\tau$-invariants are
$$\tau_{ji} = \lambda\left(X(Q_{ji})+\frac{\alpha^2}{12}\right), \ 1\le j \le 3, \ i \in \{1,2\}.$$
It is clear that $\tau_{11}$ generates $\textsf{K}_{\mathfrak{m}}$ over $\Sigma$ (see the proof below), so that $Q_{11}$ is a primitive $\mathfrak{m}$-division point.  We will handle this case by proving the following.

\newtheorem{prop}{Proposition}

\begin{prop}
The trace of $X(Q_{11})$ to the Hilbert class field $\Sigma$ is
$$\textrm{Tr}_{\Sigma}(X(Q_{11})) = -3 - \alpha.$$
\label{prop:4}
\end{prop}
\begin{proof}
Setting $(x_1,y_1) = (\xi, \xi-\xi^2)$ and $(x_2,y_2) = (\frac{-3\beta}{\alpha(\beta-3)},\frac{\beta - 3\omega}{\beta-3})$ and using the formula
$$X(Q_{11}) = \left(\frac{y_2 - y_1}{x_2 - x_1}\right)^2 + \alpha \frac{y_2 - y_1}{x_2 - x_1} - x_1 - x_2,$$
(see \cite[p. 54]{si1}) we find that
\begin{equation*}
X(Q_{11}) = \frac{1}{(\xi \alpha (\beta - 3) + 3\beta)^2 \alpha (\beta - 3)} \times P(\alpha,\beta,\xi,\omega),
\end{equation*}
where $P$ is a polynomial in the quantities $\alpha,\beta,\xi,\omega$.  We compute the inverse of $\textsf{A} = (\xi \alpha (\beta - 3) + 3\beta)$ to be
\begin{align*}
\textsf{A}^{-1} &= \frac{\beta(\beta - 3)(\alpha-3)}{81\alpha^2} \xi^2\\
& - \frac{\beta(\alpha-3)(\alpha^2\beta - 3\alpha^2 + 3\alpha \beta - 9\alpha + 3\beta)}{81\alpha^3} \xi\\
& + \frac{\beta(\alpha-3)(\alpha^3 \beta^2 - 6 \alpha^3 \beta + 3 \alpha^2 \beta^2 + 9 \alpha^3 - 9 \alpha^2 \beta + 9 \alpha \beta^2 - 27 \alpha \beta + 9\beta^2)}{81(\beta - 3)\alpha^4}.
\end{align*}
Using this in the expression for $X(Q_{11})$, the trace of $X(Q_{11})$ to $\Sigma(\xi) = \textsf{K}_{\wp_3'^2}$ is
\begin{align*}
\textrm{Tr}_{\Sigma(\xi)}(X(Q_{11})) &= \textsf{A}^{-2} \frac{1}{\alpha(\beta-3)}[P(\alpha,\beta,\xi,\omega)+P(\alpha,\beta,\xi,\omega^2)]\\
&= -A_1\xi^2 - A_2 \xi -A_3,
\end{align*}
where
\begin{align*}
A_1 &= -\frac{\beta(\beta-3)(\alpha-3)}{9\alpha};\\
A_2 &= \frac{\beta(\alpha-3)[(\beta-3)\alpha^2+3(\beta-2)\alpha+3\beta]}{9\alpha^2};\\
A_3 &= \frac{\beta(\alpha-3)[(\beta-3)^2 \alpha^3 +3(2\beta^2-3\beta-9)\alpha^2+9\beta(2\beta-3)\alpha+27\beta^2]}{81\alpha^3}.
\end{align*}

Now, using that $\textrm{Tr}_{\Sigma}(\xi) = 3+\alpha$ and $\textrm{Tr}_\Sigma(\xi^2) = (\alpha+3)^2 -2\alpha = \alpha^2 + 4\alpha+9$, we compute that
\begin{align*}
\textrm{Tr}_\Sigma(X(Q_{11})) &= -(A_1(\alpha^2+4\alpha+9)+A_2(\alpha+3)+3A_3)\\
&=-\frac{\beta^3(\alpha^3 - 27)(\alpha + 3)}{27\alpha^3} = -(\alpha+3),
\end{align*}
as claimed.
\end{proof}

Let $\psi \in \textrm{Gal}(\textsf{K}_\frak{m}/K)$ be an automorphism which permutes the $\tau$-invariants for $\mathfrak{m}$ and a given ideal class $\mathfrak{k}$.  Then $\psi$ permutes the $3$ conjugates (over $\Sigma$) of the trace of $\tau_{11}$ to $\Sigma(\xi) = \textsf{K}_{\wp_3'^2}$.  Set
\begin{align*}
t_1 &= -A_1\xi^2 - A_2 \xi -A_3,\\
t_2 &= -A_1\frac{1}{(1-\xi)^2} - A_2 \frac{1}{1-\xi} -A_3,\\
&= \frac{\beta(\alpha - 3)(\alpha + \beta)}{3\alpha^2}\xi^2 - \frac{\beta(\alpha - 3)(3\alpha^2 + 2\alpha \beta + 9\alpha + 6\beta)}{9\alpha^2}\xi\\
& \ \  - \frac{\beta(\alpha - 3)}{81\alpha^3} [\alpha^3 \beta^2 + 3\alpha^3 \beta + 6\alpha^2 \beta^2 - 18\alpha^3 + 18\alpha \beta^2 + 27 \alpha \beta + 27 \beta^2]\\
t_3 &= -A_1\left(\frac{\xi-1}{\xi}\right)^2 - A_2 \frac{\xi-1}{\xi} -A_3\\
&= -\frac{\beta^2(\alpha^2 - 9)}{9\alpha^2}\xi^2 + \frac{\beta(\alpha - 3)(\alpha^2 \beta + 5\alpha \beta + 3\alpha + 9\beta)}{9\alpha^2} \xi\\
& \ \ - \frac{\beta(\alpha - 3)}{81\alpha^3}[\alpha^3 \beta^2 + 3\alpha^3 \beta + 6\alpha^2 \beta^2 + 9 \alpha^3 + 9 \alpha^2 \beta + 18\alpha \beta^2 + 27\alpha^2 + 27\beta^2];
\end{align*}
with notation as in the proof of Proposition \ref{prop:4}, and where we have used
$$\frac{1}{\xi-1} =  \xi^2 - (\alpha + 2)\xi -2, \ \ \frac{1}{\xi} = -\xi^2 + (\alpha + 3)\xi - \alpha$$
to express the $t_i$ in terms of the basis $\{1,\xi,\xi^2\}$ for $\Sigma(\xi)/\Sigma$.  (Recall from Section 6.1 that $\xi$ is a root of the polynomial $X^3-(\alpha+3)X^2+\alpha X+1$.)  \medskip

The traces to $\Sigma(\xi)$ of the $\tau$-invariants for $\mathfrak{m}$ (and a given ideal class $\mathfrak{k}$) are the quantities $\lambda(t_i+\frac{\alpha^2}{6})$.  Hence, the minimal polynomial of these traces is the polynomial
$$F(X) = X^3+\lambda c_1 X^2+\lambda^2 c_2 X+ \lambda^3 c_3, \ \ \lambda = -2^7\cdot 3^5 \frac{g_2 g_3}{\Delta},$$
with $g_2, g_3, \Delta$ as in Section 5 and
\begin{align*}
c_1 &=  -\sum_{i=1}^3{\left(t_i+\frac{\alpha^2}{6}\right)} = -\frac{1}{2}(\alpha^2-2\alpha-6);\\
c_2 &= \sum_{1 \le i < j \le 3}{(t_i+\frac{\alpha^2}{6})(t_j+\frac{\alpha^2}{6})}\\
&= -\frac{\alpha^4 - 3\alpha^3 - 27\alpha + 81}{27\alpha^2}\beta^2 - \frac{\alpha^3 - 27}{9\alpha}\beta - \frac{\alpha^3}{3} - \frac{\alpha^2}{3} + 6 + \frac{\alpha^4}{12}.
\end{align*}
Since the automorphism $\psi$ permutes the invariants $\tau(\mathfrak{k}^*)$ for $\mathfrak{m} = \wp_3 \wp_3'^2$, it permutes their traces to the unique cubic subfield $\Sigma(\xi)$ of $\textsf{K}_{\wp_3 \wp_3'^2}/\Sigma$, and therefore also permutes the roots of the cubic polynomial $F(X)$.  Hence, the quantities $\lambda c_1$ and $\lambda^2 c_2$ are invariant under $\psi$.  We must show that $\psi$ fixes the field $\Sigma = \mathbb{Q}(\alpha)$.  This is complicated by the presence of the factors $\lambda$ and $\lambda^2$, and also because $c_2$ is not easily expressible in terms of $\alpha$ alone.  First, we set
$$L(\alpha) = -\frac{1}{6} \lambda c_1 = \frac{\alpha(\alpha^3 - 24)(\alpha^6 - 36\alpha^3 + 216)(\alpha^2 - 2\alpha - 6)}{\alpha^3 - 27}.$$
Letting $\alpha' = \alpha^\psi$ we have $L(\alpha) = L(\alpha')$. Hence, with $x = \alpha, y = \alpha'$, the following quantity is zero:
\begin{equation}
L(x)-L(y) = \frac{(x-y)}{(x^3-27)(y^3-27)} B(x,y),
\label{eqn:6.7}
\end{equation}
\begin{align*}
&B(x,y) = (x^3 - 27)y^{11} + (x^4 - 2x^3 - 27x + 54)y^{10}\\
& + (x^5 - 2x^4 - 6x^3 - 27x^2 + 54x + 162)y^9\\
& + (x^6 - 2x^5 - 6x^4 - 87x^3 + 54x^2 + 162x + 1620)y^8\\
& + (x^7 - 2x^6 - 6x^5 - 87x^4 + 174x^3 + 162x^2 + 1620x - 3240)y^7\\
& + (x^8 - 2x^7 - 6x^6 - 87x^5 + 174x^4 + 522x^3 + 1620x^2 - 3240x - 9720)y^6\\
& + (x^9 - 2x^8 - 6x^7 - 87x^6 + 174x^5 + 522x^4 + 2700x^3 - 3240x^2 - 9720x\\
&  - 29160)y^5\\
& + (x^{10} - 2x^9 - 6x^8 - 87x^7 + 174x^6 + 522x^5 + 2700x^4 - 5400x^3 - 9720x^2 \\
& \ - 29160x + 58320)y^4\\
& + (x^{11} - 2x^{10} - 6x^9 - 87x^8 + 174x^7 + 522x^6 + 2700x^5 - 5400x^4 - 16200x^3 \\
& \ - 29160x^2 + 58320x + 174960)y^3\\
& + (-27x^9 + 54x^8 + 162x^7 + 1620x^6 - 3240x^5 - 9720x^4 - 29160x^3  \\
& \ + 47952x^2 + 143856x + 139968)y^2
\end{align*}
\begin{align*}
& + (-27x^{10} + 54x^9 + 162x^8 + 1620x^7 - 3240x^6 - 9720x^5 - 29160x^4  \\
& \ + 58320x^3 + 143856x^2 + 139968x - 279936)y\\
& - 27x^{11} + 54x^{10} + 162x^9 + 1620x^8 - 3240x^7 - 9720x^6 - 29160x^5  \\
& \ + 58320x^4 + 174960x^3 + 139968x^2 - 279936x - 839808.
\end{align*}

Now we let $c_2 = c_2(\alpha,\beta)$ and consider the product
\begin{align*}
M(\alpha) &= c_2(\alpha,\beta)c_2(\alpha,\omega \beta) c_2(\alpha,\omega^2\beta)\\
&= \frac{\alpha}{1728}(\alpha^{11} - 12\alpha^{10} + 36\alpha^9 + 32\alpha^8 + 24\alpha^7- 1584\alpha^6 + 1152\alpha^5\\
& \  + 8208\alpha^4 + 4752\alpha^3 - 34560\alpha^2  - 62208\alpha + 139968).
\end{align*}
To calculate this product we have made use of the relations $27\alpha^3+27\beta^3 = \alpha^3 \beta^3$ and $\omega^2+\omega+1 = 0$.  Furthermore,
if $\psi$ is the automorphism considered above, it is clear that $\lambda^6 M(\alpha) = \lambda^{6\psi} M(\alpha')$, since we know that $\lambda^2 c_2 = \lambda^{2\psi} c_2'$, and either $\omega^\psi = \omega$ or $\omega^2$.  Now we consider the quotient
\begin{align*}
G(\alpha) &= \frac{\lambda^6 M(\alpha)}{\lambda^6 c_1^6} = \frac{M(\alpha)}{c_1^6} \\
&= \frac{\alpha}{27(\alpha^2-2\alpha-6)^6} (\alpha^{11} - 12\alpha^{10} + 36\alpha^9 + 32\alpha^8 + 24\alpha^7- 1584\alpha^6 \\
& \  + 1152\alpha^5 + 8208\alpha^4 + 4752\alpha^3 - 34560\alpha^2  - 62208\alpha + 139968).
\end{align*}
Since the numerator and denominator are fixed by $\psi$, we know that $G(\alpha') = G(\alpha)$.  With $x=\alpha, y = \alpha'$ as above, we make use of the
following difference, which we know to be zero:
\begin{equation}
27G(x)-27G(y) = -\frac{4(x-y)}{(x^2 - 2x - 6)^6(y^2 - 2y - 6)^6} H(x,y),
\label{eqn:6.8}
\end{equation}
with
\begin{align*}
H(&x,y) = (3x^{10} - 42x^9 + 171x^8 + 12x^7 - 1168x^6 - 396x^5 + 7128x^4 \\
& + 2160x^3 - 23328x^2 + 11664x - 11664)y^{11}\\
& + (3x^{11} - 78x^{10} + 675x^9 - 2040x^8 - 1312x^7 + 13620x^6 + 11880x^5\\
&  - 83376x^4 - 49248x^3 + 291600x^2 - 151632x + 139968)y^{10}\\
& + (-42x^{11} + 675x^{10} - 3648x^9 + 4772x^8 + 18804x^7 - 33624x^6 \\
& - 122256x^5 + 193104x^4 + 473040x^3 - 804816x^2 + 139968x - 419904)y^9\\
& + (171x^{11} - 2040x^{10} + 4772x^9 + 25284x^8 - 99768x^7 - 111248x^6 + 525168x^5\\
& + 900720x^4 - 2187216x^3 - 3219264x^2 + 5832000x - 373248)y^8\\
\end{align*}
\begin{align*}
& + (12x^{11} - 1312x^{10} + 18804x^9 - 99768x^8 + 159904x^7 + 300144x^6 - 512352x^5\\
& - 2828736x^4 + 2742336x^3 + 15909696x^2 - 24027840x - 279936)y^7\\
& + (-1168x^{11} + 13620x^{10} - 33624x^9 - 111248x^8 + 300144x^7 + 1323936x^6\\
&  - 2299968x^5 - 8605440x^4 + 12902976x^3 + 13670208x^2 - 20435328x\\
&  + 18475776)y^6\\
 & + (-396x^{11} + 11880x^{10} - 122256x^9 + 525168x^8 - 512352x^7 - 2299968x^6\\
 & + 525312x^5 + 26664768x^4 - 24207552x^3 - 119968128x^2 + 195395328x\\
 & - 13436928)y^5\\
 & + (7128x^{11} - 83376x^{10} + 193104x^9 + 900720x^8 - 2828736x^7 - 8605440x^6\\
 & + 26664768x^5 + 36180864x^4 - 115924608x^3 - 20715264x^2 + 137728512x\\
 & - 95738112)y^4\\
 & + (2160x^{11} - 49248x^{10} + 473040x^9 - 2187216x^8 + 2742336x^7 + 12902976x^6\\
 & - 24207552x^5 - 115924608x^4 + 235892736x^3 + 470292480x^2 - 1038002688x\\
 & - 55427328)y^3\\
 & + (-23328x^{11} + 291600x^{10} - 804816x^9 - 3219264x^8 + 15909696x^7\\
 &  + 13670208x^6 - 119968128x^5 - 20715264x^4 + 470292480x^3 - 97417728x^2 \\
 &  - 760866048x + 403107840)y^2\\
 & + (11664x^{11} - 151632x^{10} + 139968x^9 + 5832000x^8 - 24027840x^7 \\
 & - 20435328x^6 + 195395328x^5 + 137728512x^4 - 1038002688x^3 - 760866048x^2 \\
 & + 2942687232x + 725594112)y\\
 & - 11664x^{11} + 139968x^{10} - 419904x^9 - 373248x^8 - 279936x^7 + 18475776x^6\\
 & - 13436928x^5 - 95738112x^4 - 55427328x^3 + 403107840x^2 + 725594112x\\
 &- 1632586752.
\end{align*}

Now assume that $B(x,y) = H(x,y) = 0$ in \eqref{eqn:6.7} and \eqref{eqn:6.8}.  Reducing $H(x,y)$ mod $\wp_3$ yields that
$$0 = H(x,y) \equiv 2 x^6 y^6(x - y)^2(y + x + 2)^3 \ (\textrm{mod} \ \wp_3).$$
Since $x = \alpha, y = \alpha'$, where $(\alpha,\wp_3) = (\alpha',\wp_3) = 1$, $x^6 y^6$ is relatively prime to $\wp_3$.  Furthermore,
$(x-y,\wp_3) = 1$, since $x$ and $y$ are conjugate over $K$ and the discriminant of the minimal polynomial $m_d(X)$ of $\alpha$ over $K$ is not divisible
by $\wp_3$.  (See \cite[p. 880]{m1}.  In particular, the powers of $\alpha$ are a $\wp_3$-integral basis for $\Sigma/K$.)  It follows that we must have
$$y+x+2 \equiv 0 \ (\textrm{mod} \ \wp_3) \ \Rightarrow \ y \equiv 1-x  \ (\textrm{mod} \ \wp_3).$$
Putting this into the polynomial $B(x,y)$ yields that
\begin{align}
\notag 0 &= B(x,y) \equiv B(x,1-x)\\
\label{eqn:6.9} &\equiv 2x^3(x + 2)^3(x^4 + x^2 + 2x + 1)(x^4 + 2x^3 + x^2 + x + 2) \ (\textrm{mod} \ \wp_3).
\end{align}
By the theory in \cite{m1}, the minimal polynomial $p_d(x)$ of $\alpha$ over $\mathbb{Q}$ splits (mod $3$) as $x^{h(-d)}$ times a 
product of irreducible polynomials of degree $n$, where $n$ is the degree of the prime divisors of $\wp_3$ in $\Sigma$ over $K$ and 
$h(-d)$ is the class number of $K$.  The integer $n$ is
also the minimal period of $\alpha$ with respect to the algebraic function defined by the polynomial
$$g(x,y) = (y^2+3y+9)x^3-(y+6)^3.$$
We know that $(x, \wp_3) = 1$.  Also, if $x +2$ occurs as a factor of $p_d(x)$ (mod $3$), then $n = 1$ and $p_d(x) = x^2+2x+12$, since 
the minimal polynomials of fixed points for $g(x,y)$ divide
$$g(x,x) = (x-3)(x^2+4x+6)(x^2+2x+12).$$
In that case, $d=11$, and since $h(-11) = 1$, we may disregard this case, as Sugawara's conjecture is obviously true for $K = \mathbb{Q}(\sqrt{-11})$.
Furthermore, each of the quartics in \eqref{eqn:6.9} divides (modulo $3$) a unique factor $p_d(x)$ of the nested resultant
$$\textrm{Res}_{x_3}(\textrm{Res}_{x_2}(\textrm{Res}_{x_1}(g(x,x_1),g(x_1,x_2)),g(x_2,x_3)),g(x_3,x)),$$
whose roots are the points of period dividing $n = 4$ for $g(x,y)$.  The factor $x^4 + x^2 + 2x + 1$ divides
\begin{align*}
p_{203}(x) &= x^8 + 216x^7 + 19720x^6 - 105856x^5 + 677824x^4 - 2012736x^3 \\
& \ \ + 2877120x^2 - 18399744x + 40144896
\end{align*}
modulo $3$.  Hence, $x$ and $y$ are roots of $p_{203}(X)$.  Now we compute
$$\textrm{Res}_y(B(x,y),p_{203}(y)) \equiv 8(x + 9)(x^2 + 5x + 11)(x^2 + 9x + 10)(x^2 + 7x + 1)$$
$(\textrm{modd} \ p_{203}(x),13)$, i.e., modulo $p_{203}(x)$ over $\mathbb{F}_{13}$; while
$$p_{203}(x) \equiv (x^2 + 5x + 10)(x^2 + 2x + 12)(x^2 + 4x + 6)(x^2 + 10x + 5) \ (\textrm{mod} \ 13),$$
which shows that the resultant is certainly nonzero.  It follows that $B(x,y) \neq 0$ in this case.
\medskip

Finally, the quartic $x^4 + 2x^3 + x^2 + x + 2$ divides
\begin{align*}
&p_{2^6 \cdot 5}(X) = X^{16} - 800X^{15} + 273910X^{14} - 2916016X^{13} + 9685420X^{12}\\
& - 91592020X^{11} + 553552696X^{10} - 280244240X^9 + 10513962550X^8 \\
& - 15387943056X^7 + 24842673360X^6 - 396367171200X^5  \\
& - 171780654744X^4 - 1652555947680X^3 + 4704933982320X^2\\
& + 6541367738400X + 26983141920900
\end{align*}
modulo $3$.  However, a root of this polynomial generates the ring class field $\Omega_{4}$ of conductor $f = 4$ over $K = \mathbb{Q}(\sqrt{-5})$, 
so that this case cannot occur.  (This can be verified by examining the possible discriminants $-d$ for which $4\cdot 3^4 = a^2+db^2$ 
or $3^4 = a^2 + \frac{d}{4}b^2$, and showing that the polynomial $p_{320}(X)$ splits into linear factors modulo primes of the form $a^2+db^2$
or $a^2 + \frac{d}{4}b^2$, for only one of the five possible discriminants with $h(-d) = 8$, namely $-d = -320$.)   \medskip

Hence, at least one of $B(x,y)$ or $H(x,y)$ is nonzero, and this implies that $\alpha = \alpha'$ by
\eqref{eqn:6.7} and \eqref{eqn:6.8}.  Hence $j(\mathfrak{k}) = j(\mathfrak{k}')$.  This implies Sugawara's conjecture for $\mathfrak{m} = \wp_3'^2 \wp_3$.

\begin{thm} If $\mathfrak{m} = \mathfrak{p}_1^2 \mathfrak{p}_2$, where $(3) = \mathfrak{p}_1 \mathfrak{p}_2$, then the ray class field $K_\mathfrak{m} = K(\tau_i)$ is generated over $K$ by a single $\tau$-invariant for the conductor $\mathfrak{m}$.
\label{thm:8a}
\end{thm}

\section{The case $\mathfrak{m} = (2)\wp_3 \wp_3'$.}
\label{sec:7}

On the curve $E_3(\alpha)$, the doubling formula is
$$X(2P) = \frac{x(x^3 - \alpha x - 2)}{4x^3 + (\alpha x + 1)^2}.$$
Then $P = (x,y)$ on $E_3$ satisfies $2P = \pm (0,0)$ and $P \neq \pm (0,0)$ if and only if $x^3 - \alpha x - 2 = 0$.  If this condition holds, then $P$ has order $6$ on $E_3$.  We have
$$\alpha  = \frac{x^3-2}{x} = \frac{(-x)^3+2}{-x},$$
and it follows from the result of \cite[Prop. 13]{am} that
\begin{equation}
\label{eqn:7.1} x = -2c(w/3), \ \frac{1}{c(w/6)} \ \ \textrm{or} \ \ \frac{1}{c_1(w/6)},
\end{equation}
where $c(\tau)$ is Ramanujan's cubic continued fraction and $c_1(\tau) = c(\tau + \frac{3}{2})$; and $w/3$ is the basis quotient of a suitable integral ideal.  Each of these values lies in the ring class field $\Omega_2 = \textsf{K}_2$, by \cite[pp. 20, 27]{am}.  \medskip

Given the factorization
$$Y^2 + \alpha xY + Y - x^3 = Y^2 +(x^3-1)Y-x^3 = (Y-1)(Y+x^3) =0$$
we set $P = (x,1)$ and $-P = (x,-x^3)$.  Assuming $2P = (0,0)$, this yields that 
\begin{equation*}
3P = 2P+P = (0,0) + (x,1) = \left(\left(\frac{1}{x}\right)^2+\alpha\left(\frac{1}{x}\right)-x,y_1\right) = \left(\frac{-1}{x^2},y_1\right).
\end{equation*}
The alternative would be that $2P = (0,-1)$, in which case
\begin{equation*}
3P = 2P+P = (0,-1) + (x,1) = \left(\left(\frac{2}{x}\right)^2+\alpha\left(\frac{2}{x}\right)-x,y_2\right) = (x,y_2);
\end{equation*}
and this would imply that $3P = -P$, which is false.  Hence $3P = \left(\frac{-1}{x^2},\frac{-1}{x^3}\right)$ has order $2$, so that
$$\bar{\tau} = -2^7 3^5 \frac{g_2 g_3}{\Delta}\left(\frac{-1}{x^2}+\frac{\alpha^2}{12}\right)$$
is the $\tau$-invariant corresponding to $3P$.  By Theorem \ref{thm:1} and the fact that the $\tau$-invariants are independent of the choice of the elliptic curve, we know that $\bar{\tau}$ generates $\textsf{K}_2/\Sigma$ and therefore so does the root $x$, since $\alpha, g_2, g_3, \Delta$ lie in $\Sigma$ and $x \in \textsf{K}_2$.  Hence, $t(X) = X^3 - \alpha X - 2$ is irreducible over $\Sigma$.  This also follows from \cite[Thm. 1]{am}.  \medskip

Now assume $\mathfrak{m} = (6) = (2) \wp_3 \wp_3'$, where $2$ is inert in $K$.  Consider the point of order $6$:
\begin{align*}
Q &= (x,1)+\left(\frac{-3\beta}{\alpha(\beta-3)},\frac{\beta-3\omega}{\beta-3}\right)\\
& = (X(Q),Y(Q)),
\end{align*}
where
\begin{equation*}
X(Q) = \frac{3\beta[-\alpha^2(\beta^3-27)x^2-3\alpha \beta(2\beta+3\omega^2)(\beta-3\omega)x+9\beta^2(-\beta+9\omega+3)]}{\alpha(\beta^2+3\beta+9)(\alpha \beta x-3\alpha x+3\beta)^2},
\end{equation*}
on reducing the numerator modulo $x^3-\alpha x -2$ (with respect to $x$) and $\alpha^3 \beta^3-27\alpha^3-27\beta^3$ (with respect to $\alpha$), and cancelling $\beta-3$.  Using that the reciprocal of $\alpha \beta x-3\alpha x+3\beta$ is
$$\rho = \frac{(\beta^2 + 3\beta + 9)(\beta-3)^2 \alpha^2}{729\beta^3}x^2 - \frac{(\beta^3 - 27)\alpha}{243\beta^2}x - \frac{(2\beta^2 - 12\beta - 9)}{81\beta},$$
this yields that
\begin{align*}
X(Q) &= \rho^2 \frac{3\beta[-\alpha^2(\beta^3-27)x^2-3\alpha \beta(2\beta+3\omega^2)(\beta-3\omega)x+9\beta^2(-\beta+9\omega+3)]}{\alpha(\beta^2+3\beta+9)}\\
& = -\frac{(\beta^2\omega + \beta^2 - 3\beta - 9\omega)\alpha}{9\beta}x^2 + \frac{(\beta \omega + \beta - 3)}{3}x\\
& \ \ \ \ + \frac{\beta(2\beta^2 \omega + 2\beta^2 + 6\beta \omega- 3\beta - 9\omega - 9)}{\alpha(\beta^2 + 3\beta + 9)}\\
& = \left(-\frac{(\beta^2 - 9)\alpha}{9\beta}x^2 + \frac{\beta}{3}x + \frac{\beta(2\beta^2 + 6\beta - 9)}{\alpha(\beta^2 + 3\beta + 9)}\right)\omega\\
& \ \ \ \ -\frac{(\beta - 3)\alpha}{9}x^2 + \frac{(\beta - 3)}{3}x + \frac{\beta(2\beta^2 - 3\beta - 9)}{\alpha(\beta^2 + 3\beta + 9)}.
\end{align*}
Since $\{1,x,x^2\}$ is a basis of $\textsf{K}_2/\Sigma$ and $\{1,\omega\}$ is a basis of $\textsf{K}_3/\Sigma$, it is clear that $\{1,x,x^2,\omega,\omega x,\omega x^2\}$ is a basis of $\textsf{K}_6/\Sigma$ ($\textsf{K}_2 \textsf{K}_3 = \textsf{K}_6$ because the degrees match).  Hence, the above representation  shows that $X(Q)$ does not lie in any of the subfields $\textsf{K}_2, \textsf{K}_3$ or $\Sigma$.  The trace of $X(Q)$ to $\textsf{K}_2$ is
\begin{align*}
\textrm{Tr}_{\textsf{K}_6/\textsf{K}_2}(X(Q)) &=  -\left(-\frac{(\beta^2 - 9)\alpha}{9\beta}x^2 + \frac{\beta}{3}x + \frac{\beta(2\beta^2 + 6\beta - 9)}{\alpha(\beta^2 + 3\beta + 9)}\right)\\
& \ \ +2\left(-\frac{(\beta^2 - 3\beta)\alpha}{9\beta}x^2 + \frac{(\beta - 3)}{3}x + \frac{\beta(2\beta^2 - 3\beta - 9)}{\alpha(\beta^2 + 3\beta + 9)}\right)\\
& = -\frac{(\beta - 3)^2\alpha}{9\beta} x^2+\frac{(\beta - 6)}{3}x+\frac{\beta(2\beta^2 - 12\beta - 9)}{(\beta^2 + 3\beta + 9)\alpha}.
\end{align*}
From this expression we compute the trace of $X(Q)$ to $\Sigma$ to be
\begin{align*}
\textrm{Tr}_{\textsf{K}_6/\Sigma}(X(Q)) &= -\frac{(\beta - 3)^2\alpha}{9\beta} \textrm{Tr}_{\textsf{K}_2/\Sigma}(x^2)+\frac{(\beta - 6)}{3}\textrm{Tr}_{\textsf{K}_2/\Sigma}(x)\\
& \ \ \ +3\frac{\beta(2\beta^2 - 12\beta - 9)}{(\beta^2 + 3\beta + 9)\alpha}\\
& = -\frac{(\beta - 3)^2\alpha}{9\beta}(2\alpha)+ 0 + 3\frac{\beta(2\beta^2 - 12\beta - 9)}{(\beta^2 + 3\beta + 9)\alpha}\\
& = \frac{-9\beta(2\beta+3)}{\alpha(\beta^2 + 3\beta + 9)}.
\end{align*}
Now we have
\begin{align*}
\textrm{Tr}_{\textsf{K}_6/\Sigma}&\left(X(Q)+\frac{\alpha^2}{12}\right) = \frac{-9\beta(2\beta+3)}{\alpha(\beta^2 + 3\beta + 9)}+\frac{\alpha^2}{2}\\
& = \frac{(\alpha^3 \beta^2 + 3 \alpha^3 \beta + 9 \alpha^3 - 36 \beta^2 - 54 \beta)}{2\alpha(\beta^2 + 3\beta + 9)}\\
& = -\frac{9\beta(\beta^2 - 6\beta - 18)}{2\alpha(\beta^2 + 3\beta + 9)(\beta - 3)}.
\end{align*}
It follows from this that
\begin{align*}
\textrm{Tr}&_{\textsf{K}_6/\Sigma}(\tau(\mathfrak{k}^*)) = -2^7 3^5 \frac{g_2 g_3}{\Delta}\textrm{Tr}_{\textsf{K}_6/\Sigma}\left(X(Q)+\frac{\alpha^2}{12}\right)\\
& = -54\frac{(\alpha^3 - 24)(\alpha^6 - 36\alpha^3 + 216)\beta(\beta^2 - 6\beta - 18)}{(\beta^3 - 27)(\alpha^3 - 27)}\\
& = \frac{6(\beta^2 - 6\beta - 18)^2\beta(\beta^4 + 6\beta^3 + 54\beta^2 - 108\beta + 324)(\beta^3 + 216)}{(\beta^3 - 27)^3}.
\end{align*}

Now assume that the ray class invariants for two ray classes $\mathfrak{k}_1^*, \mathfrak{k}_2^*$ are the same.  Since the six invariants for $\mathfrak{m}$ are conjugate over $\Sigma$, $\tau(\mathfrak{k}_1^*)$ and $\tau(\mathfrak{k}_2^*)$ are conjugate and their traces to $\Sigma$ must be equal.  Set
$$f(x) = \frac{(x^2 - 6x - 18)^2x(x^4 + 6x^3 + 54x^2 - 108x + 324)(x^3 + 216)}{(x^3 - 27)^3}.$$
Thus, $f(x) = f(y)$, for $x = \beta_1, y = \beta_2$, implies that
\begin{align}
\notag 0 & = (x^3-27)^3(y^3-27)^3 (f(x)-f(y))\\
\label{eqn:7.2} & = (-y + x)(xy - 3x - 3y - 18)(x^{10}y^8 + x^9 y^9 + x^8 y^{10}+3q(x,y)),
\end{align}
where $q(x,y) \in \mathbb{Z}[x,y]$ and the third factor is irreducible over $\mathbb{Z}$.  If the third factor is $0$, we reduce it modulo $\wp_3'$ and find that
$$x^{10}y^8 + x^9 y^9 + x^8 y^{10}+3q(x,y) \equiv x^8y^8(x^2+xy+y^2) \equiv x^8y^8(x-y)^2 \ \textrm{mod} \ \wp_3'.$$
Since the left side of this congruence is $0$, we have $\wp_3' \mid (x-y)$ in $\Sigma$, since $(xy,\wp_3') = 1$ and $\wp_3'$ is unramified in $\Sigma/K$.  But this means that $\wp_3'$ divides $\beta_1-\beta_2$, which is impossible, since the discriminant of the minimal polynomial of $\beta$ over $K$ is not divisible by $\wp_3'$, by \cite[p.880, Eq. (4.27)]{m1}.
(Apply the automorphism switching $\alpha$ and $\beta$ and $\wp_3$ and $\wp_3'$ in that equation.) Hence, the third factor in \eqref{eqn:7.2} is not $0$.  The second factor cannot be $0$, either, since this would imply that $\beta_2 = \frac{3\beta_1+18}{\beta_1-3} = \alpha_1^{\sigma_{\wp_3}}$, where $\sigma_{\wp_3} = \left(\frac{\Sigma/K}{\wp_3}\right)$, by \cite[Prop. 3.2, p. 865]{m1}.  This is impossible because $\wp_3 \mid \beta_2$ but $(\alpha_1, \wp_3) = 1$.  Hence, the first factor must be $0$, which gives that $\beta_1 = \beta_2$ and $j(\mathfrak{k}_1) = j(\mathfrak{k}_2)$.  (See the formula for $j(E_3)$ in Section \ref{sec:5}.)  This proves the result we want.  \bigskip

\begin{thm}
If $\mathfrak{m} = (6) = (2) \wp_3\wp_3'$, where $2$ is inert in $K$, then the ray class field $K_\mathfrak{m}$ is generated by a single $\tau$-invariant for the ideal $\mathfrak{m}$.
\label{thm:9}
\end{thm}

Note that the $\tau$-invariant for the point $Q$ is certainly an invariant for the ideal $\mathfrak{m}$, since the quantity $X(Q)$ generates $K_\mathfrak{m}$.  This proves Sugawara's conjecture for the first four cases in line 2 of \eqref{eqn:4}.  We will postpone the discussion of the fifth case to Section \ref{sec:11}.  We have also now established the main results of \cite{ksy} using our methods.  Next we turn to ideals divisible by some $\wp_5$.

\section{The case $\mathfrak{m} = \wp_5$.}
\label{sec:8}

Let
\begin{equation*}
E_5(b): \ Y^2 + (1+b)XY +bY=X^3+bX^2
\end{equation*}
be the Tate normal form for a point of order $5$.  We use the computations and results of \cite{m2}.  If $\frak{m} = \wp_5'$, where $(5) = \wp_5 \wp_5'$ in $K = \mathbb{Q}(\sqrt{-d})$, then $d \equiv \pm 1$ (mod $5$).  If $b = r(w/5)^5$, where $w$ is a suitable integer in $R_K$ and $r(\tau)$ is the Rogers-Ramanujan continued fraction, then $E_5(b)$ has complex multiplication by the ring of integers $R_K$ of $K$.  See \cite[Thm. 2.1]{m2}.  By choosing $w$ appropriately, we can arrange for $\wp_5'$ to be either of the prime ideals dividing $(5)$. \medskip

The Weierstrass normal form of $E_5(b)$ is
$$Y_1^2 = 4X_1^3 - g_2X_1 -g_3,$$
where $X_1 = X + \frac{1}{12}(b^2+6b+1)$,
\begin{align*}
g_2 &=  g_2(\mathfrak{k}) = \frac{1}{12}(b^4 + 12b^3 + 14b^2 - 12b + 1),\\
g_3 &= \frac{-1}{216}(b^2 + 1)(b^4 + 18b^3 + 74b^2 - 18b + 1);\\
\Delta &=-b^5(b^2 + 11b - 1).
\end{align*}
Thus,
$$j(E_5) = j(\mathfrak{k}) = -\frac{(b^4 + 12b^3 + 14b^2 - 12b + 1)^3}{b^5(b^2 + 11b - 1)}.$$
Two of the $X$-coordinates of points in $E_5[5]$ are $X = 0, -b$.  Hence, we let 
\begin{align*}
\notag &\tau_0 = -2^7 3^5 \frac{g_2 g_3}{\Delta}(0 + \frac{1}{12}(b^2+6b+1))\\
& = - \frac{(b^4 + 12b^3 + 14b^2 - 12b + 1)(b^4 + 18b^3 + 74b^2 - 18b + 1)}{b^5(b^2 + 11b - 1)}\\
& \ \ \ \times (b^2 + 1)(b^2 + 6b + 1);
\end{align*}
and
\begin{align*}
\notag &\tau_1 = -2^7 3^5 \frac{g_2 g_3}{\Delta}(-b + \frac{1}{12}(b^2+6b+1))\\
& = -\frac{(b^4 + 12b^3 + 14b^2 - 12b + 1)(b^4 + 18b^3 + 74b^2 - 18b + 1)}{b^5(b^2 + 11b - 1)}\\
& \ \ \ \times (b^2 + 1)(b^2 - 6b + 1).
\end{align*}
Note that the only difference between $\tau_0$ and $\tau_1$ is in the final factor of the numerator, so that
$$\frac{\tau_0}{\tau_1} = \frac{b^2 + 6b + 1}{b^2 - 6b + 1}.$$
From \cite[Thm. 4.6]{m2} we know that the ray class field $\textsf{K}_\frak{f} = F_1$ of conductor $\mathfrak{m} = \wp_5'$ over $K$ is generated over the Hilbert class field $\Sigma$ (and even over $\mathbb{Q}$) by the quantity $b = r(w/5)^5$.  Also, $\sigma: b \rightarrow -1/b$ is the nontrivial automorphism of $F_1/\Sigma$.  The expressions for the $\tau_i$ imply easily that $\tau_0^\sigma = \tau_1$, so that $\tau_0$ and $\tau_1$ are quadratic over $\Sigma$ and lie in $F_1$.  Furthermore, $\tau_0 \neq \tau_1$, since
$$\frac{\tau_0}{\tau_1} - 1 = \frac{12b}{b^2 - 6b + 1}$$
and $b \neq 0$. \medskip

We want to show that $K(\tau_0) = K(\tau_1) = F_1$.  We compute that
\begin{equation}
\label{eqn:8.1} (\tau_0-\tau_1)^2 =12b^2 \frac{j(\mathfrak{k})(j(\mathfrak{k})-1728)}{g_2(\mathfrak{k})}.
\end{equation}
and
\begin{equation}
\label{eqn:8.2} \frac{\tau_0-\tau_1}{\tau_0+\tau_1} = \frac{6b}{b^2+1}.
\end{equation}
It follows that if $\{\tau_0,\tau_1\} = \{\tau_0',\tau_1'\}$ for two different ideal classes $\mathfrak{k}, \mathfrak{k}'$, corresponding to the values $b, b'$, then
$$\frac{6b}{b^2+1} = \pm \frac{6b'}{(b')^2+1}.$$
This equation easily implies that $b'$ is given by one of the possibilities $b' = b, -1/b, -b, 1/b$.  If $b' = b, -1/b$ it follows that $j(\mathfrak{k'}) = j(\mathfrak{k})$ and $\mathfrak{k} = \mathfrak{k}'$. It remains to eliminate the cases $b' = -b, 1/b$, which are equivalent, since $j(\mathfrak{k}) = j(b)$ is invariant under $\sigma$. \medskip

If, without loss of generality, $b' = -b$, then \eqref{eqn:8.1} implies that
$$\frac{j(\mathfrak{k}')(j(\mathfrak{k}')-1728)}{g_2(\mathfrak{k}')} = \frac{j(\mathfrak{k})(j(\mathfrak{k})-1728)}{g_2(\mathfrak{k})}$$
or, with easily understood notation,
\begin{align*}
0 &= g_2(b) j(-b) (j(-b) - 1728) - g_2(-b) j(b) (j(b) - 1728)\\
& = -\frac{P(b)}{3b^9(b^2 + 11b - 1)^2(b^2 - 11b - 1)^2},
\end{align*}
where
\begin{align*}
P(b) & = (b^2 + 1)^2 (b^4 + 12b^3 + 14b^2 - 12b + 1)(b^4 - 12b^3 + 14b^2 + 12b + 1)\\
& \times (b-1)(b + 1)(19b^8 - 2264b^6 - 8886b^4 - 2264b^2 + 19)\\
& \times (b^8 - 26b^6 - 11934b^4 - 26b^2 + 1).
\end{align*}
Now, the roots of the first two factors yield $j(b) = 1728, 0$, which are excluded, since they imply $d_K = -4, -3$.  The roots of the remaining quartic, as well as $b = \pm 1$, yield values of $j$ which are not algebraic integers, as can easily be checked.  We just have to eliminate the last two factors as possibilities.  Let
\begin{align*}
f_1(x) & = 19b^8 - 2264b^6 - 8886b^4 - 2264b^2 + 19,\\
f_2(x) & = b^8 - 26b^6 - 11934b^4 - 26b^2 + 1 = b^4 m\left(b-\frac{1}{b}\right),\\
m(x) & = x^4-22x^2-11984.
\end{align*}
The roots of $f_1(x)$ are not algebraic integers, but $b$ is a unit, so this polynomial cannot have $b$ as a root.  Furthermore, if $f_2(x)$ occurred as the minimal polynomial of $b$, $m(x)$ would be the minimal polynomial of $z = b-1/b \in \Sigma$.  However, the discriminant of $m(x)$ is divisible by $5^2$, and it is easily checked that $5^2 \mid d_L$,where $L$ is a root field of $m(x)$.  But this is impossible, since $5$ does not ramify in $K$ or in $\Sigma/K$.  Thus, the roots of $f_2(x)$ also cannot occur in the present situation. \medskip

This shows that $b'$ cannot be $-b$ or $1/b$, and therefore $b' = b, -1/b$ and $\mathfrak{k} = \mathfrak{k}'$.  This gives the following. \medskip

\begin{thm}
For $\mathfrak{m} = \wp_5$ or $\wp_5'$, where $(5) = \wp_5 \wp_5'$ in $K$, Sugawara's conjecture holds, namely, $\textsf{K}_\mathfrak{m} = K(\tau(\mathfrak{k}^*))$ is generated over the quadratic field $K$ by a single ray class invariant for the modulus $\mathfrak{m}$.
\label{thm:10}
\end{thm}

If $\mathfrak{m} = \mathfrak{p}$, where $\mathfrak{p}^2 = (5)$, then the above computations and arguments all hold, except for the argument which eliminated $f_2(x)$.  In this case the ray class field $\Sigma_\mathfrak{p} = \mathbb{Q}(b)$ is normal over $\mathbb{Q}$.  (See the discussion in \cite[Section 5, pp. 123-129]{m3} for this case.)  However, the polynomial $f_2$ splits modulo $41$ into distinct linears and irreducible quadratics.  Since a normal polynomial splits into irreducible factors of the same degree mod $p$, for all primes not dividing the discriminant, this fact shows that $f_2(x)$ is not normal and hence can be eliminated as a possibility.  Thus, Sugawara's conjecture also holds in this case.

\section{The case $\mathfrak{m} = \wp_3 \wp_5'$.}
\label{sec:9}

For this case we might think of using the Tate normal form for a point of order $15$.  However, the coefficients of the defining equation for $E_{15}$ are unwieldy, so it turns out to be more convenient to again make use of the curve
\begin{equation}
\label{eqn:9.1} E_5(b): \ Y^2 + (1+b)XY +bY=X^3+bX^2,
\end{equation}
where, as in \cite{m2}, $b \in \textsf{K}_{\wp_5'}$, and $E_5(b)$ has the Weierstrass normal form
$$E': \ Y_1^2 = 4X_1^3 - g_2X_1 -g_3,$$
with $X_1 = X + \frac{1}{12}(b^2+6b+1)$ and
\begin{align*}
g_2 &= \frac{1}{12}(b^4 + 12b^3 + 14b^2 - 12b + 1),\\
g_3 &= \frac{-1}{216}(b^2 + 1)(b^4 + 18b^3 + 74b^2 - 18b + 1);\\
\Delta &=-b^5(b^2 + 11b - 1).
\end{align*}
The doubling formula on $E_5(b)$ (see \cite[p. 54]{si1}) is
$$X(2P) = F(X) =  \frac{X^4 - (b^2 + b)X^2 - 2b^2 X - b^3}{4X^3 + (b^2 + 6b + 1)X^2 + 2b(b + 1)X + b^2}, \ X = X(P);$$
so setting $F(X) = X, X = X(P)$, yields the polynomial whose roots are the $X$-coordinates of points $P$ of order $3$ on $E_5(b)$:
\begin{equation}
\label{eqn:9.2} g(X) = 3X^4 + (b^2 + 6b + 1)X^3 + (3b^2 + 3b)X^2 + 3b^2X + b^3.
\end{equation}
Assuming that $-d \equiv 1, 4$ (mod $15$), the ideals $\wp_3, \wp_3',$ and $\wp_3 \wp_3' = (3)$ are associated to $1, 1$, and $2$ $\tau$-invariants, respectively.  There are four roots of $g(X)$, so these correspond to these three ideals in some permutation.  \medskip

Now we solve $g(X) = 0$.  First, we shift to eliminate the $X^3$ term:
\begin{equation}
\label{eqn:9.3} \frac{1}{3} g\left(X-\frac{b^2+6b+1}{12}\right) = X^4 +p X^2 + qX+r,
\end{equation}
where
\begin{align*}
p & = -\frac{1}{24}(b^4 + 12b^3 + 14b^2 - 12b + 1) = -\frac{g_2}{2},\\
q & = \frac{(b^2 + 1)(b^4 + 18b^3 + 74b^2 - 18b + 1)}{216} = -g_3,\\
r & = -\frac{(b^4 + 12b^3 + 14b^2 - 12b + 1)^2}{6912} = \frac{-g_2^2}{48}.
\end{align*}
The cubic resolvent of (\ref{eqn:9.3}) is
\begin{equation*}
k(y) = y^3-2py^2+(p^2-4r)y+q^2 = y^3+g_2 y^2+\frac{g_2^2}{3} y + g_3^2,
\end{equation*}
for which we have
$$k\left(y-\frac{g_2}{3}\right) = y^3-\frac{g_2^3-27g_3^2}{27} = y^3-\frac{\Delta}{27}.$$
Hence, setting $z = b-\frac{1}{b}$, one root of $k(y)$ is
\begin{align}
\label{eqn:9.4} \theta_1 &= -\frac{g_2}{3}+\frac{\Delta^{1/3}}{3}\\
\notag & = \frac{-1}{36}(b^4 + 12b^3 + 14b^2 - 12b + 1)-\frac{1}{3}b^{5/3}(b^2+11b-1)^{1/3}\\
\notag & = \frac{-1}{36}(b^4 + 12b^3 + 14b^2 - 12b + 1)-\frac{1}{3} b^2 \left(b-\frac{1}{b}+11\right)^{1/3}\\
\notag & =  \frac{-b^2}{36}\left(b^2 + 12b + 14 - \frac{12}{b} + \frac{1}{b^2}\right)-\frac{1}{3} b^2 (z+11)^{1/3}.
\end{align}
Now put $z+11 = \rho^3$.  Note that $z+11 \cong \wp_5'^3$ and $\rho \cong \wp_5'$ (see \cite[p. 1193]{m2}).  This gives that
\begin{align}
\label{eqn:9.5} -\theta_1 &= \frac{b^2}{36}(z^2 + 12z + 16)+\frac{1}{3} b^2 \rho = \frac{b^2}{36}(z^2 + 12z + 16+12\rho)\\
\notag & = \frac{b^2}{36}(\rho^6-10\rho^3+12\rho+5)\\
\notag & =  \frac{b^2}{36}(\rho^2 + 2\rho + 5)(\rho^2 - \rho - 1)^2,
\end{align}
and therefore
\begin{equation}
\label{eqn:9.6} \sqrt{-\theta_1} = \frac{b}{6}(\rho^2 - \rho - 1)\sqrt{\rho^2 + 2\rho + 5}.
\end{equation}
\noindent {\bf Remark.} For the computations to follow, note that $\theta_1 \neq 0$ and therefore $\rho^2+2\rho+5 \neq 0$, since the constant term of $k(y)$ is $g_3^2 \neq 0$.  This holds because we are excluding the field $\mathbb{Q}(\sqrt{-4})$ from consideration, for which the only corresponding $j$-invariant is $j(\mathfrak{k}) = 1728$.  \medskip

The other two roots of $k(y)$ are obtained by replacing $\Delta^{1/3}$ in \eqref{eqn:9.4} by $\Delta^{1/3} \omega^i$, or $\rho$ by $\omega^i \rho$ in \eqref{eqn:9.6}, for $i = 1,2$, giving
\begin{align*}
\sqrt{-\theta_2} &= \frac{b}{6}(\omega^2 \rho^2 - \omega \rho - 1)\sqrt{\omega^2 \rho^2 + 2\omega \rho + 5}\\
\sqrt{-\theta_3} &= \frac{b}{6}(\omega \rho^2 - \omega^2 \rho - 1)\sqrt{\omega \rho^2 + 2\omega^2 \rho + 5}.
\end{align*}
The roots of $g(X)$ are
\begin{align*}
X(P_1) = x_1 &= -\frac{b^2+6b+1}{12}+\frac{1}{2}(\sqrt{-\theta_1}+\sqrt{-\theta_2}+\sqrt{-\theta_3}),\\
X(P_2) = x_2 &= -\frac{b^2+6b+1}{12}+\frac{1}{2}(\sqrt{-\theta_1}-\sqrt{-\theta_2}-\sqrt{-\theta_3}),\\
X(P_3) = x_3 &= -\frac{b^2+6b+1}{12}+\frac{1}{2}(-\sqrt{-\theta_1}+\sqrt{-\theta_2}-\sqrt{-\theta_3}),\\
X(P_4) = x_4 &= -\frac{b^2+6b+1}{12}+\frac{1}{2}(-\sqrt{-\theta_1}-\sqrt{-\theta_2}+\sqrt{-\theta_3}).
\end{align*}
Therefore, the corresponding $\tau$-invariants are:
\begin{align*}
\tau(\mathfrak{k}_1^*) &= \lambda\left(X(P_1) +\frac{b^2+6b+1}{12}\right) = \frac{\lambda}{2}(\sqrt{-\theta_1}+\sqrt{-\theta_2}+\sqrt{-\theta_3}),\\
\tau(\mathfrak{k}_2^*) &= \lambda\left(X(P_2) +\frac{b^2+6b+1}{12}\right)= \frac{\lambda}{2}(\sqrt{-\theta_1}-\sqrt{-\theta_2}-\sqrt{-\theta_3}),\\
\tau(\mathfrak{k}_3^*) &= \lambda\left(X(P_3) +\frac{b^2+6b+1}{12}\right)= \frac{\lambda}{2}(-\sqrt{-\theta_1}+\sqrt{-\theta_2}-\sqrt{-\theta_3}),\\
\tau(\mathfrak{k}_4^*) &= \lambda\left(X(P_4) +\frac{b^2+6b+1}{12}\right)= \frac{\lambda}{2}(-\sqrt{-\theta_1}-\sqrt{-\theta_2}+\sqrt{-\theta_3}),
\end{align*}
where
\begin{align*}
\lambda &= -2^7 3^5 \frac{g_2 g_3}{\Delta}\\
& = -12\frac{(b^4 + 12b^3 + 14b^2 - 12b + 1)(b^2 + 1)(b^4 + 18b^3 + 74b^2 - 18b + 1)}{b^5(b^2 + 11b - 1)}.
\end{align*}

We know that these $\tau$-invariants lie in $\textsf{K}_{(3)} = \Sigma(\omega)$.  By our earlier arguments for $\mathfrak{m} = (3)$ we also know two of them are conjugate and generate $\Sigma(\omega)$ over $K$.  The other two must therefore lie in $\textsf{K}_{\wp_3} = \textsf{K}_{\wp_3'} = \Sigma$.  Consider the sum
\begin{equation}
\label{eqn:9.7} \tau(\mathfrak{k}_1^*)+\tau(\mathfrak{k}_2^*) = \lambda \sqrt{-\theta_1} = \frac{\lambda b}{6}(\rho^2 - \rho - 1)\sqrt{\rho^2 + 2\rho + 5}.
\end{equation}
The quantity $\lambda b$ lies in $\textsf{K}_{\wp_5'}$ and $z \in \Sigma$, so that \eqref{eqn:9.5} and \eqref{eqn:9.7} imply that $\rho \in \textsf{K}_{\wp_5'} \textsf{K}_{(3)} = \textsf{K}_{\wp_5'}(\omega)$.  But this extension has degree $4$ over $\Sigma$, so the cubic $Y^3-(z+11)$ must have a root in $\Sigma$.  Hence, we may assume $\rho \in \Sigma$.  With this assumption I claim that $\sqrt{\rho^2 + 2\rho + 5} \in \textsf{K}_{\wp_5'}$.  Using $z = b-1/b$ we have
\begin{align*}
\lambda b &= -12\frac{b^6(z^2+12z+16)(b + 1/b)(z^2+18z+76)}{b^6(z+11)}\\
& = -12\left(b+\frac{1}{b}\right) \frac{(z^2+12z+16)(z^2+18z+76)}{z+11};
\end{align*}
therefore,
$$(\lambda b)^2 = 144 \left(b+\frac{1}{b}\right)^2 \left(\frac{(z^2+12z+16)(z^2+18z+76)}{z+11}\right)^2.$$
But $(b+1/b)^2 = z^2+4$, so that $(\lambda b)^2 \in \Sigma$ and $\lambda b$ is a Kummer element for $\textsf{K}_{\wp_5'}/\Sigma$.  It follows from \eqref{eqn:9.5} that
\begin{align*}
(\tau(\mathfrak{k}_1^*)+\tau(\mathfrak{k}_3^*))^2 &= (\tau(\mathfrak{k}_2^*)+\tau(\mathfrak{k}_4^*))^2\\
& = (\lambda \sqrt{-\theta_2})^2 = \frac{(\lambda b)^2}{36}(z^2 + 12z + 16+12\rho \omega);\\
(\tau(\mathfrak{k}_1^*)+\tau(\mathfrak{k}_4^*))^2 &= (\tau(\mathfrak{k}_2^*)+\tau(\mathfrak{k}_3^*))^2\\
& = (\lambda \sqrt{-\theta_3})^2 = \frac{(\lambda b)^2}{36}(z^2 + 12z + 16+12\rho \omega^2).
\end{align*}
Now $\rho \omega$ and $\rho \omega^2$ are conjugates over $\Sigma$, whence it follows that $\sqrt{-\theta_2}$ and $\pm \sqrt{-\theta_3}$ are conjugate over $\textsf{K}_{\wp_5'} = \Sigma(b)$.  Choosing signs so that $\sqrt{-\theta_2}$ and $\sqrt{-\theta_3}$ are conjugate, it follows from the relation
$$\sqrt{-\theta_2} \sqrt{-\theta_3} = \frac{g_3}{\sqrt{-\theta_1}}$$
(with the sign of $\sqrt{-\theta_1}$ chosen correctly) that $\tau(\mathfrak{k}_1^*)$ and $\tau(\mathfrak{k}_2^*)$ are fixed by the automorphism $\alpha = (b, \omega) \rightarrow (b, \omega^2)$ of $\textsf{K}_{\wp_5'}\textsf{K}_{(3)}$ over $\Sigma$, while $\tau(\mathfrak{k}_3^*)$ and $\tau(\mathfrak{k}_4^*)$ are interchanged.  Since the squares above lie in $\Sigma(\omega) \backslash \Sigma$, it follows that $\{\tau(\mathfrak{k}_1^*), \tau(\mathfrak{k}_2^*)\}$ must be the pair of invariants which lies in $\Sigma$.  It follows that either
$$\wp_3 P_1 = O \ \textrm{or} \ \wp_3' P_1 = O \ \textrm{on} \ E_5(b).$$
 Also, $(\lambda b)\sqrt{\rho^2 + 2\rho + 5} \in \Sigma$, so $\sqrt{\rho^2 + 2\rho + 5}$ must also be a Kummer element for $\textsf{K}_{\wp_5'}/\Sigma$.  \medskip

As a corollary of the discussion so far, we have the following fact, which follows from $\rho^3 = z + 11 = -(\eta(w/5)/\eta(w))^6$ (see \cite{m2}).
\begin{thm}
If $-d \equiv 1, 4$ (mod $15$), then for some $i \in \{0,1,2\}$ we have that
$$\rho = -\omega^i \left(\frac{\eta(w/5)}{\eta(w)}\right)^2 \in \Sigma = \textsf{K}_1,$$
where $w = \frac{v+\sqrt{d_K}}{2} \in R_K$ satisfies  $5^2 \mid N(w)$, as in \cite[Thm. 1.1]{m2}; and $\Sigma(\sqrt{\rho^2+2\rho+5}) = \textsf{K}_{\wp_5'}$.
\label{thm:11}
\end{thm}

Our next task is to find primitive $\mathfrak{m}$-division points on the curve $E_5(b)$, where $\mathfrak{m} = \wp_3\wp_5'$. \medskip

The kernel of multiplication by $\wp_5'$ on $E_5(b)$ is
$$\textrm{ker}(\wp_5') = \{(0,0), (0,-b), (-b,0), (-b,b^2)\} = \{Q_1, Q_2, Q_3, Q_4\},$$
since the associated $\tau$-invariants generate $\textsf{K}_{\wp_5'}$.  Hence the points $Q_i+P_1$ satisfy $\wp_3 \wp_5' (Q_i+P_1) = O$ or $\wp_3' \wp_5' (Q_i+P_1) = O$, depending on whether $\wp_3 (P_1) = O$ or $\wp_3' (P_1) = O$.  Assume the former.  Setting $y_1 = Y(P_1)$, we have that
\begin{equation*}
Q_1 + P_1 = (0,0) + (X(P_1),Y(P_1)) = (0,0) + (x_1,y_1);
\end{equation*}
so that
\begin{equation*}
\tilde{x}_1 = X(Q_1+P_1) = \left(\frac{y_1}{x_1}\right)^2 +(1+b) \frac{y_1}{x_1} - b - 0 -x_1 = \frac{-by_1}{x_1^2}. 
\end{equation*}
(Note that $x_i \neq 0$, for $1 \le i \le 4$, since the constant term of $g(X)$ is $b^3 \neq 0$.)  Hence, we have that
\begin{align*}
\frac{x_1}{b^2} \tilde{x}_1^2-\frac{b+1}{b} \tilde{x}_1 &= \frac{y_1^2+(1+b)x_1y_1}{x_1^3}\\
& = \frac{-by_1+x_1^3+bx_1^2}{x_1^3} = \frac{1}{x_1}\tilde{x}_1+1 + \frac{b}{x_1};
\end{align*}
and $\tilde{x}_1 = X(Q_1+P_1)$ satisfies the equation
\begin{equation}
\label{eqn:9.8} \frac{x_1}{b^2} X^2-\left(\frac{b+1}{b}+\frac{1}{x_1}\right)X-1 - \frac{b}{x_1} = 0.
\end{equation}
Similarly, we find
$$\tilde{x}_2 = X(Q_2+P_1) = \frac{by_1+b^2}{x_1^2}+\frac{b+b^2}{x_1},$$
so that
$$\tilde{x}_1+\tilde{x}_2 = \frac{b^2}{x_1^2}+\frac{b+b^2}{x_1}$$
is the coefficient of $X$ in \eqref{eqn:9.8} after dividing through by $x_1/b^2$; and
$$\tilde{x}_1 \tilde{x}_2 = -\frac{b^2}{x_1}-\frac{b^3}{x_1^2},$$
which is the constant term in \eqref{eqn:9.8} after dividing by $x_1/b^2$.  Now the coefficients of \eqref{eqn:9.8} involve the quantities $b$ and
$$x_1 = -\frac{b^2+6b+1}{12}+\frac{\tau(\mathfrak{k}_1^*)}{\lambda},$$
and therefore lie in $\Sigma(b) = \textsf{K}_{\wp_5'}$ (recall that $\tau(\mathfrak{k}_1^*) \in \Sigma$).  Thus, $\tilde{x}_1,\tilde{x}_2$ are roots of the equation
$$h_1(X) = X^2-\left(\frac{b^2}{x_1^2}+\frac{b+b^2}{x_1}\right)X-\frac{b^2}{x_1}-\frac{b^3}{x_1^2} \in \textsf{K}_{\wp_5'}[X],$$
which I claim is irreducible over $\textsf{K}_{\wp_5'}$.  Its roots satisfy
\begin{align*}
& \lambda\left(\tilde{x}_1+\frac{b^2+6b+1}{12}\right) = \tau(\bar{\mathfrak{k}}_1^*),\\
& \lambda\left(\tilde{x}_2+\frac{b^2+6b+1}{12}\right) = \tau(\bar{\mathfrak{k}}_2^*),
\end{align*}
which are the $\tau$-invariants for $\mathfrak{m} = \wp_3 \wp_5'$ corresponding to the points $Q_1+P_1, Q_2+P_1$.  Since these invariants satisfy a quadratic equation over $\textsf{K}_{\wp_5'}$ and must generate $\textsf{K}_{\wp_5'\wp_3}$ over $\Sigma$, they are conjugate over $\textsf{K}_{\wp_5'}$.  It follows that $h_1(X)$ is irreducible over $\textsf{K}_{\wp_5'}$.  Furthermore, since the roots
\begin{align*}
\tilde{x}_3 & = X(Q_3+P_1) = \frac{b^2y_1-bx_1^2-b^2x_1}{(x_1+b)^2},\\
\tilde{x}_4 & = X(Q_4+P_1) = -\frac{b^2 y_1 + bx_1^2 + (b^3 +2b^2) x_1 + b^3}{(x_1+b)^2},
\end{align*}
are related by the same linear transformation to the invariants $\tau(\bar{\mathfrak{k}}_i^*), (i = 3,4)$, as are $\tilde{x}_1, \tilde{x}_2$ to their invariants, the roots $\tilde{x}_3, \tilde{x}_4$ must satisfy the equation
$$h_2(X) = X^2+\left(\frac{b(2x_1^2 + (b^2+3b)x_1+b^2)}{(x_1 + b)^2}\right)X+\frac{b^2 x_1}{x_1+b},$$
where
\begin{equation}
\label{eqn:9.9} \lambda^{-4} T_\mathfrak{m}\left(\lambda X+\lambda \frac{b^2+6b+1}{12},\mathfrak{k}\right) = h_1(X) h_2(X), \ \ \mathfrak{m} = \wp_5' \wp_3.
\end{equation}
(Here, as below, $x_i+b \neq 0$ because $g(-b) = b^5 \neq 0$.)  This implies that $h_2(X)$ is also irreducible over $\textsf{K}_{\wp_5'}$.  As a corollary of this discussion we note:

\begin{thm}
If the torsion point $P_1$ on $E_5(b)$ satsifies $\wp_3(P_1) = O$, its coordinates $P_1 = (x_1,y_1)$ generate $\textsf{K}_{\wp_5'\wp_3}$ over $\mathbb{Q}$.
\label{thm:12}
\end{thm}

\begin{proof}
We have that $\textsf{K}_{\wp_5'}(y_1) = \textsf{K}_{\wp_5'\wp_3}$, since $\tilde{x}_1 = -by_1/x_1^2$ generates this field over $\textsf{K}_{\wp_5'}$ (recalling that $\tau(\mathfrak{k}_1^*) \in \Sigma \Rightarrow x_1 \in \textsf{K}_{\wp_5'}$).  Moreover, if $\textsf{P}_\mathfrak{m}$ denotes $\left(R_K/\mathfrak{m}\right)^\times$, then 
$$\textsf{P}_{\wp_5' \wp_3}/\langle -1 \rangle \cong (\textsf{P}_{\wp_5'} \times \textsf{P}_{\wp_3})/\langle(-1,-1)\rangle$$
is cyclic of order $4$, generated by $(2,-1)$, from which it follows that $\textsf{K}_{\wp_5' \wp_3}/\Sigma$ is a cyclic quartic extension.  Since $y_1 \in \textsf{K}_{\wp_5'\wp_3} \backslash \textsf{K}_{\wp_5'}$, this gives that $\textsf{K}_{\wp_5'\wp_3} = \Sigma(y_1)$.  Furthermore, by the defining equation for $E_5(b)$, we have
$$y_1^2 +x_1y_1-x_1^3 = b(-x_1y_1-y_1+x_1^2).$$
The right side is clearly nonzero, so that $b \in \mathbb{Q}(x_1,y_1)$.  Now the fact that $\mathbb{Q}(b) = \textsf{K}_{\wp_5'}$ yields the assertion of the theorem.  (See \cite[Thm. 4.6, p. 1196]{m2}.)
\end{proof}

Now assume that
$$T_\mathfrak{m}\left(X,\mathfrak{k}\right) = T_\mathfrak{m}\left(X,\mathfrak{k}'\right) = T_\mathfrak{m}\left(X,\mathfrak{k}\right)^\psi$$
for two ideal classes $\mathfrak{k},\mathfrak{k}'$, where $j(\mathfrak{k}') = j(\mathfrak{k})^{\psi}$ and $\psi$ is an automorphism of $\Sigma/K$.  Assume $\psi$ has been extended to an automorphism of $\textsf{K}_{\wp_5' \wp_3}/K$, and denote the images $\alpha^\psi = \alpha'$ under $\psi$ by primes.  It is clear that $(0,0) \in E_5(b)$ maps to $(0,0) \in E_5(b')$ and $P_1 = (x_1,y_1)$ maps to $P_1' = (x_1',y_1')$.  The field $\textsf{K}_{\wp_5'}$ is normal over $K$, so the polynomials $h_1(X), h_2(X)$ above are mapped to the pair of irreducible polynomials $h_1^\psi(X), h_2^\psi(X)$ over $\textsf{K}_{\wp_5'}$. \medskip

Assume first that $\{\tau(\bar{\mathfrak{k}}_1'^*),\tau(\bar{\mathfrak{k}}_2'^*)\} = \{\tau(\bar{\mathfrak{k}}_1^*),\tau(\bar{\mathfrak{k}}_2^*)\}$.  Then the differences
\begin{align*}
\tau(\bar{\mathfrak{k}}_1^*) - \tau(\bar{\mathfrak{k}}_2^*) &= \lambda (\tilde{x}_1-\tilde{x}_2)\\
\tau(\bar{\mathfrak{k}}_1'^*) - \tau(\bar{\mathfrak{k}}_2'^*) &= \lambda' (\tilde{x}_1'-\tilde{x}_2')
\end{align*}
imply that
\begin{equation}
\label{eqn:9.10} \lambda (\tilde{x}_1-\tilde{x}_2) = \pm   \lambda' (\tilde{x}_1'-\tilde{x}_2').
\end{equation}
There is a similar relation for $\tilde{x}_3, \tilde{x}_4$ and their images under $\psi$.  Now
$$\frac{\tau(\bar{\mathfrak{k}}_1^*) - \tau(\bar{\mathfrak{k}}_2^*)}{\tau(\bar{\mathfrak{k}}_3^*) - \tau(\bar{\mathfrak{k}}_4^*)} = \frac{\tilde{x}_1-\tilde{x}_2}{\tilde{x}_3-\tilde{x}_4} = -\frac{(x_1+b)^2}{bx_1^2}.$$
Hence, \eqref{eqn:9.10} gives that
\begin{equation}
\label{eqn:9.11} \frac{(x_1+b)^2}{bx_1^2} = \pm \frac{(x_1'+b')^2}{b'x_1'^2}.
\end{equation}
We now recall that $z+11 \cong \wp_5'^3$ (see \cite[p. 1193] {m2}).  Let $\mathfrak{q}$ be the product of prime ideals dividing $\wp_5'$ in $\textsf{K}_{\wp_5'}$, so that $\mathfrak{q}^2 = \wp_5'$.  Since $z = b-1/b$ and $b$ is a unit, this gives that $b^2+11b-1 \equiv (b-57)^2 \equiv 0$ mod $\mathfrak{q}^6$, so $b \equiv 2$ mod $\mathfrak{q}$.  Furthermore, from \eqref{eqn:9.2} the congruence
$$g(X) \equiv 3X^4 + 2X^3 + 3X^2 + 2X + 3 \equiv 3(X + 1)^4 \ \textrm{mod} \ \mathfrak{q}$$
implies that $x_1 \equiv -1$ mod $\mathfrak{q}$.  This gives that
\begin{equation}
\label{eqn:9.12} \frac{(x_1+b)^2}{bx_1^2} \equiv \frac{(-1+2)^2}{2(-1)^2} \equiv 3 \ \textrm{mod} \ \mathfrak{q}.
\end{equation}
Since this congruence also holds for $\frac{(x_1'+b')^2}{b'x_1'^2}$ (and $\mathfrak{q}$ is invariant under $\psi$), this shows that only the plus sign in \eqref{eqn:9.11} can hold.  It follows from
\begin{equation}
\label{eqn:9.13} \frac{(x_1+b)^2}{bx_1^2} = \frac{(x_1'+b')^2}{b'x_1'^2}
\end{equation} 
that $b/b' = b/b^\psi = \textsf{B}^2, \ \textsf{B} \in \textsf{K}_{\wp_5'}$. \medskip

This suggests the following conjecture.  For its statement recall that a quadratic discriminant can be written as $d_K = \prod_{p \mid d_K}{p^*}$, where $p^* = (-1)^{(p-1)/2} p$, if $p$ is odd, and $2^*$ is one of the possibilities $2^* = -4, 8, -8$.

\begin{conj}
Assume that  $-d \equiv 1, 4$ (mod $15$) and the $2$-factor of $d_k = \prod_{p \mid d}{p^*}$ is not $2^* = -4$.  If $b/b^\psi \in (\textsf{K}_{\wp_5'}^\times)^2$ for some $\psi \in \textrm{Gal}(\textsf{K}_{\wp_5'}/K)$, then $\psi = 1$.  Moreover, there is a unique $\psi \in \textrm{Gal}(\textsf{K}_{\wp_5'}/K)$ for which $b/b^\psi = -\textsf{B}^2$, with $\textsf{B} \in \textsf{K}_{\wp_5'}$.  Namely, $\psi: b \rightarrow -1/b$ is the unique automorphism in $\textrm{Gal}(\textsf{K}_{\wp_5'}/K)$ with this property.
\label{conj:2}
\end{conj}

The assumption on $2^*$ is equivalent to the assertion that $\sqrt{-1}$ does not lie in the genus field of $K$. \medskip

If \eqref{eqn:9.13} holds, then 
\begin{equation*}
\frac{(x_1+b)^2}{bx_1^2} = \frac{(x_1'+b')^2}{b'x_1'^2} = \left(\frac{(x_1+b)^2}{bx_1^2}\right)^\psi
\end{equation*}
implies that $\frac{(x_1+b)^2}{bx_1^2} = \textsf{A}$ lies in the fixed field $L$ of $\langle \psi \rangle$ inside $\textsf{K}_{\wp_5'}$.  This is the case if 
$$\{\tau(\bar{\mathfrak{k}}_1'^*),\tau(\bar{\mathfrak{k}}_2'^*)\} = \{\tau(\bar{\mathfrak{k}}_1^*),\tau(\bar{\mathfrak{k}}_2^*)\}.$$
If, on the other hand,
$$\{\tau(\bar{\mathfrak{k}}_1'^*),\tau(\bar{\mathfrak{k}}_2'^*)\} = \{\tau(\bar{\mathfrak{k}}_3^*),\tau(\bar{\mathfrak{k}}_4^*)\},$$
then we have, by the congruence condition \eqref{eqn:9.12}, the equation
\begin{equation}
\label{eqn:9.14} \frac{(x_1+b)^2}{bx_1^2} = -\frac{b'x_1'^2}{(x_1'+b')^2}.
\end{equation}
Let $\bar{\psi}$ denote the automorphism $\bar{\psi}: b \rightarrow -1/b$ in $\textrm{Gal}(\textsf{K}_{\wp_5'}/\Sigma)$.  Setting $\tau = \tau_1(\mathfrak{k}_1^*) \in \Sigma$, we have that $(b\lambda)^{\bar{\psi}} = -b\lambda$ and 
\begin{align*}
x_1^{\bar{\psi}} &= \left(-\frac{b^2+6b+1}{12}+\frac{\tau}{\lambda}\right)^{\bar{\psi}}\\
&= -\frac{b^2-6b+1}{12b^2}+\frac{\tau}{b^2\lambda}\\
&= \frac{1}{b^2}(x_1+b).
\end{align*}
Hence,
$$(x_1+b)^{\bar{\psi}} = \frac{1}{b^2}(x_1+b) -\frac{1}{b} = \frac{x_1}{b^2}$$
and
\begin{equation*}
\left(\frac{(x_1+b)^2}{bx_1^2}\right)^{\bar{\psi}} = -\frac{bx_1^2}{(x_1+b)^2}.
\end{equation*}
Using the fact that $\bar{\psi}$ commutes with $\psi$, \eqref{eqn:9.14} implies that
\begin{equation}
\label{eqn:9.15} \left(\frac{(x_1+b)^2}{bx_1^2}\right)^\psi= \left(\frac{(x_1+b)^2}{bx_1^2}\right)^{\bar{\psi}}.
\end{equation}
In either case, \eqref{eqn:9.13} and \eqref{eqn:9.15} show that
\begin{align*}
\xi &= \frac{(x_1+b)^2}{bx_1^2}+\left(\frac{(x_1+b)^2}{bx_1^2}\right)^{\bar{\psi}}\\
&= \frac{(x_1+b)^2}{bx_1^2}-\frac{bx_1^2}{(x_1+b)^2} \\
&= \frac{(bx_1^2 + b^2 + 2bx_1 + x_1^2)(-b x_1^2 + b^2 + 2bx_1 + x_1^2)}{bx_1^2 (x_1 + b)^2}
\end{align*}
is fixed by $\psi$ and therefore lies in $\Sigma \cap L$.  We compute the following polynomial satisfied by $\xi$.  First, note that $\frac{(x_1+b)^2}{bx_1^2}$ is a root of the resultant
\begin{align*}
\textrm{Res}_{x_1}(g(x_1), bx_1^2 X - (x_1 + b)^2) &= b^9(b X^4 + (6b - 1) X^3 + 9bX^2\\
&  - (b^2 + 6b)X + b)\\
 &= b^9g_1(X,b).
 \end{align*}
Since $b$ is a unit, it is clear that $\frac{(x_1+b)^2}{bx_1^2}$ is also a unit and $\xi$ is an algebraic integer.  Now let 
\begin{align*}
F(X) &= -g_1(X,b)g_1\left(X,\frac{-1}{b}\right)\\
&= X^8 + (12 + z)X^7 + (53 + 6z)X^6 + (96 + 8z)X^5 + (-z^2 - 12z + 9)X^4\\
& \ \  - (96 + 8z)X^3 + (53 + 6z)X^2 - (12 + z)X + 1,
\end{align*}
where $z = b - \frac{1}{b} \in \Sigma$.  This polynomial also has $\frac{(x_1+b)^2}{bx_1^2}$ as a root and has coefficients in $\Sigma$.  Since this polynomial is reverse reciprocal in $X$, we can write it as a polynomial in $X-\frac{1}{X}$:
\begin{align*}
F(X) &= X^4 G\left(X-\frac{1}{X},z\right),\\
G(X,z) &= X^4 + (12 + z)X^3 + (57 + 6z)X^2 + (132 + 11z)X - z^2 + 117.
\end{align*}
Now the polynomial $G(X,z)$ has $X=\xi$ as a root.  We write $G(X,z)$ as a polynomial in $z_1 = z+11$:
$$G(X,z) = -z_1^2 + (X^3 + 6X^2 + 11X + 22)z_1 + X^4 + X^3 - 9X^2 + 11X - 4,$$
so that $z_1 =  z+11$ is a root of
$$G_1(\xi,Z) = -Z^2 + (\xi^3 + 6\xi^2 + 11\xi + 22)Z + \xi^4 + \xi^3 - 9\xi^2 + 11\xi - 4 = 0.$$
Since $\xi$ lies in $\Sigma \cap L$, $z_1$ is at most quadratic over the latter field, i.e., $[\Sigma: \Sigma \cap L] \le 2$.  On the other hand, $z_1^\psi$ must also be a root of this polynomial.  Suppose that $z_1^\psi \neq z_1$.  Then
\begin{equation}
\label{eqn:9.16} z_1 + z_1^\psi = \xi^3 + 6\xi^2 + 11\xi + 22.
\end{equation}
But $z_1 = z + 11 \cong \wp_5'^3$ implies, since $\psi$ fixes $K$ and therefore $\wp_5'$, that $z_1^\psi \cong  \wp_5'^3$, as well.  Hence, we have that
$$\xi^3 + 6\xi^2 + 11\xi + 22 \equiv 0 \ (\textrm{mod} \ \wp_5'^3);$$
and $G_1(\xi, z_1) = 0$ implies that its constant term satisfies
$$\xi^4 + \xi^3 - 9\xi^2 + 11\xi - 4 \equiv 0  \ (\textrm{mod} \ \wp_5'^3).$$
Now we argue $5$-adically.  The unique root of $X^3 + 6X^2 + 11X + 22 = 0$ in $\mathbb{Q}_5$ is
$$\alpha = 1 + 2\cdot 5 + 2 \cdot 5^2 + 4 \cdot 5^6 + 2 \cdot 5^7 + 2\cdot5^8 + \cdots,$$
since
$$X^3 + 6X^2 + 11X + 22 \equiv (X^2 + 2X + 3)(X + 4) \ (\textrm{mod} \ 5).$$
 On the other hand, 
$$\xi^4 + \xi^3 - 9\xi^2 + 11\xi - 4 = (\xi+4)(\xi-1)^3 \equiv (\xi+4)^4 \ (\textrm{mod} \ \wp_5'),$$
and we conclude that $\xi \equiv 1$ modulo each prime divisor of $\wp_5'$ in $\Sigma$ and therefore $\xi \equiv 1$ (mod $\wp_5')$.  Now the expansion of $\alpha$ shows that the 
unique root of $f(X) = X^3 + 6X^2 + 11X + 22$ modulo $\wp_5'^3$ which is congruent to $1$ mod $\wp_5'$ is $X \equiv 61$.  For example, we have
$$f(x)-f(y) = (x-y)(x^2 + xy + y^2 + 6x + 6y + 11),$$
so since $\xi \equiv 1$ (mod $\wp_5'$), we have
$$0 \equiv f(\xi)-f(61) = (\xi-61)(\xi^2 + 67\xi + 4098) \ (\textrm{mod} \ \wp_5'^3),$$
where $\xi^2 + 67\xi + 4098 \equiv 1$ (mod $\wp_5'$) is relatively prime to $\wp_5'$.  Thus we must have
$$\xi \equiv 61 \ (\textrm{mod} \ \wp_5'^3).$$
But now consider the equation
$$0 = G_1(\xi,z_1) = -z_1^2 + (\xi^3 + 6\xi^2 + 11\xi + 22)z_1 + \xi^4 + \xi^3 - 9\xi^2 + 11\xi - 4.$$
The quadratic and linear terms in $z_1$ on the right are each divisible by $\wp_5'^6$.  This gives that
\begin{equation}
\label{eqn:9.17} \xi^4 + \xi^3 - 9\xi^2 + 11\xi - 4 \equiv 0 \ (\textrm{mod} \ \wp_5'^6).
\end{equation}
However, $\xi \equiv 61 + h$ (mod $\wp_5'^6$), where $\wp_5'^3 \mid h$.  Setting
$$k(x) = x^4 + x^3 - 9x^2 + 11x - 4 = (x+4)(x-1)^3,$$
we have
$$k(x+h) = h^4 + (4x + 1)h^3 + (6x^2 + 3x - 9)h^2 + (4x^3 + 3x^2 - 18x + 11)h + k(x),$$
where the coefficient of $h$ is
\begin{equation*}
a_3 = 4x^3 + 3x^2 - 18x + 11 \equiv (4x + 1)(x + 4)^2 \ (\textrm{mod} \ 5).
\end{equation*}
Thus,
\begin{align*}
k(\xi) \equiv k(61+h) &\equiv (4\cdot 61^3 + 3\cdot 61^2 - 18\cdot 61 + 11)h + k(61)\\
&\equiv (2^4 \cdot 3^3 \cdot 5^3 \cdot 17)h + k(61)\\
&\equiv k(61) = 2^6 \cdot 3^3 \cdot 5^4 \cdot 13 \ (\textrm{mod} \ \wp_5'^6),
\end{align*}
which is not zero modulo $\wp_5'^6$, contradicting \eqref{eqn:9.17}! \medskip

This contradiction shows that $z^\psi = z$ and therefore $j(\mathfrak{k}') = j(\mathfrak{k})^\psi = j(\mathfrak{k})$, since $j(\mathfrak{k})$ is a rational function of $z$:
$$j(\mathfrak{k}) = -\frac{(z^2+12z+16)^3}{z+11}.$$
(See \cite[p. 54]{m2}.)  This completes the proof of Sugawara's conjecture in this case.  It is immaterial whether we take the ideal $\wp_3$ or $\wp_3'$ in this argument, since we have only argued $5$-adically.  Thus, we have proved

\begin{thm}
If $\mathfrak{m} = \wp_3\wp_5'$ or $\wp_3' \wp_5'$, where $d_K \equiv 1, 4$ (mod $15$), then the ray class field $K_\mathfrak{m}$ is generated by a single $\tau$-invariant for the ideal $\mathfrak{m}$.
\label{thm:13}
\end{thm}

It remains to consider the cases when $(3) = \wp_3^2$ or $(5) = \wp_5^2$ in $K$.  In case $(5) = \wp_5^2$, we appeal to the result of \cite[Prop. 5.1, pp. 124-125]{m3} and its proof, which applies to any fundamental discriminant $d_K$ divisible by $5$, and which shows that $b = r^5(w_i/5)$ (denoted $\rho_i, i = 1, 2$ in \cite{m3}) has the same properties with respect to the curve $E_5(b)$ as in the cases discussed above.  Namely: $E_5(b)$ has complex multiplication by $R_K$, the ring of integers in $K = \mathbb{Q}(\sqrt{-d})$; $\textsf{K}_{\wp_5} = K(b)$ and $\psi: b \rightarrow -1/b$ is the non-trivial automorphism of the quadratic extension $\textsf{K}_{\wp_5}/\Sigma$; and $b-1/b+11 \cong \wp_5^3$ in $\Sigma$.  The same arguments above apply, up to and including the first sentence after equation \eqref{eqn:9.16} (with $\wp_5' = \wp_5$).  In this case the final argument does not lead to a contradiction, since $5^4 \cong \wp_5^8$.  However, the constant term of $G_1(X,z)$ gives that
$$-k(\xi) = -(\xi+4)(\xi-1)^3 = z_1 z_1^\psi \cong \wp_5^6.$$
Let $\bar{\mathfrak{p}} \mid \wp_5$ be a prime divisor of $\wp_5$ in $\Sigma$.  The ideal $\bar{\mathfrak{p}}^6$ exactly divides $\wp_5^6$, since $\wp_5$ is unramified in $\Sigma/K$.  Now $\bar{\mathfrak{p}}$ divides at least one of $\xi+4$ and $\xi-1$ and therefore both, since $(\xi+4) - (\xi-1) = 5$.  Furthermore, $\bar{\mathfrak{p}}^2$ divides at least one of these factors -- otherwise, the left side would only be divisible by $\bar{\mathfrak{p}}^4$.  But $\bar{\mathfrak{p}}^2 \mid 5$, so $\bar{\mathfrak{p}}^2$ must divide both factors, which leads to a contradiction, since then the left side would be divisible by $\bar{\mathfrak{p}}^8$.  Hence, $z_1^\psi = z_1$ and the same conclusion follows.  \medskip

Finally, suppose that  $(3) = \wp_3^2$ in $K$.  Here,
$$[\textsf{K}_{(3)}:\Sigma] = \frac{\varphi(\wp_3^2)}{2} = 3,$$
so that $\textsf{K}_{(3)}$ is a cubic extension of $\Sigma$.  In this case, there are three $\tau$-invariants corresponding to $\mathfrak{m} = \wp_3^2$ and one invariant corresponding to $\mathfrak{m} = \wp_3$.  Thus, one of the four $\tau$-invariants lies in $\Sigma$ and the other three generate $\textsf{K}_{(3)}$ over $\Sigma$.  I claim that $\rho^3 = z+11$, where $\rho$ generates $\textsf{K}_{(3)}$ over $\Sigma$.  If $\rho$ were an element of $\Sigma$, then since $\omega \in \Sigma$, the square-roots $\sqrt{-\theta_i}$ would be quadratic over $\textsf{K}_{\wp_5'}$ (or $\textsf{K}_{\wp_5}$, if $5 \mid d_K$), and their sums  could not generate the cubic extension $\textsf{K}_{\wp_5'} \textsf{K}_{(3)}$ of $\textsf{K}_{\wp_5'}$. \medskip

Let $\sigma$ be the generator of $\textrm{Gal}(\textsf{K}_{(3)}/\Sigma)$, for which $\rho^\sigma = \omega \rho$, and define
$$\sqrt{-\theta_1}^\sigma = \sqrt{-\theta_2}, \ \ \sqrt{-\theta_2}^\sigma = \sqrt{-\theta_3}.$$
Then $\sqrt{-\theta_3}^\sigma = \sqrt{-\theta_1}^{\sigma^3} = \sqrt{-\theta_1}$. Extending $\sigma$ to $\textsf{K}_{\wp_5'}\textsf{K}_{(3)}$ so that it fixes $\textsf{K}_{\wp_5'}$, it follows that
$$\tau(\mathfrak{k}_1^*) = \frac{\lambda}{2} \textrm{Tr}_{\textsf{K}_{\wp_5'}}(\sqrt{-\theta_1}) \in \textsf{K}_{\wp_5'},$$
which implies that $\tau(\mathfrak{k}_1^*)$ cannot generate $\textsf{K}_{(3)}$ and must therefore lie in $\Sigma$.  Now the rest of the calculations are the same, since they only involve $\tau(\mathfrak{k}_1^*) \in \Sigma$ and $x_1 \in \textsf{K}_{\wp_5'}$.  Hence, we have proved the following. \bigskip

\begin{thm}
If $\mathfrak{m} = \wp_3\wp_5$, where $\wp_3, \wp_5$ are first degree prime divisors of $3, 5$ in $K$, then the ray class field $K_\mathfrak{m}$ is generated by a single $\tau$-invariant for the ideal $\mathfrak{m}$.
\label{thm:14}
\end{thm}

Before proceeding to the next case we consider the following example. \medskip

\noindent {\bf Example.} Let $K = \mathbb{Q}(\sqrt{-11})$.  Even though this field has class number $1$, it is interesting to use the arithmetic of this section to compute the various $\tau$-invariants.
Using $j(\mathfrak{o}) = -32^3$, we find that $b$ is a root of
$$B(x) = x^4 + 4x^3 +46x^2 -4x+1,$$
and we take
$$b = -1 + \sqrt{-11} + \frac{3 - \sqrt{-11}}{4} \sqrt{-2\sqrt{-11} + 6}.$$
With this value of $b$ we compute the roots of $g(X) = 0$ in $\textsf{K}_{\wp_5'} = K(b) = K(\sqrt{-2\sqrt{-11} + 6})$ to be
\begin{align*}
x_1 &=  - \frac{2}{3} + \frac{\sqrt{-11}}{6}  + \frac{7-\sqrt{-11}}{24}\sqrt{-2\sqrt{-11} + 6},\\
x_2 &= - \frac{5}{6} -\frac{\sqrt{-11}}{3}  + \frac{1+\sqrt{-11}}{8} \sqrt{-2\sqrt{-11}+ 6}. 
\end{align*}
Now $\tilde{x}_1,\tilde{x}_2$ are roots of the polynomial
\begin{align*}
h_1(X) &= X^2 + \left(- \frac{31}{4} - \frac{7\sqrt{-11}}{4} + \frac{13+7\sqrt{-11}}{8} \sqrt{-2\sqrt{-11} + 6}\right)X\\
& \ \ - \frac{63}{2} - 15\sqrt{-11}  + \frac{23+47\sqrt{-11}}{8} \sqrt{-2\sqrt{-11} + 6};
\end{align*}
while $\tilde{x}_3,\tilde{x}_4$ are roots of
\begin{align*}
h_2(X) &= X^2 + \left(\frac{-15+\sqrt{-11}}{4} + \frac{11+\sqrt{-11}}{8} \sqrt{-2\sqrt{-11} + 6}\right)X\\
& \ \ - \frac{11+5\sqrt{-11}}{4} + \frac{1+2\sqrt{-11}}{4} \sqrt{-2\sqrt{-11} + 6}.
\end{align*}
Note that the polynomials $h_1, h_2$ are not conjugate over $K$, but the polynomials
\begin{align*}
\tilde{h}_1(X) &= \lambda^2 h_1\left(\frac{X}{\lambda} -\frac{b^2+6b+1}{12}\right)\\
&= X^2 + \left(- 19712 - 16128\sqrt{-11} + 5376\sqrt{-11} \sqrt{-2\sqrt{-11} + 6}\right)X\\
& \ \ - 2684616704 - 1271660544\sqrt{-11} + 423886848\sqrt{-11} \sqrt{-2\sqrt{-11} + 6}
\end{align*}
and
\begin{align*}
\tilde{h}_2(X) &= \lambda^2 h_2\left(\frac{X}{\lambda} -\frac{b^2+6b+1}{12}\right)\\
&= X^2 + \left(- 19712 - 16128\sqrt{-11} - 5376\sqrt{-11} \sqrt{-2\sqrt{-11} + 6}\right)X\\
& \ \ - 2684616704 - 1271660544\sqrt{-11} - 423886848\sqrt{-11} \sqrt{-2\sqrt{-11} + 6}
\end{align*}
are conjugate over $K$, and their product is the polynomial
\begin{align*}
T_\mathfrak{m}(X,\mathfrak{k}) &= X^4 + (- 39424 -32256\sqrt{-11})X^3\\
& \ \  + (- 5934415872-2543321088\sqrt{-11})X^2\\
& \ \  + (- 44563506921472+36461051117568\sqrt{-11})X\\
& \ \  + 1277724870452445184 + 2874880958518001664\sqrt{-11}.
\end{align*}
This question of conjugacy is the source of the difficulty in proving Sugawara's conjecture.  The roots of $T_\mathfrak{m}(X, \mathfrak{k})$ are the $\tau$-invariants $\tau(\bar{\mathfrak{k}}_i^*)$ for the ideal $\mathfrak{m} = \wp_3 \wp_5'$ and the four ray classes mod $\mathfrak{m}$.  Note in this case that
\begin{equation*}
z = -2 + 2\sqrt{-11}, \ \ \rho = \frac{-3+\sqrt{-11}}{2}, \ \ \rho^2+2\rho+5 = \frac{3-\sqrt{-11}}{2},
\end{equation*}
where $(\rho) = \wp_5'$ in $K$.  (In Theorem \ref{thm:11}, the value $i = 0$, so $\rho = -\frac{\eta(w/5)^2}{\eta(w)^2}$, where $w = (33+\sqrt{-11})/2$.)  Taking the product of $T_\mathfrak{m}(X, \mathfrak{k})$ with its image under complex conjugation yields the interesting polynomial
\begin{align*}
T(X) &= T_{\wp_3\wp_5'}(X,\mathfrak{k}) T_{\wp_3'\wp_5}(X,\mathfrak{k})\\
& = X^8 - (2^{10} \cdot 7 \cdot 11) X^7 + (2^{21} \cdot 7^2 \cdot 11) X^6 + (2^{30} \cdot 7^5 \cdot 11^2) X^5\\
& \ \  + (2^{40} \cdot 7^4 \cdot 11^2 \cdot 271) X^4 - (2^{50} \cdot 5 \cdot 7^5 \cdot 11^3 \cdot 29) X^3\\
& \ \ - (2^{60} \cdot 7^6 \cdot 11^3 \cdot 883) X^2 + (2^{71} \cdot 7^8 \cdot 11^5) X\\
& \ \ + 2^{80} \cdot 7^8 \cdot 11^4 \cdot 907.
\end{align*}
The polynomial $T(X)$ is irreducible over $\mathbb{Q}$, so that a root of $T_{\wp_3\wp_5'}(X,\mathfrak{k})$ even generates $\textsf{K}_{\wp_3\wp_5'}$ over $\mathbb{Q}$.  The Galois group of $T(X)$ over $\mathbb{Q}$ has order $32$ and is isomorphic to the wreath product $\mathbb{Z}_4 \wr \mathbb{Z}_2$.  If we were to replace $x_1$ with the root $x_2$ of $g(X)$ we would obtain the polynomial $\tilde{T}(X) = T_{\wp_3'\wp_5'}(X,\mathfrak{k}) T_{\wp_3 \wp_5}(X,\mathfrak{k})$.

\section{The case $\mathfrak{m} = \wp_3 \wp_3' \wp_5$.}
\label{sec:10}

We can use the same points $P_i, Q_j$ and the calculations from Section \ref{sec:9} to deal with the case $\mathfrak{m} = \wp_3 \wp_3' \wp_5'$. \medskip

By the results of Section \ref{sec:5} and the arguments preceding Theorem \ref{thm:11}, the invariants $\tau(\mathfrak{k}_3^*)$ and $\tau(\mathfrak{k}_4^*)$ generate $\textsf{K}_{(3)} = \Sigma(\omega)$ and
$$\wp_3 \wp_3' P_3 = \wp_3 \wp_3' P_4 = O \ \ \textrm{on} \ E_5(b).$$
The points $P_3, P_4$ are primitive $(3)$-division points; otherwise the associated $\tau$-invariants would lie in $\Sigma$, which we showed in Section \ref{sec:9} is not the case.  Hence, the points $P_i + Q_j$, for $i = 3,4$ and $1 \le j \le 4$, are primitive $\mathfrak{m}$-division points.  We also know that $\varphi(\mathfrak{m}) = 16$, so that
$$[\textsf{K}_\mathfrak{m}: \Sigma] = 8;$$
$\textsf{K}_\mathfrak{m}$ is quadratic over $\textsf{K}_\mathfrak{\wp_3 \wp_5'}$ and quartic over $\textsf{K}_{(3)} = \Sigma(\omega)$.
The $X$-coordinates of the points $P_3+ Q_j$ are obtained by replacing $x_1$ by $x_3$ in the formulas for $\tilde{x}_j$ in Section \ref{sec:9}, where $P_3 = (x_3, y_3)$:
\begin{align*}
\tilde{x}_5 &= X(P_3 + Q_1) = \frac{-by_3}{x_3^2},\\
\tilde{x}_6 &= X(P_3 + Q_2) = \frac{by_3+b^2}{x_3^2}+\frac{b+b^2}{x_3},\\
\tilde{x}_7 &= X(P_3 + Q_3) = \frac{b^2y_3-bx_3^2-b^2x_3}{(x_3+b)^2},\\
\tilde{x}_8 &= X(P_3 + Q_4) = -\frac{b^2y_3+bx_3^2+(b^3+2b^2)x_3+b^3}{(x_3+b)^2}.
\end{align*}
The coordinates $X(P_4+Q_j)$ are formed by replacing $x_3$ by $x_4$ in these formulas.  See the discussion just before Theorem \ref{thm:11}.  Now the quantity
$$\tau(\bar{\mathfrak{k}}_5^*) = \lambda\left(\tilde{x}_5+\frac{b^2+6b+1}{12}\right) = \lambda\left( \frac{-by_3}{x_3^2}+\frac{b^2+6b+1}{12}\right)$$
must generate $\textsf{K}_\mathfrak{m}$ over $\Sigma$, while the quantity we considered in Section \ref{sec:9},
$$\tau(\mathfrak{k}_3^*) = \lambda\left(x_3+\frac{b^2+6b+1}{12}\right) = \frac{\lambda}{2}\left(-\sqrt{-\theta_1}-\sqrt{-\theta_2}+\sqrt{-\theta_3} \right),$$
generates $\textsf{K}_{(3)} = \Sigma(\omega)$ over $\Sigma$. \medskip
It follows that $x_3 \in \textsf{K}_{\wp_5'}(\omega)$ and $\textsf{K}_\mathfrak{m} = \Sigma(b, x_3, y_3)$.  Putting $F = \textsf{K}_{\wp_5'}$, 
the trace of $\tilde{x}_5$ to $F(\omega) = \textsf{K}_{\wp_5'}(\omega)$ is
$$\textrm{Tr}_{\textsf{K}_\mathfrak{m}/F(\omega)}(\tilde{x}_5) = \tilde{x}_5+\tilde{x}_6 = \frac{(b^2+b)x_3+b^2}{x_3^2},$$
since
$$\tilde{x}_5 \tilde{x}_6 = -\frac{b^2(x_3+b)}{x_3^2},$$
as in Section \ref{sec:9}, so that $\tilde{x}_5, \tilde{x}_6$ satisfy a quadratic polynomial with coefficients in $F(\omega)$:
$$h_3(X) = X^2 - \frac{(b^2+b)x_3+b^2}{x_3^2} X  -\frac{b^2(x_3+b)}{x_3^2}.$$
Hence,
$$\textrm{Tr}_{\textsf{K}_\mathfrak{m}/F}(\tilde{x}_5) =  \frac{(b^2+b)x_3+b^2}{x_3^2} + \frac{(b^2+b)x_4+b^2}{x_4^2},$$
since $\tau(\mathfrak{k}_3^*)$ and $\tau(\mathfrak{k}_4^*)$ are conjugates over $F = \textsf{K}_{\wp_5'}$.
\medskip
Similarly, $\tilde{x}_7, \tilde{x}_8$ are roots of the polynomial
$$h_4(X) = X^2 + \frac{b(2x_3^2 + (b^2+3b) x_3+b ^2)}{(x_3 + b)^2} X + \frac{b^2x_3}{x_3 + b},$$
which is irreducible over $F(\omega)$.  Hence, setting
\begin{align*}
s_1 &= \tilde{x}_5 + \tilde{x}_6 = \frac{b^2}{x_3^2} + \frac{b+b^2}{x_3},\\
s_2 &= \tilde{x}_7 + \tilde{x}_8 = -\frac{b(2x_3^2+(b^2+3b)x_3+b^2)}{(x_3+b)^2} ,
\end{align*}
we have
\begin{align}
\label{eqn:10.1} s_1 - s_2 &= \tilde{x}_5 + \tilde{x}_6 -  \tilde{x}_7 - \tilde{x}_8\\
\label{eqn:10.2} &= b\frac{(b^2 + 1)x_3^3 + (3b^2 + 3b)x_3^2 + (3b^3 + 3b^2)x_3 + b^3}{3x_3^2(x_3+b)^2} = u(x_3,b),
\end{align}
on reducing the numerator modulo $g(x_3)$.  With the automorphism $\bar{\psi}: b \rightarrow -1/b$ from Section \ref{sec:9}, extended to $\textsf{K}_{\wp_5'}(\omega)$ to fix $\omega$, a similar calculation to the displayed lines following \eqref{eqn:9.14}, using that $\tau = \tau(\mathfrak{k}_3^*) \in \textsf{K}_{(3)} = \Sigma(\omega)$, shows that
$$x_3^{\bar{\psi}} = \frac{1}{b^2}(x_3+b).$$
Furthermore, with
\begin{equation*}
\tau(\bar{\mathfrak{k}}_5^*) + \tau(\bar{\mathfrak{k}}_6^*)  = \lambda \left(s_1+\frac{b^2+6b+1}{6}\right),
\end{equation*}
we have
\begin{align*}
\left(\tau(\bar{\mathfrak{k}}_5^*) + \tau(\bar{\mathfrak{k}}_6^*)\right)^{\bar{\psi}} &= \lambda^{\bar{\psi}}  \left(s_1+\frac{b^2+6b+1}{6}\right)^{\bar{\psi}}\\
&= \lambda b^2\left(s_1^{\bar{\psi}} + \frac{b^2-6b+1}{6b^2}\right) \\
&= \lambda\left(b^2 \frac{(1-b)x_3+b}{(x_3+b)^2} + \frac{b^2-6b+1}{6}\right)\\
&= \lambda\left(\frac{(b^2-b^3)x_3+b^3-2b(x_3+b)^2}{(x_3+b)^2} + \frac{b^2+6b+1}{6}\right)\\
&= \lambda \left(-\frac{b(2x_3^2+(b^2+3b)x_3+b^2)}{(x_3+b)^2} + \frac{b^2+6b+1}{6}\right)\\
&= \lambda \left(s_2 + \frac{b^2+6b+1}{6}\right)\\
&= \tau(\bar{\mathfrak{k}}_7^*) + \tau(\bar{\mathfrak{k}}_8^*).
\end{align*}
Therefore, the traces to $\textsf{K}_{\wp_5'}(\omega)$ of the $\tau$-invariants for $\mathfrak{m}$ are
$$S = \{\lambda (s_1 + \beta), \ \lambda (s_2 + \beta), \ \lambda (s_3 + \beta), \ \lambda (s_4 + \beta)\},$$
where
\begin{equation*}
s_3 = s_1^\alpha, \ s_4 = s_2^\alpha,\ \ \alpha: (b, \omega) \rightarrow (b, \omega^2); \ \ \beta = \frac{b^2+6b+1}{6}.
\end{equation*}
The traces in $S$ are the respective images of $\lambda (s_1 + \beta)$ by the automorphisms in
$$\textrm{Gal}(\textsf{K}_{\wp_5'}(\omega)/\Sigma) = \{1, \bar{\psi}, \alpha, \bar{\psi} \alpha \}.$$

\begin{lem}
The conjugates over $\Sigma$ of the trace $\tau(\bar{\mathfrak{k}}_5)+\tau(\bar{\mathfrak{k}}_6) = \lambda(s_1+\beta)$ are distinct.
\label{lem:2}
\end{lem}

\begin{proof}
If $\lambda(s_1+\beta) = \lambda(s_3 + \beta)$, then $s_1 = s_3$, or
\begin{align*}
0 = s_1-s_3 &= \frac{(b^2 + b)x_3 + b^2}{x_3^2} - \frac{(b^2 + b)x_4 + b^2}{x_4^2}\\
&= \frac{-b(x_3 - x_4)[(b+1)x_3 x_4 + b(x_3 + x_4)]}{x_3^2 x_4^2}.
\end{align*}
Now the roots of $g(X)$ are distinct, since
$$\textrm{disc}(g(X)) = -27b^{10} (b^2 + 11b - 1)^2 = -27 \Delta^2.$$
Hence, $x_3 \neq x_4$.  It follows that $(b+1)x_3 x_4 + b(x_3 + x_4) = 0$.  But
\begin{align*}
\textrm{Res}_{x_4}(\textrm{Res}_{x_3}&((b+1)x_3 x_4 + b(x_3 + x_4),g(x_3)),g(x_4))\\
&= b^{28} (b^4 + 2b^3 - 32b^2 + 14b - 1)(b^4 + 3b^3 + 2b^2 - 10b + 1)^2.
\end{align*}
However, the discriminants of these quartics are $2^{12} \cdot 19 \cdot 103$ and $-7^3 \cdot 701$, neither of which is divisible by $5$.  Therefore, their product cannot be $0$ and $s_1 \neq s_3$.  From $(s_2-s_4)^{\bar{\psi}} = \frac{1}{b^2}(s_1 - s_3)$ we also know $s_2 \neq s_4$.  (See the proof of Lemma \ref{lem:3} below.) \medskip

Now suppose that $s_1 = s_2$.  Then
\begin{align*}
0 = s_1 - s_2 &= \frac{(b^2 + b)x_3 + b^2}{x_3^2} + \frac{b(2x_3^2+(b^2+3b)x_3+b^2)}{(x_3+b)^2}\\
&= \frac{b}{x_3^2(x_3+b)^2}\\
& \ \  \times \left(2x_3^4 + (b^2 + 4b + 1)x_3^3 + (3b^2 + 3b)x_3^2+ (b^3 + 3b^2)x_3 + b^3\right).
\end{align*}
In this case we have
\begin{align*}
\textrm{Res}_{x_3}( 2x_3^4 &+ (b^2 + 4b + 1)x_3^3 + (3b^2 + 3b)x_3^2+ (b^3 + 3b^2)x_3 + b^3,g(x_3))\\
&= b^{14}(b^2 + b - 1)(b^2 + 11b - 1).
\end{align*}
Since the roots of $x^2+x-1$ are real and $K \subset \mathbb{Q}(b)$, this polynomial cannot have $b$ as a root.  Hence $s_1 \neq s_2$.  Applying $\alpha$ gives that $s_3 \neq s_4$.  \medskip

Finally, suppose $s_1 = s_4$.  In this case
\begin{align*}
0 &= s_1 - s_4 = \frac{(b^2 + b)x_3 + b^2}{x_3^2} + \frac{b(2x_4^2+(b^2+3b)x_4+b^2)}{(x_4+b)^2}\\
&= \frac{b}{x_3^2 (x_4+b)^2}\\
& \ \ \times \big((2x_4^2 + (b^2 + 3b)x_4 + b^2)x_3^2 + ((b + 1)x_4^2 + (2b^2 + 2b)x_4 + b^3 + b^2)x_3\\
& \ \  + b^3 + 2b^2x_4 + bx_4^2\big)\\
&= \frac{b}{x_3^2 (x_4+b)^2}k(x_3,x_4).
\end{align*}
We compute that
\begin{align*}
&\textrm{Res}_{x_4}(\textrm{Res}_{x_3}(k(x_3,x_4),g(x_3)),g(x_4))\\
&= b^{55}(b^2 + b - 1)(b^2 + 11b - 1)^3\\
& \ \ \times (b^{10} + 15b^9 + 47b^8 + 44b^7 + 1014b^6 + 58b^5 - 1014b^4 + 44b^3\\
& \ \ - 47b^2 + 15b - 1).
\end{align*}
The tenth degree factor $L(b)$ has a nonsolvable Galois group, since $L(b) = b^5G\left(b-\frac{1}{b}\right)$, where
$$G(x) = x^5 + 15x^4 + 52x^3 + 104x^2 + 1160x + 176$$
has Galois group $S_5$.  It follows that $s_1 \neq s_4$ and $s_2 \neq s_3$.  This proves the lemma.
\end{proof}

\begin{lem}
The cross-ratio $\displaystyle \kappa = (s_1,s_2; s_3, s_4) = \frac{(s_1-s_3)(s_2-s_4)}{(s_1-s_4)(s_2-s_3)}$ lies in the Hilbert class field $\Sigma$.  We also have
$$s_1-s_2 + s_3-s_4 = \eta b, \ \ \eta \in \Sigma.$$
\label{lem:3}
\end{lem}

\begin{proof}
We compute that
\begin{equation*}
\kappa^\alpha = (s_3, s_4; s_1, s_2) = \frac{(s_3-s_1)(s_4-s_2)}{(s_3-s_2)(s_4-s_1)} = \kappa.
\end{equation*}
Using $\lambda^{\bar{\psi}} = b^2 \lambda$, we find that
\begin{align*}
(s_1-s_3)^{\bar{\psi}} &= \left(\frac{\lambda s_1 - \lambda s_3}{\lambda}\right)^{\bar{\psi}}\\
&= \frac{\lambda s_2 - \lambda s_4}{b^2 \lambda} = \frac{1}{b^2}(s_2-s_4).
\end{align*}
The same idea yields that
\begin{equation*}
(s_2-s_4)^{\bar{\psi}} = \frac{1}{b^2}(s_1 - s_3), \ \ (s_1-s_4)^{\bar{\psi}} = \frac{1}{b^2}(s_2 - s_3), \ \ (s_2-s_3)^{\bar{\psi}} = \frac{1}{b^2}(s_1 - s_4).
\end{equation*}
It follows that
$$\kappa^{\bar{\psi}} = \frac{b^{-4}(s_2-s_4)(s_1-s_3)}{b^{-4}(s_2-s_3)(s_1-s_4)} = \kappa.$$
Thus, $\kappa$ is fixed by $\textrm{Gal}(\textsf{K}_{\wp_5'}(\omega)/\Sigma)$.  This proves the first assertion.  The second follows from the calculation
\begin{align*}
\eta^{\bar{\psi}} = \left(\frac{s_1-s_2 + s_3-s_4}{b}\right)^{\bar{\psi}} &= \frac{\frac{-1}{b^2}(s_1-s_2)+\frac{-1}{b^2}(s_3-s_4)}{\frac{-1}{b}}\\
&= \frac{s_1-s_2 + s_3-s_4}{b} = \eta,
\end{align*}
together with the fact that the quantity $\eta$ is obviously fixed by the automorphism $\alpha$.  Hence, $\eta \in \Sigma$.
\end{proof}

\begin{lem}
We have that $\rho^4 - 2\rho^3 - \rho^2 - 10\rho + 25 = \nu^2 \rho^2$, where $\nu \in \Sigma$ is an algebraic integer.
\label{lem:4}
\end{lem}

\begin{proof}
On one hand, we have
\begin{align}
\notag \sqrt{-\theta_2} \sqrt{-\theta_3} &= \frac{b^2}{36} (\omega^2 \rho^2 -\omega \rho-1)(\omega \rho^2 -\omega^2 \rho-1)\\
\notag & \ \ \times \sqrt{(\omega^2 \rho^2 +2\omega \rho+5)(\omega \rho^2 +2\omega^2 \rho+5)}\\
\label{eqn:10.3} &= \frac{b^2}{36} (\rho^4 + \rho^3 + 2\rho^2 - \rho + 1) \sqrt{\rho^4 - 2\rho^3 - \rho^2 - 10\rho + 25}.
\end{align}
We also have
\begin{align*}
\sqrt{-\theta_2} \sqrt{-\theta_3} &= \frac{g_3}{\sqrt{-\theta_1}}\\
&= \frac{-b^2}{36} \left(b+\frac{1}{b}\right) \frac{z^2+18z+76}{(\rho^2-\rho-1) \sqrt{\rho^2+2\rho+5}}\\
&= \frac{-b^2}{36} \frac{z^2+18z+76}{\rho^2-\rho-1} \frac{\sqrt{z^2+4}}{\sqrt{\rho^2+2\rho+5}}.
\end{align*}
The last equation and \eqref{eqn:10.3} give that
\begin{align*}
& (\rho^4 + \rho^3 + 2\rho^2 - \rho + 1) \sqrt{\rho^4 - 2\rho^3 - \rho^2 - 10\rho + 25}\\
&= -\frac{z^2+18z+76}{\rho^2-\rho-1} \frac{\sqrt{z^2+4}}{\sqrt{\rho^2+2\rho+5}}\\
&= -\frac{(\rho^2 - \rho - 1)(\rho^4 + \rho^3 + 2\rho^2 - \rho + 1)}{\rho^2-\rho-1} \frac{\sqrt{z^2+4}}{\sqrt{\rho^2+2\rho+5}}\\
&= -(\rho^4 + \rho^3 + 2\rho^2 - \rho + 1) \frac{\sqrt{z^2+4}}{\sqrt{\rho^2+2\rho+5}}.
\end{align*}
Hence,
$$\sqrt{\rho^4 - 2\rho^3 - \rho^2 - 10\rho + 25} = -\frac{\sqrt{z^2+4}}{\sqrt{\rho^2+2\rho+5}}.$$
Now we know that both square roots on the right side of this equation are Kummer elements for $\textsf{K}_{\wp_5'}/\Sigma$, so the quotient lies in $\Sigma$.
Hence,
$$\rho^4 - 2\rho^3 - \rho^2 - 10\rho + 25 = \gamma^2, \ \ \gamma \in \Sigma.$$
But $\rho \cong \wp_5'$ also lies in $\Sigma$ and $\rho^2$ divides the left side of this equation.  It follows that $\gamma = \nu \rho$, with $\nu \in R_\Sigma$.  This proves the lemma.
\end{proof}

\newtheorem{cor}{Corollary}

\begin{cor}
For the quantity $\nu$ in Lemma \ref{lem:3} we have
$$\nu^2 + 12 = \left(\rho-1+\frac{5}{\rho}\right)^2.$$
\label{cor:1}
\end{cor}

\begin{proof}
This follows directly from
$$\nu^2 = \frac{\rho^4 - 2\rho^3 - \rho^2 - 10\rho + 25}{\rho^2} = \frac{(\rho^2-\rho+5)^2-12\rho^2}{\rho^2}.$$
\end{proof}

\begin{lem}
If $5 \nmid d_K$, then $\mathbb{Q}(\rho) = \Sigma$; while if $5 \mid d_K$, then $K(\rho) = \Sigma$.
\label{lem:5}
\end{lem}

\begin{proof}
This is clear from $z = \rho^3-11$ and $j = -\frac{(z^2+12z+16)^3}{z+11}$, using that $\mathbb{Q}(z) = \Sigma$ from \cite[Prop. 3.2]{m2} 
when $5 \nmid d_K$; and $K(z) = \Sigma$ from \cite[Prop. 5.1]{m3} when $5 \mid d_K$.
\end{proof}

\begin{lem} 
We have: \smallskip

\noindent (a) For any root $x$ of $g(X) = 0$,
$$\frac{1}{x} = \frac{-1}{b^3} (3x^3 + (b^2 + 6b + 1)x^2 + (3b^2 + 3b)x + 3b^2)$$
and
$$\frac{1}{x+b} = \frac{1}{b^5}(3x^3 + (b^2 + 3b + 1)x^2 + (-b^3 + 2b)x + b^4 + b^2).$$

\noindent (b) \begin{align*}
s_1 = f_1(x_3) = \frac{(b^2+b)x_3+b^2}{x_3^2} &= \frac{-3(b - 2)}{b^2}x_3^3 - \frac{(b^3 + 4b^2 - 8b - 2)}{b^2} x_3^2\\
& \ \ - \frac{(4b^2 + 3b - 5)}{b}x_3 - 6b + 3.
\end{align*}

\noindent (c) \begin{align*}
s_2 = f_2(x_3) &= \frac{-b(2x_3^2 + (b^2+3b) x_3 + b^2)}{(x_3+b)^2}\\
& = \frac{3(2b + 1)}{b^3} x_3^3 + \frac{(2b^3 + 10b^2 + 5b + 1)}{b^3}x_3^2 - \frac{(b^3 + b^2 - 5b - 2)}{b^2}x_3\\
& \ \  - \frac{(b^2 - 3b - 1)}{b}.
\end{align*}
\label{lem:6}
\end{lem}

\begin{proof}
This follows by straightforward calculation, using the formula in (a) for $x=x_3$.
\end{proof}

Note that the formulas for $s_3 = f_1(x_4)$ and $s_4 = f_2(x_4)$ are the same as the formulas in (b) and (c), with $x_3$ replaced by $x_4$.

\begin{prop}
\begin{align}
\label{eqn:10.4} \frac{s_1-s_2-s_3+s_4}{(x_3-x_4)} &= \frac{-(\rho + 1) \sqrt{z^2 + 4}}{2} + \frac{(\rho^3 - 2\rho - 7) \sqrt{\rho^2 + 2\rho + 5}}{2}\\
\label{eqn:10.5} &= \frac{\big(\rho^3 + \nu \rho^2 + (\nu - 2)\rho-7\big) \sqrt{\rho^2 + 2\rho + 5}}{2}.
\end{align}
\label{prop:1}
\end{prop}

\begin{proof}
Write $s_1 = f_1(x_3), s_3 = f_1(x_4), s_2 = f_2(x_3), s_4 = f_2(x_4)$.  We start by noting the polynomial identity
\begin{align*}
s_1 -s_2 -s_3+s_4 &= f_1(x_3) -f_2(x_3)-f_1(x_4)+f_2(x_4)\\
&=\frac{-(x_3-x_4)}{b^3}[(x_3 + x_4 + 3)b^4 + (6x_3 + 6x_4 + 2)b^3\\
& \ \ + (3x_3^2 + 3x_4^2 + 2x_3+ 2x_4+ 3x_3x_4)b^2\\
& \ \ + (3x_3+3x_4 + 2)b + 3x_3^2 + 3x_4^2 + x_3+x_4 + 3x_3x_4]
\end{align*}
Now substitute
$$x_3 + x_4 = -\frac{b^2+6b+1}{6}-\sqrt{-\theta_1}, \ \ x_3^2+x_4^2 = (x_3 + x_4)^2 -2x_3x_4$$
into the last expression.  This yields
\begin{align}
\notag s_1 -s_2 &-s_3+s_4 = \frac{-(x_3-x_4)}{b^3} 3b\sqrt{-\theta_1}\\
\label{eqn:10.6} &  +\frac{(x_3-x_4)}{b^3} \frac{(b^2 + 1)(b^4 + 12b^3 + 2b^2 - 18b  + 1 + 36x_3x_4  + 36\theta_1)}{12}.
\end{align}
Putting $36\theta_1 = -b^2(z^2+12z+16+12\rho)$ from (\ref{eqn:9.5}) with $z = b-\frac{1}{b}$ in (\ref{eqn:10.6}) gives
\begin{equation}
\label{eqn:10.7} s_1 -s_2 -s_3+s_4 = \frac{-(x_3-x_4)}{b^3}\left( 3b\sqrt{-\theta_1} + (b^2 + 1)\big((\rho + 1)b^2 + \frac{b}{2} - 3x_3 x_4)\big)\right).
\end{equation}
Now use the result of the calculation
\begin{align*}
x_3 x_4 &= \left(-\frac{b^2+6b+1}{12}-\frac{1}{2} \sqrt{-\theta_1} + \frac{(\sqrt{-\theta_2}-\sqrt{-\theta_3})}{2}\right)\\
& \ \ \times \left(-\frac{b^2+6b+1}{12}-\frac{1}{2} \sqrt{-\theta_1} - \frac{(\sqrt{-\theta_2}-\sqrt{-\theta_3})}{2}\right)\\
&= \left(\frac{b^2+6b+1}{12}\right)^2+\frac{b^2+6b+1}{12} \sqrt{-\theta_1} - \frac{\theta_1}{4} -\frac{(\sqrt{-\theta_2}-\sqrt{-\theta_3})^2}{4}\\
&= \left(\frac{b^2+6b+1}{12}\right)^2+\frac{b^2+6b+1}{12} \sqrt{-\theta_1} - \frac{\theta_1}{4} +\frac{\theta_2+\theta_3+2\sqrt{-\theta_2}\sqrt{-\theta_3}}{4},
\end{align*}
with
$$\theta_2+\theta_3 = \frac{-b^2}{36}(2z^2+24z+32+12\omega \rho+12\omega^2 \rho) =  \frac{-b^2}{36}(2z^2+24z+32-12\rho)$$
in \eqref{eqn:10.7}, giving
\begin{align*}
&s_1 -s_2 -s_3+s_4 =\frac{-(x_3-x_4)}{b^3} \\
& \ \ \times \left(\frac{(-b^4 - 6b^3 - 2b^2 + 6b - 1)\sqrt{-\theta_1}}{4} + \frac{(b^2 + 1)(-3\sqrt{-\theta_2}\sqrt{-\theta_3} + (\rho + 1)b^2)}{2}\right)\\
&= (x_3-x_4)\left(\frac{(b^4 + 6b^3 + 2b^2 - 6b + 1)\sqrt{-\theta_1}}{4b^3} - \frac{(b^2 + 1)(-\frac{3\sqrt{-\theta_2}\sqrt{-\theta_3}}{b^2} + (\rho + 1))}{2b}\right).
\end{align*}
Now we use
\begin{align*}
\frac{(b^4 + 6b^3 + 2b^2 - 6b + 1)}{b^2} &= z^2+6z+4, \ \ b+\frac{1}{b} = \sqrt{z^2+4},\\
\sqrt{-\theta_1} &= \frac{b}{6}(\rho^2-\rho-1)\sqrt{\rho^2+2\rho+5},\\
\frac{-3\sqrt{-\theta_2} \sqrt{-\theta_3}}{b^2} &= \frac{(\rho^4+\rho^3+2\rho^2-\rho+1)\sqrt{z^2+4}}{12\sqrt{\rho^2+2\rho+5}},
\end{align*}
(the last from the proof of Lemma \ref{lem:4}) to write the factor in the large parenthesis as
\begin{align*}
\textsf{A} &= -\frac{2(\rho + 1) \sqrt{z^2 + 4}\sqrt{\rho^2 + 2\rho + 5} -z( \rho^4 + \rho^3 + 2\rho^2) }{4\sqrt{\rho^2 + 2\rho + 5}}\\
& \ \ \ + \frac{(z^2 + 7z + 4)\rho + z^2 + 5z + 4}{4\sqrt{\rho^2 + 2\rho + 5}}.
\end{align*}
The final terms in the numerator of $\textsf{A}$ factor on setting $z = \rho^3-11$:
$$ -z( \rho^4 + \rho^3 + 2\rho^2) + (z^2 + 7z + 4)\rho + z^2 + 5z + 4 = -2(\rho^3 - 2\rho - 7)(\rho^2 + 2\rho + 5),$$
giving
\begin{align*}
\textsf{A} &= -\frac{2(\rho + 1) \sqrt{z^2 + 4}\sqrt{\rho^2 + 2\rho + 5} -2(\rho^3 - 2\rho - 7)(\rho^2 + 2\rho + 5)}{4\sqrt{\rho^2 + 2\rho + 5}}\\
&= \frac{-(\rho + 1) \sqrt{z^2 + 4} + (\rho^3 - 2\rho - 7)\sqrt{\rho^2 + 2\rho + 5}}{2}.
\end{align*}
This proves \eqref{eqn:10.4}.  Equation \eqref{eqn:10.5} follows from \eqref{eqn:10.4} on using
$$\sqrt{z^2+4} = -\sqrt{\rho^2 + 2\rho + 5}\sqrt{\rho^4 - 2\rho^3 - \rho^2 - 10\rho + 25} = -\nu \rho \sqrt{\rho^2 + 2\rho + 5}$$
by Lemma \ref{lem:4} and its proof.
\end{proof}

\begin{cor}
The quantity $s_1-s_2-s_3+s_4 \neq 0$.
\label{cor:2}
\end{cor}

\begin{proof}
This follows from \eqref{eqn:10.4}, since
$$(\rho+1)^2((\rho^3-11)^2+4)-(\rho^3-2\rho-7)^2(\rho^2+2\rho+5) = 12(\rho - 2)(\rho^2 + 2\rho + 5),$$
where $\rho \neq 2$ by Lemma \ref{lem:5} and $\rho^2 + 2\rho + 5 \neq 0$ by the remark after \eqref{eqn:9.16}.
\end{proof}

A similar calculation leads to the following result.

\begin{prop}
\begin{align}
\notag &\frac{s_1+s_2-s_3-s_4}{(x_3-x_4)} = \\
\label{eqn:10.8} &  \frac{-(\rho^4 + 3\rho^3 - 15\rho - 23)\sqrt{\rho^2 + 2\rho + 5} + (\rho^2 + 4\rho + 5)\sqrt{z^2 + 4}}{2\sqrt{\rho^2 + 2\rho + 5}}\\
\label{eqn:10.9} &= - \frac{(\rho^4 + 3\rho^3 - 15\rho - 23) + (\rho^2 + 4\rho + 5)\nu\rho}{2}.
\end{align}
\label{prop:2}
\end{prop}

\begin{proof}
This follows from
\begin{align*}
s_1+s_2-s_3-s_4 &= f_1(x_3)+f_2(x_3)-f_1(x_4) - f_2(x_4)\\
&= \frac{-(x_3-x_4)}{b^3} \textsf{B}_1,\\
\textsf{B}_1 &= (x_3 + x_4+5)b^4 + (2x_3 + 2x_4 + 4)b^3\\
& \ \  + (3x_3^2 + 3x_4^2 -18x_3  - 18x_4 + 3x_3x_4  - 10)b^2\\
& \ \  + (-12x_3^2 - 12x_4^2 -7x_3 - 7x_4 - 12x_3x_4  - 2)b\\
& \ \  - 3x_3^2 - 3x_4^2 - x_3 - x_4 - 3x_3x_4,
\end{align*}
using the same substitutions as in the proof of Proposition \ref{prop:1}.  However, in place of 
$$\frac{(b^4 + 6b^3 + 2b^2 - 6b + 1)}{b^2} = z^2+6z+4$$
we use
$$\frac{b^4 + 2b^3 + 2b - 1}{b^2} = \frac{(b^2+1)(b^2+2b-1)}{b^2} = \sqrt{z^2+4}(z+2).$$
This leads to
$$\textsf{B} = \frac{-1}{b^3} \textsf{B}_1 = \frac{-(\rho^4 + 3\rho^3 - 15\rho - 23)\sqrt{\rho^2 + 2\rho + 5} + (\rho^2 + 4\rho + 5)\sqrt{z^2 + 4}}{2\sqrt{\rho^2 + 2\rho + 5}}.$$
\end{proof}

Now we find a formula for the quotient of the sums in Propositions \ref{prop:1} and \ref{prop:2}.  A straightforward calculation establishes the following result.

\begin{thm}
If $\textsf{A}$ and $\textsf{B}$ are the right-hand sides of equations \eqref{eqn:10.4} and \eqref{eqn:10.8}, we have
\begin{align}
\label{eqn:10.10} \frac{s_1+s_2-s_3-s_4}{s_1-s_2-s_3+s_4} = \frac{\textsf{B}}{\textsf{A}} &= \frac{(-2\rho + 3)\sqrt{\rho^2 + 2\rho + 5} + \sqrt{z^2 + 4}}{(\rho - 2)(\rho^2 + 2\rho + 5)}\\
\label{eqn:10.11} &= \frac{\big(3-(\nu+2)\rho\big)\sqrt{\rho^2+2\rho+5}}{(\rho-2)(\rho^2+2\rho+5)}.
\end{align}
\label{thm:15}
\end{thm}

\noindent {\bf Remark.} It is clear that the quotient in \eqref{eqn:10.10} is the quotient of sums and differences of the traces of the eight $\tau$-invariants for $\wp_3 \wp_3' \wp_5'$ in $\textsf{K}_{(3)\wp_5'}/\textsf{K}_{\wp_5'}(\omega)$.
\medskip

Now assume that $\psi \in \textrm{Gal}(\textsf{K}_{(3)\wp_5'}/K)$ is an automorphism fixing the set of $\tau$-invariants for $\wp_3 \wp_3' \wp_5'$.  Then the set of traces of these invariants to $\textsf{K}_{\wp_5'}(\omega)$ is also fixed, and
$$\{\lambda(s_1+\beta),\lambda(s_2+\beta),\lambda(s_3+\beta),\lambda(s_4+\beta)\}^\psi$$
is a permutation of
$$S =\{\lambda(s_1+\beta),\lambda(s_2+\beta),\lambda(s_3+\beta),\lambda(s_4+\beta)\}.$$
This permutation must respect the orbits of $S$ under the automorphisms $\alpha, \bar{\psi}, \alpha \bar{\psi}$, since they commute with $\psi$.  The numerator on the left side of \eqref{eqn:10.10} is the difference of the orbit sums of $S$ under $\bar{\psi}$ divided by $\lambda$, while the denominator is the difference of orbit sums under $\alpha \bar{\psi}$ divided by $\lambda$.  Hence, the quotient is preserved up to sign, and we have the relation
$$\left(\frac{\big(3-(\nu+2)\rho\big)\sqrt{\rho^2+2\rho+5}}{(\rho-2)(\rho^2+2\rho+5)}\right)^\psi = \pm \frac{\big(3-(\nu+2)\rho\big)\sqrt{\rho^2+2\rho+5}}{(\rho-2)(\rho^2+2\rho+5)}.$$
Now $(\rho) = \wp_5'$ (see the line just before \eqref{eqn:9.5}) implies that $\rho^\psi = a \rho$, for some unit $a \in \Sigma$.  Furthermore, $\sqrt{\rho^2+2\rho+5}^\psi = \varepsilon \sqrt{\rho^2+2\rho+5}$, with $\varepsilon \in \Sigma$, since both square-roots are Kummer elements for $\textsf{K}_{\wp_5'}/\Sigma$.  This implies the relation
$$\frac{\big(3-(\nu^\psi+2)a\rho\big)}{(a\rho-2)} \frac{1}{\varepsilon}= \pm \frac{\big(3-(\nu+2)\rho\big)}{(\rho-2)}.$$
This implies that $\varepsilon \equiv \pm 1$ mod $(\rho)$, i.e., mod $\wp_5'$.  On the other hand, we also have
$$(\rho^2+2\rho+5)^\psi = a^2 \rho^2 + 2a \rho+5 = \varepsilon^2(\rho^2+2\rho+5),$$
or
$$(a^2-\varepsilon^2)\rho^2 + 2(a-\varepsilon^2)\rho + 5-5\varepsilon^2 = 0.$$
Dividing this equation by $\rho$ gives
$$(a^2-\varepsilon^2)\rho + 2(a-\varepsilon^2) + 5\frac{1-\varepsilon^2}{\rho} = 0,$$
hence
$$a \equiv \varepsilon^2 \equiv 1 \ (\textrm{mod} \ (\rho))$$
and
\begin{equation}
\label{eqn:10.12} \rho^\psi = a \rho \equiv \rho \ (\textrm{mod} \ (\rho)^2).
\end{equation}
Now we appeal to the formula
$$j(\mathfrak{k}) = j = -\frac{(z^2+12z+16)^3}{z+11}.$$
(See \cite[p. 1180]{m2}.)  Putting $z = \rho^3 -11$ implies that
$$j(\mathfrak{k}) = j = -\frac{(\rho^6-10\rho^3+5)^3}{\rho^3} = -\left(\rho^5-10\rho^2+\frac{5}{\rho}\right)^3.$$
Since $5/\rho$ is an algebraic integer, this gives further that
\begin{equation}
\label{eqn:10.13} j^\psi - j = -\left(\rho^{5\psi}-10\rho^{2\psi} +\frac{5}{\rho^\psi} - \big(\rho^5 -10\rho^2+\frac{5}{\rho}\big)\right) \Xi,
\end{equation}
for some algebraic integer $\Xi$.  Now by \eqref{eqn:10.12},
$$\rho^{5\psi} \equiv \rho^5, \ \ \rho^{2\psi} \equiv \rho^2 \ \ (\textrm{mod} \ (\rho)^2),$$
and
$$\frac{5}{\rho^\psi}-\frac{5}{\rho} = 5\frac{\rho-\rho^\psi}{\rho \rho^\psi} \equiv 0 \ (\textrm{mod} \ (\rho)),$$
since $\rho \rho^\psi \cong \rho^2$.  Thus, \eqref{eqn:10.13} yields that
$$j(\mathfrak{k})^\psi \equiv j(\mathfrak{k}) \  (\textrm{mod} \ \wp_5');$$
which implies $\psi |_\Sigma = 1$ when $5 \nmid d_K$, since the discriminant of $H_{d_K}(X)$ is not divisible by $5$ in this case.  (See \cite{d1a}.). This proves that $j(\mathfrak{k})^\psi = j(\mathfrak{k})$ and the polynomials $T_\mathfrak{m}(X,\mathfrak{k})$ are distinct for different ideal classes $\mathfrak{k}$. \medskip

This proves:

\begin{thm}
If $K = \mathbb{Q}(\sqrt{d})$, where $d  = d_k \equiv 1, 4$ (mod $15$), then Sugawara's conjecture holds for $K$ and the ideal $\mathfrak{m} = \wp_3 \wp_3' \wp_5'$, i.e., $\textsf{K}_\mathfrak{m}$ is generated over $K$ by a single $\tau$-invariant for the ideal $\mathfrak{m}$.
\label{thm:16}
\end{thm}

Note that the ideal $\mathfrak{m} = \wp_3^2 \wp_5$ satisfies Sugawara's condition \eqref{eqn:3}, so there is no need to consider fields in which $3$ ramifies.

\subsection{Proof of the conjecture when $(5) = \wp_5^2$.}

All the above arguments apply in the case that $(5) = \wp_5^2$, except for the final conclusion.  In this case, namely, when $5 \mid d_K$, the conjugates of $j(\mathfrak{k})$ can be congruent to each other modulo prime divisors of $\wp_5$.  To handle this case we set
$$R(\rho,\nu) = \frac{\big(3-(\nu+2)\rho\big)^2}{(\rho-2)^2(\rho^2+2\rho+5)}.$$
This is the square of the expression in \eqref{eqn:10.11}.   Then we must show:
\begin{equation}
\label{eqn:10.14} \psi \in \textrm{Gal}(\textsf{K}_{\mathfrak{m}}/K) \ \wedge \ R(\rho,\nu)^\psi = R(\rho,\nu) \ \Rightarrow \ \psi |_\Sigma = 1.
\end{equation}
As before, we know $\rho^\psi = a \rho$ for some unit $a$ and $(\rho^2+2\rho+5)^\psi = \varepsilon^2 (\rho^2+2\rho+5)$, where $\varepsilon \in \Sigma$ and
$$a \equiv 1, \ \ \varepsilon \equiv \pm 1 \ (\textrm{mod} \ (\rho)).$$
In this case, dividing the equation
\begin{equation*}
a^2\rho^2+2a\rho+5 = \varepsilon^2(\rho^2+2\rho+5)
\end{equation*}
by $\rho$ yields that
$$(a^2 -\varepsilon^2)\rho+2(a-\varepsilon^2) + 5\frac{(1-\varepsilon^2)}{\rho} = 0,$$
and implies the congruence
$$a \equiv \varepsilon^2 \ (\textrm{mod} \ \rho^2).$$

\begin{prop} There are at most $2$ automorphisms $\psi \in \textrm{Gal}(\Sigma/K)$ for which $R(\rho,\nu)^\psi = R(\rho, \nu)$.
\label{prop:3}
\end{prop}

\begin{proof}
Expanding $R(\rho,\nu)$ and using Corollary \ref{cor:1} to Lemma \ref{lem:4} gives that
\begin{align*}
r = R(\rho,\nu) &= \frac{\big(3-(\nu+2)\rho\big)^2}{(\rho-2)^2(\rho^2+2\rho+5)}\\
&= \frac{2(2\rho - 3)\rho \nu}{(\rho - 2)^2 (\rho^2 + 2\rho + 5)} + \frac{\rho^4 - 2\rho^3 + 3\rho^2 - 22\rho + 34}{(\rho - 2)^2(\rho^2 + 2\rho + 5)}.
\end{align*}
Squaring and using Corollary \ref{cor:1} again shows that
\begin{align*}
r^2 &= R(\rho, \nu)^2 = \frac{4(\rho^4 - 2\rho^3 + 3\rho^2 - 22\rho + 34)\rho(2\rho - 3)\nu}{(\rho - 2)^4(\rho^2 + 2\rho + 5)^2}\\
& \ \ + \frac{\rho^8 - 4\rho^7 + 26\rho^6 - 136\rho^5 + 281\rho^4 - 452\rho^3 + 1532\rho^2 - 3056\rho + 2056}{(\rho - 2)^4 (\rho^2 + 2\rho + 5)^2}.
\end{align*}
Solving for $\nu$ in terms of $r$ in the first equation yields that
$$\nu = \frac{(\rho^4 - 2\rho^3 + \rho^2 - 12\rho + 20)r}{2(2\rho - 3)\rho} + \frac{-\rho^4 + 2\rho^3 - 3\rho^2 + 22\rho - 34}{2(2\rho - 3)\rho}.$$
Now substitute for $\nu$ in the expression for $r^2$:
\begin{align*}
&r^2 = \frac{(2\rho^6 + 8\rho^4 - 52\rho^3 + 10\rho^2 - 84\rho + 340)r}{(\rho^2 + 2\rho + 5)^2 (\rho - 2)^2}\\
& \ \  + \frac{-\rho^6 + 10\rho^4 + 16\rho^3 - 25\rho^2 - 80\rho - 64}{(\rho^2 + 2\rho + 5)^2 (\rho - 2)^2}\\
&= \frac{2(\rho^2 + 2\rho + 5)(\rho^4 - 2\rho^3 + 3\rho^2 - 22\rho + 34)r}{(\rho^2 + 2\rho + 5)^2 (\rho - 2)^2}-\frac{(\rho^3 - 5\rho - 8)^2}{(\rho^2 + 2\rho + 5)^2 (\rho - 2)^2}.
\end{align*}
Rearranging gives the polynomial identity
\begin{align}
\notag 0 = F(r,\rho) &= (\rho^2 + 2\rho + 5)^2 (\rho - 2)^2 r^2\\
\notag & \ \ - 2(\rho^2 + 2\rho + 5)(\rho^4 - 2\rho^3 + 3\rho^2 - 22\rho + 34)r\\
\label{eqn:10.15} & \ \ + (\rho^3 - 5\rho - 8)^2,
\end{align}
which can also be computed using the resultant, with respect to $\nu$, of the above polynomial defining $r$ in terms of $\nu$ and $\rho$ 
and the result of Corollary \ref{cor:1} to Lemma \ref{lem:4}.  If $ r = 1$, then a computation using the above formula for $\nu$ above gives
$$\nu = \frac{-\rho^2+5\rho-7}{(2\rho-3)\rho},$$
and substituting yields
$$0 = R\left(\rho, -\frac{\rho^2 - 5\rho + 7}{\rho(2\rho - 3)}\right) - 1 = -4\frac{\rho^4 - \rho^3 - \rho^2 - 9\rho + 11}{(\rho^2 + 2\rho + 5)(2\rho - 3)^2}.$$
Thus, $\rho$ would satisfy the irreducible equation
$$\rho^4 - \rho^3 - \rho^2 - 9\rho + 11 = 0.$$
But this is impossible, since $(\rho) = \wp_5 \nmid 11$.  \medskip

Rewriting equation \eqref{eqn:10.15} as a polynomial in $\rho$ gives us an equation $F(r,\rho)  = 0$ satisfied by $\rho$ over the field $K(r)$, 
where $r = R(\rho,\nu)$ and
\begin{align}
\notag F(r,\rho) &=  (r^2 - 2r + 1)\rho^6 + (2r^2 - 8r - 10)\rho^4 + (-20r^2 + 52r - 16)\rho^3\\
\notag & \ \ \ + (r^2 - 10r + 25)\rho^2 + (-20r^2 + 84r + 80)\rho + 100r^2 - 340r + 64.\\
\label{eqn:10.16}
\end{align}
Dividing through in \eqref{eqn:10.16} by the leading coefficient gives an equation with constant term
$$c = \frac{100r^2-340r+64}{(r-1)^2} = \frac{4(5r - 1)(5r - 16)}{(r-1)^2}.$$
Since
$$r = \frac{\big(3-(\nu+2)\rho\big)^2}{(\rho-2)^2(\rho^2+2\rho+5)} \cong \frac{\mathfrak{a}}{\mathfrak{b}\wp_5},$$
where $\mathfrak{a}, \mathfrak{b}$ are integral ideals which are relatively prime to $\wp_5$, it is clear that $c = \frac{4(5r - 1)(5r - 16)}{(r-1)^2}$ is exactly divisible by $\wp_5^2 \cong \rho^2$.  Furthermore, the coefficients of the scaled equation are integral for $\wp_5$, and its roots are therefore integral for $\wp_5$.  Also, since the conjugates of $\rho$ are $a_\psi \rho \cong \rho$, there can be at most two such conjugates which are roots of \eqref{eqn:10.16}.  Hence, there can be at most two automorphisms $1, \psi$, for which $\rho^\psi$ is a root of \eqref{eqn:10.16}.  Therefore, either $\psi = 1$ is the only such automorphism; or $\psi^2 = 1$ and $\Sigma$ is quadratic over $K(r)$.  This proves the proposition.
\end{proof}

Assume that $R(\rho,\nu)^\psi = R(\rho, \nu)$, where $\psi$ has order $2$.  Then $\rho^\psi = a \rho$ implies that $\rho^{\psi^2} = a^\psi a \rho = \rho$, giving that $a^\psi = \frac{1}{a}$.  \medskip

From the proofs of Proposition 5.1 and Theorem 5.3 in \cite[pp. 124-129]{m3} (which are also valid for fundamental discriminants $d_K = -5d$ in place of the discriminants $-5l$ considered in that paper), we know that, for some $w \in R_K$
$$\lambda = \left(\frac{\eta(w/5)}{\eta(w)}\right)^6 = -\rho^3$$
is conjugate to the value
$$\lambda^{\sigma_{\wp_5}} = \frac{5^3}{\lambda}$$
over $K$, where $\sigma_{\wp_5} = \left(\frac{\Sigma/K}{\wp_5}\right)$.  (See Theorem \ref{thm:11} in Section \ref{sec:9}.)  Since the cube roots of unity are not in $\Sigma$, this implies that
\begin{equation*}
\rho^{\sigma_{\wp_5}} = \frac{5}{\rho}.
\end{equation*}
Now the automorphism $\sigma = \sigma_{\wp_5}$ acts on $a = \frac{\rho^\psi}{\rho}$ by
\begin{equation}
\label{eqn:10.17} a^\sigma = \frac{(5/\rho)^\psi}{5/\rho} = \frac{\rho}{\rho^\psi} = \frac{1}{a}.
\end{equation}
But we also have $a^\psi =  \frac{1}{a}$ from above.  Now there are two cases.  \medskip

\noindent {\it Case 1}:  Assume that $\psi = \sigma = \sigma_{\wp_5}$.  This implies, using the notation of \eqref{eqn:10.16}, that
$$0 = F(r,\rho)^\psi = F(r, \rho^{\sigma}) = F\left(r, \frac{5}{\rho}\right).$$
Thus, $\rho$ must be a root of the resultant
\begin{align*}
\textrm{Res}_r&\big(F(r,\rho), \rho^6F\left(r, \frac{5}{\rho}\right)\big)\\
&= 256(\rho^2 - 5)^2 (\rho^2 + 2\rho + 5)^4 (11\rho^4 - 14\rho^3 - 86\rho^2 - 70\rho + 275)\\
& \times (11\rho^8 - 164\rho^7 + 829\rho^6 - 1980\rho^5 + 3685\rho^4 - 9900\rho^3 + 20725\rho^2\\
& \ \ - 20500\rho + 6875).
\end{align*}
But $\rho$ cannot be a root of the irreducible quartic or octic factor, since $\rho$ is an algebraic integer.  Further, $\rho^2+2\rho+5 \neq 0$, from the remark following \eqref{eqn:9.6}.  Thus, $\rho^2-5 = 0$.  A simple computation shows this is only possible for the discriminant $d_K = -20$: namely, in this case
\begin{align*}
&\textrm{Res}_\rho(\rho^2-5,z+11-\rho^3) = z^2+22z -4,\\
&\textrm{Res}_z(z^2+22z-4,(z+11)j+(z^2+12z+16)^3)\\
& \ \ \  = -5^3(j^2-1264000j-681472000) = -5^3 H_{-20}(j).
\end{align*}
Hence, $K = \mathbb{Q}(\sqrt{-5})$ and $\Sigma = K(\sqrt{5}) = K(\sqrt{-1})$; since $\wp_5$ splits in $\Sigma$, we have $\sigma_{\wp_5} = \left(\frac{\Sigma/K}{\wp_5}\right) = 1$, so $\psi = \sigma_{\wp_5} = 1$.  Notice that for $\rho = \sqrt{5}$ the nontrivial automorphism $\psi$ of $\Sigma/K$ satisfies $a = \frac{\rho^\psi}{\rho} = -1$, which is certainly not congruent to $1$ (mod $\rho$). \medskip

\noindent {\it Case 2:}  Assume that $\psi \neq \sigma_{\wp_5}$ fixes $r = R(\rho,\nu)$, so that $F(r,\rho)= 0$, as in Case 1.  Then $\psi$ also fixes
$$s = r^{\sigma} = R(\rho^\sigma,\nu^\sigma), \ \ \sigma = \sigma_{\wp_5}.$$
By Corollary \ref{cor:1} to Lemma \ref{lem:4} and the fact that $\rho^\sigma=\frac{5}{\rho}$, it is clear that $\sigma$ fixes $\nu^2$, so that
$\nu^\sigma = \pm \nu$. \medskip

Assuming first that $\nu^\sigma = \nu$, we obtain that
$$s = R(\rho^\sigma,\nu) = R\left(\frac{5}{\rho},\nu \right) = \frac{(3\rho - 5\nu - 10)^2\rho^2}{5(-5 + 2\rho)^2(\rho^2 + 2\rho + 5)}.$$
Proceeding as we did in Proposition \ref{prop:3} with $R(\rho,\nu)$, we compute that
\begin{align*}
0 = F_2(s,\rho) &= (100s^2 - 340s + 64)\rho^6 + (-100s^2 + 420s + 400)\rho^5\\
\notag & \ \  + (25s^2 - 250s + 625)\rho^4 + (-2500s^2 + 6500s - 2000)\rho^3\\
\notag & \ \   + (1250s^2 - 5000s - 6250)\rho^2 + 15625s^2 - 31250s + 15625.
\end{align*}
Next, we use Maple to compute the resultant
$$\textrm{Res}_r(F(r,\rho),F(r,a \rho)) = 256 \rho^2 (a-1)^2 A,$$
where
$$A = (a^{10} + 2a^9 + a^8)\rho^{18} + \cdots + 540000$$
is a polynomial in $a$ and $\rho$.  Similarly, a calculation on Maple gives that
$$\textrm{Res}_s(F_2(s,\rho),F_2(s,a \rho)) = 160000 \rho^4 (a-1)^2 B,$$
where
$$B = 21600 a^{10} \rho^{18} + \cdots + 152587890625$$
is also a polynomial in $a$ and $\rho$.  See the Appendix (Section 12) for these polynomials written out in full.  Now we compute
\begin{align*}
\textrm{Res}_a&(A,B) = \rho^{80} (\rho^2+2\rho+5)^{20}(11\rho^4 - 14\rho^3 - 86\rho^2 - 70\rho + 275)\\
& \ \ \times (11\rho^8 - 164\rho^7 + 829\rho^6 - 1980\rho^5 + 3685\rho^4 - 9900\rho^3 + 20725\rho^2\\
& \ \ - 20500\rho + 6875)\\
& \ \ \times (2488869\rho^8 - 9142380\rho^7 - 12557677\rho^6 + 46638270\rho^5 - 3856021\rho^4\\
& \ \  + 233191350\rho^3 - 313941925\rho^2 - 1142797500\rho + 1555543125)\\
& \ \ \times \textsf{P},
\end{align*}
where $\textsf{P}$ is a product of $7$ primitive, irreducible polynomials of degrees $12, 20, 24$:
\begin{align*}
\textsf{P} &= (6543999\rho^{12} - 23871960\rho^{11} - 93034852\rho^{10} + 328442160\rho^9 + 296712789\rho^8\\
& \ \  + 1001170920\rho^7 - 9112922960\rho^6 + 5005854600\rho^5 + 7417819725\rho^4\\
& \ \  + 41055270000\rho^3 - 58146782500\rho^2 - 74599875000\rho + 102249984375)
 \end{align*}
\begin{align*}
& \ \ \times (47989326\rho^{20} - 687672216\rho^{19} + \cdots + 468645761718750)\\
& \ \ \times (18251706\rho^{20} - 636597576\rho^{19} + \cdots + 178239316406250)\\
& \ \ \times (195570225\rho^{20} - 944115600\rho^{19} + \cdots + 1909865478515625)\\
& \ \ \times (1363919525\rho^{20} - 15974832650\rho^{19} + \cdots + 13319526611328125)\\
& \ \ \times (125164944\rho^{24} - 1106740704\rho^{23} + \cdots + 30557847656250000)\\
& \ \ \times (872908496\rho^{24} - 5290763616\rho^{23} + \cdots + 213112425781250000).
\end{align*}
For each factor of this resultant (with degree at least $4$), the leading coefficient does not divide the next coefficient, from which we conclude that $\rho$ cannot be a root of any of these factors.  It follows that $\textrm{Res}_a(A,B) \neq 0$, and therefore at least one of $A$ or $B$ is nonzero, giving that $a = 1$.  Hence, $\psi = 1$ in $\textrm{Gal}(\Sigma/K)$.
\medskip

A similar calculation works if $\nu^\sigma = -\nu$.  In this case, we have that
$$t = R(\rho^\sigma,-\nu) = R\left(\frac{5}{\rho},-\nu \right) = \frac{(3\rho + 5\nu - 10)^2\rho^2}{5(-5 + 2\rho)^2(\rho^2 + 2\rho + 5)}.$$
This leads, as in Proposition \ref{prop:3}, to the equation
\begin{align*}
0 = F_3(t,\rho) &= (100t^2 - 340t + 64)\rho^6 + (-100t^2 + 420t + 400)\rho^5 \\
& \ \  + (25t^2 - 250t + 625)\rho^4 + (-2500t^2 + 6500t - 2000)\rho^3\\
& \ \  + (1250t^2 - 5000t - 6250)\rho^2 + 15625t^2 - 31250t + 15625,
\end{align*}
which is the same as the equation $0 = F_2(s, \rho)$.  This case also leads to $a = 1$ and $\psi = 1$.  \medskip

This completes the proof of \eqref{eqn:10.14} and shows that Theorem \ref{thm:16} also holds in the ramified case, when $(5) = \wp_5^2$ in $K$. \medskip

In connection with the result of this section, I put forward the following conjecture.

\begin{conj}
 Assume $5 \mid d_K$, $(5) = \wp_5^2$ and $1 \neq \psi \in \textrm{Gal}(\Sigma/K)$.  If $\rho^\psi = a \rho$, then $a \not \equiv 1$ mod $\wp_5$ in $\Sigma = \textsf{K}_1$.
 \label{conj:3}
 \end{conj}
 
 It can be shown in the situation of this conjecture that $a = \varepsilon^2$ is a square in $\Sigma$.  The congruence just after equation \eqref{eqn:10.14} shows that Sugawara's conjecture for this case would also follow immediately from Conjecture \ref{conj:3}.

\section{$E_{12}$ and the case $\mathfrak{m} = \wp_2^2 \wp_3$.}
\label{sec:11}
In this section we assume that $(2) = \wp_2 \wp_2'$ splits in $K = \mathbb{Q}(\sqrt{-d})$ and that $(3) = \wp_3 \wp_3'$ or $(3) = \wp_3^2$.  We do not need to consider 
the possibility $(2) = \wp_2^2$, since in that case $\mathfrak{m} = \wp_2^2 \wp_3 = (2)\wp_3$ is not a conductor, as remarked in the introduction.  (See the comments following \eqref{eqn:4}.)  \medskip

We work with the Tate normal form for a point of order $12$, which can be given in the form
\begin{align*}
E_{12}(t) &: \ Y^2+\frac{3t^4 + 4t^3 -2t^2 + 4t -1}{(t-1)^3} XY -\frac{t(t+1)(t^2 + 1)(3t^2 + 1)}{(t-1)^4}Y\\
& \ \ \ = X^3 -\frac{t(t+1)(t^2 + 1)(3t^2 + 1)}{(t-1)^4}X^2.
\end{align*}
Denote the nontrivial coefficients in this equation by
$$a_1 = \frac{3t^4 + 4t^3 -2t^2 + 4t -1}{(t-1)^3}, \ \ a_3 = a_2 = -\frac{t(t+1)(t^2 + 1)(3t^2 + 1)}{(t-1)^4}.$$
For the curve $E_{12}(t)$ we have
\begin{align*}
g_2 &= \frac{(3t^4 + 6t^2 - 1)(3t^{12} + 234t^{10} + 249t^8 + 60t^6 - 27t^4 - 6t^2 - 1)}{12(t - 1)^{12}},\\
g_3 &= \frac{(3t^8 + 24t^6 + 6t^4 - 1)}{216(t - 1)^{18}}\\
& \ \ \times (9t^{16} - 1584t^{14} - 3996t^{12} - 3168t^{10} + 30t^8 + 528t^6 - 12t^4 + 1),\\
\Delta &= \frac{(t+1)^{12}t^6(t^2 + 1)^3(3t^2 + 1)^4(3t^2 - 1)}{(t - 1)^{24}}.
\end{align*}
Its $j$-invariant is
$$j(E_{12}(t)) = \frac{(3t^4 + 6t^2 - 1)^3(3t^{12} + 234t^{10} + 249t^8 + 60t^6 - 27t^4 - 6t^2 - 1)^3}{t^6(t^2 - 1)^{12}(t^2 + 1)^3(3t^2 + 1)^4(3t^2 - 1)}.$$
Note that $j_{12}(t) = j(E_{12}(t))$ is invariant under the map $t^2 \rightarrow A(-t^2)$, where
$$A(x) = \frac{x-1}{3x+1}.$$
If $E_{12}(t)$ has complex multiplication by $R_K$, then $j_{12}(t)$ is an algebraic integer, and putting $t = 1/t'$ shows that $t' = 1/t$ is an algebraic integer. \medskip

Now recall the standard addition formulas from \cite[pp. 53-54]{si1} for the curve $E_{12}(t)$: if $P_1 = (x_1,y_1), P_2 = (x_2, y_2)$ and $x_1 \neq x_2$, then
$$X(P_1+P_2) = \left(\frac{y_2-y_1}{x_2-x_1}\right)^2+a_1\frac{y_2-y_1}{x_2-x_1} -a_2-x_1-x_2;$$
while the doubling formula for the $X$-coordinate is
$$X(2P) = \frac{x^4-a_1a_2x^2-2a_2^2x-a_2^3}{4x^3+(a_1^2+4a_2)x^2+2a_1a_2x+a_2^2}, \ \ P = (x,y).$$
(Note: the formula $b_2=a_1^2+4a_2$ corrects the formula for $b_2$ in \cite[p. 42]{si1}.)  Also,
$$Y(-P) = -Y(P)-a_1X(P)-a_2.$$
Hence, the point $P=(0,0)$ satisfies
$$X(2P) = -a_2 = \frac{t(t+1)(t^2 + 1)(3t^2 + 1)}{(t - 1)^4},$$
with corresponding $Y$-coordinate
$$Y(2P) = -\frac{t^2(t^2 + 1)(t+1)^2(3t^2 + 1)^2}{(t - 1)^7}.$$
The point
$$2P = \left(\frac{t(t+1)(t^2 + 1)(3t^2 + 1)}{(t - 1)^4},-\frac{t^2(t^2 + 1)(t+1)^2(3t^2 + 1)^2}{(t - 1)^7}\right)$$
has order $6$.  Furthermore, the point
$$4P = \left(\frac{(t^2 + 1)t(t+1)}{(t - 1)^2}, -\frac{2t^2(t+1)^2(t^2 + 1)^2}{(t - 1)^5}\right)$$
has order $3$.  Hence,
\begin{equation*}
6P = 2P +4P = \left(\frac{(t+1)(3t^2 + 1)}{4(t - 1)}, -\frac{(t+1)^2(3t^2 + 1)^2}{8(t - 1)^3}\right)
\end{equation*}
has order $2$ (verifying that $E_{12}(t)$ is the Tate normal form).  We also have the points
\begin{align*}
3P &= P + 2P = \left(-\frac{t(t+1)(3t^2 + 1)}{(t - 1)^3},\frac{t^2(t+1)^2 (3t^2 + 1)}{(t - 1)^4}\right),\\
5P &= P + 4P = \left(-\frac{2t(t+1)(t^2 + 1)(3t^2 + 1)}{(t - 1)^5}, \frac{t(t+1)^2(t^2 + 1)(3t^2 + 1)^2}{(t - 1)^7}\right),
\end{align*}
which have order $4$ and $12$, respectively. \medskip

The $\tau$-invariant for $\mathfrak{m} = \wp_2^2 \wp_3$ and a given point $Q$ is given by
$$\tau(\mathfrak{k}^*) = -2^7 3^5 \frac{g_2 g_3}{\Delta}\left(X(Q)+\frac{a_1^2+4a_2}{12}\right).$$
Let $\tau_k$ denote the $\tau$-invariant for the point $kP$.

\begin{lem} Assume that $E_{12}(t)$ has complex multiplication by $R_K$, where $\wp_2, \wp_3$ are first degree prime divisors of $2$ and $3$ in $R_K$, and $P = (0,0)$ is a primitive $\mathfrak{a}$-division point on $E_{12}(t)$ for the ideal $\mathfrak{a} = \wp_2^2 \wp_3$.  The invariants $\tau_2, \tau_3, \tau_4, \tau_6$ lie in $\Sigma$, and $t$ generates $\textsf{K}_{\wp_2^2 \wp_3}$ over $\Sigma$.  Moreover, the invariants $\tau_i$, for $1 \le i \le 6$ are all algebraic integers.
\label{lem:7}
\end{lem}

\begin{proof}
For the ideal $\mathfrak{a} = \wp_2^2 \wp_3$, there are always $\mathfrak{a}$-division points on an elliptic curve $E$ with $R_K$ as its multiplier ring.  This is because, by \cite[p. 42]{d2}, the number of primitive (``echten'') $\mathfrak{a}$-division points is
\begin{align*}
\sum_{\mathfrak{d} \mid \mathfrak{a}}&{\mu(\mathfrak{a}/\mathfrak{d})N(\mathfrak{d})}\\
&= \mu(1) N(\wp_2^2\wp_3) +\mu(\wp_2) N(\wp_2 \wp_3) + \mu(\wp_3) N(\wp_2^2)+\mu(\wp_2 \wp_3) N(\wp_2)\\
&= 12 - 6 - 4 +2 = 4.
\end{align*}
If a point $\tilde P$ is one such division point on $E$, then $\tilde P$ has order $12$ in $E[12]$, so there must be a value of $t$ and an isomorphism $E \rightarrow E_{12}(t)$ for which $\tilde P \rightarrow P = (0,0)$, and $P$ is a primitive $\mathfrak{a}$-division point on $E_{12}(t)$.  (See \cite[p. 250]{m}, where the assignment $y \rightarrow u^3y + vx + d$ in the proof of Lemma \ref{lem:4} should be $y \rightarrow u^3y + vu^2x + d$.)  Now let $\mathcal{K}$ be the elliptic function field defined by $E_{12}(t)$.  Since $N(\mathfrak{a}) = 12 = \textrm{ord}(P)$ and $\mathfrak{a}P = O$, the field $\mathcal{K}^\mathfrak{a}$ is the fixed field of the translation group 
$$\textrm{Gal}(\mathcal{K}/\mathcal{K}^\mathfrak{a}) = \langle \sigma_{\mathfrak{o},\mathfrak{p}} \rangle$$
generated by the translation $\sigma_{\mathfrak{o},\mathfrak{p}}$ taking $\mathfrak{o}$ to the prime divisor $\mathfrak{p}$ of $\mathcal{K}$ corresponding to $P = (0,0)$.  This gives an isogeny $\phi: E_{12} \rightarrow E_{12}^\mathfrak{a}$ for which $\textrm{ker}(\phi) = \{Q \in E_{12}: \ \mathfrak{a}Q = O\} = \langle P \rangle$ is generated by $P$.  (In Deuring's terminology \cite{d1}, $\mathcal{K}^\mathfrak{a}$ has the same multiplier-ring $R_K$ as $\mathcal{K}$, so that $\phi: \ E_{12}(t) \rightarrow E_{12}^\mathfrak{a}$ is a Heegner point on $X_0(12)$.) \medskip

The following containments follow from these considerations:
\begin{equation*}
2P \in \textrm{ker}(\wp_2 \wp_3), \ \ 3P \in \textrm{ker}(\wp_2^2), \ \ 4P \in \textrm{ker}(\wp_3), \ \ 6P \in \textrm{ker}(\wp_2).
\end{equation*}
Since the conductors of each of these ideals is $\mathfrak{f} = 1$, it follows that $\tau_2, \tau_3, \tau_4, \tau_6$ all lie in $\Sigma$.  On the other hand, $\tau_1$ and $\tau_5$ do not lie in $\Sigma$, since the $\tau$-invariants for $\mathfrak{m} = \wp_2^2\wp_3$ must generate the quadratic extension $\textsf{K}_\mathfrak{m}/\Sigma$. \medskip

Note that
\begin{equation*}
\notag \frac{\tau_1-\tau_6}{\tau_4-\tau_6} = -\frac{(3t^2 + 1)(t - 1)}{(1 + t)^3}, \ \ \frac{\tau_2 - \tau_3}{\tau_6 - \tau_3} = \frac{4t^2}{t^2 - 1}.
\end{equation*}
It follows from the above arguments that $t^2 \in \Sigma$ but $t \notin \Sigma$, so that $t$ is a Kummer element for $\textsf{K}_{\wp_2^2 \wp_3}$ over $\Sigma$.  Hence, the nontrivial automorphism of $\textsf{K}_{\wp_2^2 \wp_3}/\Sigma$ is $\rho: t \rightarrow -t$. \medskip

Finally, $\tau_3, \tau_4$, and $\tau_6$ are algebraic integers by the formulas for $F(X,\mathfrak{k})$ in Sections \ref{sec:2}, \ref{sec:3} and \ref{sec:5}, since they correspond to points of orders $4, 3$ and $2$, respectively.  Also, $\tau_1, \tau_2$ and $\tau_5$ are algebraic integers by Hasse's results \cite[eq. (35), p. 134]{h1}, since they are invariants for the ideals $\wp_2^2\wp_3, \wp_2\wp_3$ and $\wp_2^2\wp_3$, respectively.
\end{proof}

\begin{cor} For the value of $t$ in this Lemma, $[K(t):K] = [\textsf{K}_{\wp_2^2 \wp_3}:K] = 2h(d_K)$.
\label{cor:3}
\end{cor}

\begin{proof}
This statement holds because $j(E_{12}(t)) \in K(t)$ and $K(j(E_{12}(t)),t) = \Sigma(t) = \textsf{K}_{\wp_2^2 \wp_3}$.
\end{proof}

Now assume that $\psi$ is an automorphism of $\textsf{K}_{\wp_2^2 \wp_3}/K$ for which $j(\mathfrak{k}') = j(\mathfrak{k})^\psi$ and $\{\tau_1,\tau_5\}^\psi = \{\tau_1,\tau_5\}$.  If the fixed field of $\langle \psi \rangle$ is $L$, and $L \subseteq \Sigma$, with $\psi \neq 1$, then the nontrivial automorphism $\rho: t \rightarrow -t$ of $\textsf{K}_{\wp_2^2 \wp_3}/\Sigma$ lies in $\langle \psi \rangle$. In that case, for some $i$, $\tau_1^{\psi^i} = \tau_1^\rho = \tau_5$.

\begin{lem}
In $\textsf{K}_{\wp_2^2 \wp_3}$ we have $t \cong \mathfrak{q}_3^{-1}$, where $\mathfrak{q}_3^2 = \wp_3$, i.e., $\mathfrak{q}_3$ is the product of the prime divisors of $\textsf{K}_{\wp_2^2 \wp_3}$ lying over $\wp_3$.  The same holds if $(3) = \wp_3^2$.
\label{lem:8}
\end{lem}

\begin{proof} First assume that $(3) = \wp_3 \wp_3'$ in $R_K$.  Put $t = u/\sqrt{3}$ in $j_{12}(t) = j(E_{12}(t))$.  This gives that
$$j_{12}(t) = \frac{(u^4 + 6u^2 - 3)^3(u^{12} + 234u^{10} + 747u^8 + 540u^6 - 729u^4 - 486u^2 - 243)^3}{(u^2 - 3)^{12}u^6(u^2 + 3)^3(u^2 + 1)^4(u^2 - 1)}.$$
Hence, $u$ is an algebraic integer, so that $u^2=3t^2$ is an algebraic integer.  This shows that the denominator of $t$ is divisible at most by prime divisors of $(3)= \mathfrak{q}_3^2 \wp_3'$.  It follows that no prime divisor of $\wp_3'$ can divide this denominator, and furthermore, at most the first power of any prime divisor of $\mathfrak{q}_3$ can divide the denominator.  If $\sigma \in \textrm{Gal}(\textsf{K}_{\wp_2^2 \wp_3}/K)$, then $t^\sigma = t \alpha$, for some $\alpha \in \Sigma$, since $t^\sigma$ is also a Kummer element for $\textsf{K}_{\wp_2^2 \wp_3}/\Sigma$.  But since the reciprocals of $t$ and $t^\sigma$ are algebraic integers, the numerators in their ideal factorizations are both $1$, so that the prime divisors of $t^\sigma/t = \alpha$ occur in the numerator and denominator at most to the first power.  I claim that at least one prime ideal divisor of $\mathfrak{q}_3$ divides the denominator of $t$.  If not, $t^2 \in \Sigma$ would be a unit and the prime divisors of $\wp_3$ would not be ramified in $\Sigma(t)/\Sigma$, which is false.  But since $\alpha \in \Sigma$ is only divisible by (ramified) prime divisors of $\wp_3$, the ideal $(\alpha)$ must be a square in the group of ideals in $\textsf{K}_{\wp_2^2 \wp_3}$.  This can only be the case if $(t^\sigma) = (t)$.  Since the prime divisors of $\mathfrak{q}_3$ are conjugate to each other over $K$, this shows that all of them must divide the denominator of $t$.  If $(3) = \wp_3^2$, then the denominator of $t^2$ is divisible by at most the square of a prime divisor of $\wp_3$ in $\Sigma$.  But if the square of $\mathfrak{q}$ divides the denominator, then $\mathfrak{q}$ would not be ramified in $\textsf{K}_{\wp_2^2 \wp_3}/\Sigma$, since it divides an odd prime.  Hence, exactly the first power of $\mathfrak{q}$ divides the denominator.  The rest of the argument is the same.  This proves the lemma.
\end{proof}

\begin{cor} If $\sigma \in \textrm{Gal}(\textsf{K}_{\wp_2^2 \wp_3}/K)$, then $t^\sigma = \alpha t$, where $\alpha$ is a unit in $\Sigma$.
\label{cor:4}
\end{cor}

Now we use the fact that both
\begin{align*}
\tau_1+\tau_5 &= \frac{2(3t^4 + 6t^2 - 1)(3t^{12} + 234t^{10} + 249t^8 + 60t^6 - 27t^4 - 6t^2 - 1)}{(t^2 - 1)^{12} t^6 (t^2 + 1)^3 (3t^2 + 1)^4 (3t^2 - 1)}\\
& \ \times (3t^8 + 24t^6 + 6t^4 - 1)(3t^8 - 42t^4 - 24t^2 - 1)\\
& \ \times (9t^{16} - 1584t^{14} - 3996t^{12} - 3168t^{10} + 30t^8 + 528t^6 - 12t^4 + 1)\\
&=  \frac{2N_1(t)}{D_1(t)}
\end{align*}
and
\begin{align*}
\frac{\tau_1+\tau_5}{\tau_1-\tau_5} &= -\frac{3t^8 - 42t^4 - 24t^2 - 1}{12t(t ^2- 1)(3t^2 + 1)(t^2 + 1)}\\
&= -\frac{N_2(t)}{12D_2(t)}
\end{align*}
are invariant (up to sign) under $\psi: t \rightarrow \alpha t$.   Assuming first that $\left(\frac{N_2(t)}{D_2(t)}\right)^\psi =  \frac{N_2(t)}{D_2(t)}$, we compute the resultant
\begin{align*}
&R_t(a) = \textrm{Res}_t\left(N_1(a t) D_1(t)-N_1(t)D_1(a t),\frac{1}{t}(N_2(at) D_2(t)-N_2(t)D_2(at))\right)\\
&= 2^{384} 3^{76} (a - 1)^{108} (a + 1)^{50} a^{362} (a^8 - 68a^6 - 570a^4 - 68a^2 + 1)^2\\
& \ \ \times (a^8 - 24a^7 + 28a^6 - 168a^5 + 582a^4 - 168a^3 + 28a^2 - 24a + 1)^2\\
& \ \ \times (a^{16} + 72a^{14} - 1412a^{12} + 3960a^{10} + 11142a^8 + 3960a^6 - 1412a^4 + 72a^2 + 1)^2\\
& \ \ \times (3a^2 - 1)^6(a^2 - 3)^6(a^2 + 3)^8(3a^2 + 1)^8(a^2 + 1)^{12} \times R_{192}(a)^2,
\end{align*}
where $R_{192}(a) \in \mathbb{Z}[a]$ is an irreducible palindromic\footnote{This follows from the fact that $t^{\psi^{-1}} = \alpha^{-\psi^{-1}} t$, so $\alpha^{-\psi^{-1}}$ is a conjugate of $\alpha$, and therefore so is $\alpha^{-1}$, when $R_{192}(\alpha) = 0$.} polynomial (with incompatible factorizations modulo $31$ and $43$) of degree $192$ with leading coefficient $11664 = 2^4 \cdot 3^6$ and coefficient of $a$ equal to $-2970432 = -2^6 \cdot 3^5 \cdot 191$.  It follows that the unit $\alpha$ is a root of $R_t(a)$, but cannot be a root of $R_{192}(a)$ or of any of the factors in the last line of the resultant formula, other than $a^2+1$.  The three nontrivial factors (of degrees $8, 8$ and $16$) are symmetric in $a$, and in these cases $x = \alpha+1/\alpha$ satisfies the polynomial, respectively, given by
\begin{align*}
A_1(x) &= x^4 - 72x^2 - 432 ,\\
A_2(x) &= x^4 - 24x^3 + 24x^2 - 96x + 528,\\
A_3(x) & = x^8 + 64x^6 - 1824x^4 + 10240x^2 + 256.
\end{align*}
The discriminants of $A_1(x)$ and $A_2(x)$ are
$$\textrm{disc}(A_1(x)) = -2^4 \cdot 3^9, \ \ \ \textrm{disc}(A_2(x)) = -2^{36} \cdot 3^3.$$
Since $\alpha+\alpha^{-1}$ lies in $\Sigma$, which is normal over $\mathbb{Q}$, it would follow that $\sqrt{-3} \in \Sigma$ if $\alpha+\alpha^{-1}$ is a root of $A_1(x)$ or $A_2(x)$.  But this is impossible in the case that $(3) = \wp_3 \wp_3'$ is unramified in $K$.  The same argument holds for $(\alpha+\alpha^{-1})^2$ and $A_3(x) = f_3(x^2)$, where $\textrm{disc}(f_3(x)) = 2^{40} \cdot 3^7 \cdot 11^2$ (so that $\sqrt{3} \in \Sigma$).
\medskip

If $(3) = \wp_3^2$, then the Galois groups of $A_1(x)$ and $A_2(x)$ are both $D_4$, and
$$\textrm{Gal}(A_3(x)/\mathbb{Q}) \cong (\mathbb{Z}/2\mathbb{Z})^3 \rtimes D_4$$
is a semi-direct product of order $64$.  The last is impossible, since the normal closure of $\mathbb{Q}(\alpha + \alpha^{-1})/\mathbb{Q}$ is abelian over the quadratic field $K$, and is thus at most quadratic over the root field $\mathbb{Q}(\alpha + \alpha^{-1})$.  Hence, $\alpha+\alpha^{-1}$ cannot be a root of $A_3(x)$.  If it were a root of $A_1(x)$ or $A_2(x)$, then since the normal closure $L$ of $\mathbb{Q}(\alpha+\alpha^{-1}) \subseteq \Sigma$ has Galois group $D_4$, it would have to contain the field $K$; otherwise $L$ and $K$ would be linearly disjoint and $\textrm{Gal}(LK/K) \cong \textrm{Gal}(L/\mathbb{Q})$ would not be abelian.  In that case only $2$ and $3$ can divide the discriminant of $K$.
The only possibilities are then $d_K =-3, -8, -24$.  But $\mathbb{Q}(\sqrt{-3})$ is excluded and $2$ ramifies in $\mathbb{Q}(\sqrt{-2})$ and $\mathbb{Q}(\sqrt{-6})$. 
This shows that $\alpha$ cannot be a root of the $8$-th degree or $16$-th degree polynomials dividing $R_t(a)$.  
\medskip

Now suppose that $\alpha$ is a root of $a^2+1 = 0$.  Then $\alpha = \pm i$ and $\psi(t) = \alpha t$ implies that $\psi(t^2) = -t^2$.  Now computing $\frac{N_2(\alpha t)}{D_2(\alpha t)}$ and setting it equal to $\pm \frac{N_2(t)}{D_2(t)}$, we find that $t$ is a root of the $20$-th degree polynomial
\begin{align*}
P_1(t) &= (-9t^{10} + 3it^8 + 126t^6 + 30it^4 - 21t^2 - i)\\
& \ \ \times (9t^{10} + 3it^8 - 126t^6 + 30it^4 + 21t^2 -i)\\
&= -(81t^{20} - 2259t^{16} + 16434t^{12} - 4398t^8 + 381t^4 + 1) = -P_2(t^4).
\end{align*}
But $\textrm{Gal}(P_2(t)/\mathbb{Q}) \cong S_5$ is not solvable, showing that this is impossible.  Hence $\alpha$ cannot be a root of $a^2+1$.\medskip

On the other hand, if $\left(\frac{N_2(t)}{D_2(t)}\right)^\psi =  -\frac{N_2(t)}{D_2(t)}$, then $N_2(-t) = N_2(t), D_2(-t) = -D_2(t)$ implies that
\begin{align*}
\textrm{Res}_t&\left(N_1(a t) D_1(t)-N_1(t)D_1(a t),\frac{1}{t}(N_2(a t) D_2(t)+N_2(t)D_2(a t))\right)\\
& = R_t(-a)
\end{align*}
is obtained by replacing $a$ by $-a$ in the above resultant, and the argument is the same.  \medskip

Hence, whether $3$ is ramified in $K$ or not, the only possibilities are $\alpha = \pm 1$, in which case $\psi: t \rightarrow \pm t$ fixes $L = \Sigma$.  Hence, we find that $j(\mathfrak{k}') = j(\mathfrak{k})^{\psi} = j(\mathfrak{k})$.  Thus, Sugawara's conjecture holds for these fields and $\mathfrak{m} = \wp_2^2 \wp_3$.

\begin{thm}
If $\mathfrak{m} = \wp_2^2 \wp_3$, where $\wp_2, \wp_3$ are first degree prime ideals in $K$, then $K(\tau_1) = \Sigma_{\wp_2^2 \wp_3}$ and Sugawara's conjecture holds for $K$ and $\mathfrak{m}$.
\label{thm:17}
\end{thm}

This completes the proof of Sugawara's Conjecture for all ideals $\mathfrak{m}$ of $K$ listed in \eqref{eqn:4}.  All other ideals satisfy one of Sugawara's conditions \eqref{eqn:1}-\eqref{eqn:3}, and for these Sugawara's arguments in \cite{su1, su2} apply.  This completes the proof of the Main Theorem.

\section{Appendix.}

The polynomials $A$ and $B$ computed in Section 10.1 are given below in full.  Note the palindromic nature of the coefficients of the powers of $\rho$.

\begin{align*}
A &=  (a^{10} + 2a^9 + a^8)\rho^{18} + (-2a^{10} - 4a^9 - 4a^8 - 2a^7)\rho^{17}\\
& \ + (-a^{10} - 2a^9 - a^8 - 2a^7 - a^6)\rho^{16}\\
& \  + (-16a^{10} - 32a^9 + 8a^8 + 8a^7 - 32a^6 - 16a^5)\rho^{15}\\
& \ + (41a^{10} + 82a^9 + 50a^8 + 42a^7 + 50a^6 + 82a^5 + 41a^4)\rho^{14}\\
& \  + (-4a^{10} - 8a^9 - 23a^8 + 63a^7 + 63a^6 - 23a^5 - 8a^4 - 4a^3)\rho^{13}\\
& \ + (59a^{10} + 118a^9 - 359a^8 - 211a^7 + 548a^6 - 211a^5 - 359a^4 + 118a^3 + 59a^2)\rho^{12}\\
& \ + (-198a^{10} - 396a^9 + 187a^8 - 334a^7 - 1141a^6 - 1141a^5 - 334a^4\\
& \ \ \ + 187a^3 - 396a^2 - 198a)\rho^{11}\\
& \ + (121a^{10} + 242a^9 + 242a^8 - 245a^7 + 392a^6 + 1886a^5 + 392a^4\\
& \ \ \ - 245a^3 + 242a^2 + 242a + 121)\rho^{10}\\
& \ + (1782a^8 + 1389a^7 - 3238a^6 + 6139a^5 + 6139a^4 - 3238a^3 + 1389a^2 + 1782a)\rho^9\\
& \ + (-2556a^8 + 1106a^7 + 5987a^6 - 2479a^5 - 5007a^4 - 2479a^3 + 5987a^2\\
& \ \ \ + 1106a - 2556)\rho^8\\
& \ + (-2570a^7 - 9742a^6 - 15917a^5 - 1995a^4 - 1995a^3 - 15917a^2 - 9742a - 2570)\rho^7\\
& \ + (11241a^6 - 22456a^5 - 25170a^4 + 36506a^3 - 25170a^2 - 22456a + 11241)\rho^6\\
& \ + (50010a^5 + 38488a^4 + 5842a^3 + 5842a^2 + 38488a + 50010)\rho^5\\
& \ + (-1466a^4 + 60818a^3 + 139800a^2 + 60818a - 1466)\rho^4\\
& \ + (-171270a^3 + 34000a^2 + 34000a - 171270)\rho^3\\
& \ + (-237084a^2 - 279582a - 237084)\rho^2 + (141750a + 141750)\rho + 540000.
\end{align*}

\begin{align*}
&B = 21600a^{10}\rho^{18} + (28350a^{10} + 28350a^9)\rho^{17}\\
& \ + (-237084a^{10} - 279582a^9 - 237084a^8)\rho^{16}\\
& \ + (-856350a^{10} + 170000a^9 + 170000a^8 - 856350a^7)\rho^{15}\\
& \ + (-36650a^{10} + 1520450a^9 + 3495000a^8 + 1520450a^7 - 36650a^6)\rho^{14}\\
& \ + (6251250a^{10} + 4811000a^9 + 730250a^8 + 730250a^7 + 4811000a^6 + 6251250a^5)\rho^{13}\\
& \ + (7025625a^{10} - 14035000a^9 - 15731250a^8 + 22816250a^7 - 15731250a^6\\
& \ \  - 14035000a^5 + 7025625a^4)\rho^{12}\\
& \ + (-8031250a^{10} - 30443750a^9 - 49740625a^8 - 6234375a^7 - 6234375a^6 \\
& \ \ \ - 49740625a^5 - 30443750a^4- 8031250a^3)\rho^{11}\\
& \  + (-39937500a^{10} + 17281250a^9 + 93546875a^8 - 38734375a^7 - 78234375a^6 \\
& \ \ \ - 38734375a^5 + 93546875a^4 + 17281250a^3 - 39937500a^2)\rho^{10}\\
 & \  + (139218750a^9 + 108515625a^8 - 252968750a^7 + 479609375a^6 + 479609375a^5\\
 & \ \ \  - 252968750a^4 + 108515625a^3 + 139218750a^2)\rho^9\\
 & \  + (47265625a^{10} + 94531250a^9 + 94531250a^8 - 95703125a^7 + 153125000a^6 \\
 & \ \ \ + 736718750a^5 + 153125000a^4 - 95703125a^3 + 94531250a^2 + 94531250a\\
 & \ \ \  + 47265625)\rho^8\\
 & \  + (-386718750a^9 - 773437500a^8 + 365234375a^7 - 652343750a^6 - 2228515625a^5\\
 & \ \ \ - 2228515625a^4 - 652343750a^3 + 365234375a^2 - 773437500a - 386718750)\rho^7\\
 & \ + (576171875a^8 + 1152343750a^7 - 3505859375a^6 - 2060546875a^5 + 5351562500a^4\\
 & \ \ \  - 2060546875a^3 - 3505859375a^2 + 1152343750a + 576171875)\rho^6\\
 & \ + (-195312500a^7 - 390625000a^6 - 1123046875a^5 + 3076171875a^4 + 3076171875a^3\\
 & \ \ \ - 1123046875a^2 - 390625000a - 195312500)\rho^5\\
 & \ + (10009765625a^6 + 20019531250a^5 + 12207031250a^4 + 10253906250a^3\\
 & \ \ \  + 12207031250a^2 + 20019531250a + 10009765625)\rho^4\\
 & \ + (-19531250000a^5 - 39062500000a^4 + 9765625000a^3 + 9765625000a^2 \\
 & \ \ \ - 39062500000a - 19531250000)\rho^3\\
 & \ + (-6103515625a^4 - 12207031250a^3 - 6103515625a^2 - 12207031250a \\
 & \ \ \ - 6103515625)\rho^2\\
 & \ + (-61035156250a^3 - 122070312500a^2 - 122070312500a - 61035156250)\rho\\
 & \ + 152587890625a^2 + 305175781250a + 152587890625.
\end{align*}

Dept. of Mathematical Sciences, Indiana University Indianapolis, \smallskip

402 N. Blackford St., Indianapolis, Indiana, 46202 United States \smallskip

\it{pmorton@iu.edu}

\end{document}